\documentclass{article}

\usepackage[a4paper,outer=3cm]{geometry}

\DeclareOldFontCommand{\bf}{\normalfont\bfseries}{\mathbf}

\let\emph\relax % there's no \RedeclareTextFontCommand
\DeclareTextFontCommand{\emph}{\bfseries\em}

\usepackage[utf8]{inputenc}
\usepackage[T1]{fontenc}
\usepackage{times}

\usepackage{xspace}

\usepackage{longtable}
\usepackage{amsmath,amsthm,amssymb,amsthm, stmaryrd}
\usepackage{enumerate}
\usepackage{multirow}
\usepackage{algorithm}
\usepackage{algpseudocode}
\algnewcommand\algorithmicinput{\textbf{Input:}}
\algnewcommand\algorithmicoutput{\textbf{Output:}}
\algnewcommand\Input{\item[\algorithmicinput]}
\algnewcommand\Output{\item[\algorithmicoutput]}

\usepackage{pgf,tikz}
\usepackage{svg}
\usetikzlibrary{arrows}
\usetikzlibrary{decorations.markings}

\usepackage{pdftricks}
\begin{psinputs}
	\usepackage{pstricks,pst-plot,pst-node,pst-func}
\end{psinputs}
\usepackage{graphics,graphicx}

\tikzstyle{vertex}=[circle, draw, inner sep=0pt, minimum size=6pt]
\newcommand{\vertex}{\node[vertex]}

\usetikzlibrary{decorations.pathreplacing}

\tikzset{->-/.style={decoration={
  markings,
  mark=at position .5 with {\arrow{>}}},postaction={decorate}}}

 \usepackage{marginnote}
\newcommand{\m}[1]{}
\usepackage[nothm]{thmbox}
\usepackage{amsthm}
\usepackage{thmtools}
\usetikzlibrary{patterns}
\usetikzlibrary{calc}

\declaretheorem[parent=section,thmbox=M]{theorem}

\declaretheorem[numberlike=theorem,thmbox=M]{corollary}
\declaretheorem[numberlike=theorem,thmbox=M]{conjecture}

\declaretheorem[numberlike=theorem,thmbox=M]{definition}

\newtheorem{claim}{Claim}[theorem]

\declaretheorem[numberlike=theorem]{lemma}

\declaretheorem[numberlike=theorem]{property}
\newcommand{\Haj}{Hajnal\xspace}
\newcommand{\Lov}{Lov\'asz\xspace}

\newcommand{\mc}{\mathcal}

\newcommand{\ra}{\rightarrow}
\newcommand{\Ra}{\Rightarrow}

\newcommand{\olra}{\overleftrightarrow}
\newcommand{\ora}[1]{\overrightarrow{#1}}

\newcommand{\ovlra}{\overleftrightarrow}

\newcommand{\bid}{\protect\ovlra}%{\overset \leftrightarrow}

\DeclareMathOperator{\dic}{\ora \chi}

\tikzstyle{vertex}=[circle,draw, top color=gray!5, 
	    bottom color=gray!30, minimum size=12pt, scale=1, inner sep=0.5pt]
\tikzstyle{arc}=[->, > = latex',  thick]
\tikzstyle{edge}=[thick, blue]

\def\centerarc[#1](#2)(#3:#4:#5) {\draw[#1] ($(#2)+({#5*cos(#3)},{#5*sin(#3)})$) arc (#3:#4:#5); }

\newenvironment{proofclaim}%[1][]%
	{\noindent {\bf Proof of Claim
	:}}
	{\hfill $\square$ \par\vspace{11pt}}

\usepackage{comment}
\newcommand{\htk}{\vec{E\mc{HT}}_{k}}
\newcommand{\hk}{\vec{\mc{H}}_{k}}
\newcommand{\htthree}{\vec{E\mc{HT}}_{3}}

\usepackage{authblk}

\title{Digraph Colouring and Arc-Connectivity}
\author[1]{Pierre Aboulker}
\author[1,2]{Guillaume Aubian}
\author[2]{Pierre Charbit}

\affil[1]{DIENS, \'Ecole normale sup\'erieure, CNRS, PSL University, Paris, France.}
\affil[2]{Université de Paris, CNRS, IRIF, F-75006, Paris, France.}

\begin{document}

\maketitle

\begin{abstract}
The dichromatic number $\dic(D)$ of a digraph $D$ is the minimum size of a partition of its vertices into acyclic induced subgraphs. We denote by $\lambda(D)$ the maximum local arc connectivity of a digraph $D$. Neumann-Lara proved that for every digraph $D$, $\dic(D) \leq \lambda(D) + 1$. In this paper, we characterize the digraphs $D$ for which $\dic(D) = \lambda(D) + 1$. This generalizes an analogue result for undirected graphs proved by Stiebitz and Toft as well as the directed version of Brooks' Theorem proved by Mohar. 
Along the way, we introduce a generalization of Haj\'os join that gives a new way to construct families of dicritical digraphs that is  of independent interest. 
\end{abstract}
\pagebreak

%--------------------------------------------

\tableofcontents

%---------------------------------------------------

\section{Introduction}
The \emph{dichromatic number} $\dic(D)$ of a digraph $D$ is the least integer $k$ such that $D$ can be partitioned into $k$ acyclic digraphs. 

The dichromatic number was first introduced by Neumann-Lara~\cite{NL82} in 1982 and was  rediscovered by Mohar~\cite{M03} 20 years later.
%who showed that this parameter allows natural generalization of circular colourings of graphs to digraphs. 
It is easy to see that for any undirected graph $G$, the \textit{symmetric digraph} $\olra G$ obtained from $G$ be replacing each edge by a digon satisfies $\chi(G) = \dic(\olra G)$. This simple fact permits to generalize results on the chromatic number of undirected graphs to digraphs via the dichromatic number. Such results have (recently) been found in various areas of graph colouring such as extremal graph theory~\cite{BBSS20, HK15, KS20}, algebraic graph theory~\cite{M10}, substructure forced by large dichromatic number~\cite{AAC21, ACN21, ACL19, hero,GSS20, HLNT19, S21}, list dichromatic number~\cite{BHL18, HM11}, dicolouring digraphs on surfaces~\cite{AHKR21, LM17, S19}, flow theory~\cite{H17, KV12}, links between dichromatic number and girth~\cite{HM12, S20}, complexity~\cite{HLM21}. 
Thus, more and more efforts are made to extend colouring results from undirected graphs to directed graphs. This paper participates in this effort. 
We explain the undirected version of our result in the next section. The reader in a hurry who would like to read our result directly can jump to Section~\ref{sec:our_result}. 

%---------------------------------------------------------

\subsection{The undirected case}

%------------------------------------------------------------

Let $G$ be a graph. The \emph{chromatic number} $\chi(G)$   of $G$ is the least integer $k$ such that $G$ can be partitioned into $k$ stable sets. A graph is \emph{$k$-critical} if it has chromatic number $k$ and all its proper subgraphs have chromatic number at most $k-1$.  

It is an easy observation that, for every graph $G$, $\chi(G) \leq  \Delta(G) + 1$, where $\Delta(G)$ is the maximum degree of $G$. 
Moreover, equality holds for odd cycles and complete graphs, and the chromatic number of a graph equals the maximum chromatic number of its connected components. This leads to a full characterization of graphs $G$ for which $\chi(G) = \Delta(G) + 1$, famously known as Brooks' Theorem. 
Let  $\mathcal B_2$ be the set of odd cycles and for $k \neq 2$, let  $\mathcal B_k = \{ K_{k+1}\}$. 

\begin{theorem}[Brooks' Theorem~\cite{B41}]
    Let $G$ be a graph. Then $\chi(G) = \Delta(G) + 1=k+1$ if and only if one of the connected component of $G$ is in $\mathcal B_k$. 
\end{theorem}

Given two vertices $u,v$ of a graph $G$, $\lambda(u,v)$ is the maximum number of edge-disjoint paths linking $u$ and $v$, and $\lambda(G)$ is the maximum local edge connectivity of $G$, that is  $\max_{u\neq v} \lambda(u,v)$.  Mader~\cite{M73} proved that for every graph $G$, $\chi(G) \leq \lambda(G) + 1$. 
Moreover, it is clear that  $\lambda(G) \leq \Delta(G)$.
Thus, for every graph $G$, 
$$ \chi(G) \leq \lambda(G) + 1 \leq \Delta(G) +1$$

Hence one can ask for graphs $G$ for which 

%generalizing Brooks' Theorem, Aboulker et al.~\cite{} in the case where $\chi(G) \leq  4$ and Stiebitz and Toft~\cite{} for $\chi(G) \geq 5$ give a full characterisation of graphs satisfying 

\begin{equation}\label{eq:nono_extremal}
  \chi(G) = \lambda(G) + 1  
\end{equation}

Exception of Brooks' Theorem of course satisfies~\eqref{eq:nono_extremal}, but it turns out that there are more. 

To describe them, we need a famous construction first introduced by Haj\'os~\cite{H61} to construct an infinite family of $k$-critical graphs. Let $G_1$ and $G_2$ be two graphs, with $uv_1 \in E(G_1)$ and  $v_2w \in E(G_2)$. 
The \emph{Haj\'os join} of $G_1$ and $G_2$ with respect to $(uv_1, v_2w)$ is the graph $G$ obtained from the disjoint union of $G_1 - uv_1$ and $G_2 - v_2w$, by identifying $v_1$ and $v_2$ to a new vertex $v$, and adding the edge $uw$. 

An  \emph{odd wheel}, is a graph obtained from an odd cycle by adding a vertex adjacent to every vertex of the odd cycle. Note that $K_4$ is an odd wheel and that odd wheels satisfy~\eqref{eq:nono_extremal}. 

One can prove that the Haj\'os join $G$ of two graphs $G_1$ and $G_2$ satisfies \eqref{eq:nono_extremal} if and only if both $G_1$ and $G_2$ satisfies it. 
Moreover,  the maximum local edge connectivity of a graph equal the maximum maximum local edge connectivity of its blocks. 

This leads to the characterization of graphs $G$ satisfying~\eqref{eq:nono_extremal}, proved by Aboulker et al.~\cite{ABHMT17} for graphs $G$ with $\chi(G) \leq 4$, and by Stiebitz and Toft~\cite{ST16} for $\chi(G) \geq 5$.  

Let $\mathcal H_k = \mathcal B_k$ when $k \leq 2$, let $\mathcal H_3$ be the smallest class containing all odd wheels and closed under taking Haj\'os join, and for $k \geq 4$, let $\mathcal H_k$ be the smallest class of graphs containing $K_k$ and closed under taking Haj\'os join. 
\begin{theorem}[\cite{ST16}]\label{thm:nono_version}
    Let $G$ be a graph. Then $\chi(G) = \lambda(G) + 1 = k+1$ if and only if a block of $G$ is in $\mathcal H_k$. 
\end{theorem}

The goal of this paper is to generalize this result to digraphs.

This result has already been generalized for hypergraphs and our result can be compared to the hypergraph case as we explain  in Section~\ref{sec:hypergraph}.

%--------------------------------------------------------

\subsection{Our result: the directed case}\label{sec:our_result}

%------------------------------------------------------

Brooks' Theorem has been generalized to digraphs by Mohar~\cite{M10}. 
There are several notions of maximum degree for a digraph. The most suitable one to generalize Brooks' theorem is the following: Given a digraph $D$, let $\Delta_{max}(D)$ be the maximum over the vertices of $G$ of the maximum of their in-degree and their out-degree. 

It is an easy observation that $\dic(D) \leq \Delta_{max}(D)+1$. 
Moreover, equality holds for directed cycles, symmetric odd cycles and symmetric complete graphs. Finally, the dichromatic number of a digraph $D$ is the maximum dichromatic number of its connected components, where the \emph{connected components} of a digraph are the connected components of its underlying graph. This leads to the directed Brooks' Theorem. 

Let  $\vec{\mathcal B}_1$ be the set of directed cycles,  let $\vec{\mathcal B}_2$ be the set of symmetric odd cycles and, for $k=0$ and $k \geq 3$, let  $\vec{\mathcal B}_k = \{\bid K_{k+1}\}$. 

\begin{theorem}[Directed Brooks' theorem~\cite{M10}, see also \cite{AA22}]\label{thm:dir_brooks}
Let $D$ be a digraph. Then $\dic(D) = \Delta_{max}(D) + 1=k+1$ if and only if a connected component of $D$ is in  $\vec{\mathcal B}_k$
\end{theorem}

Let $D$ be a digraph. Given a pair of ordered vertices $(u,v)$, we denote by $\lambda(u,v)$ the maximum number of arc-disjoint directed paths from $u$ to $v$, and by $\lambda(D)$ the \emph{maximum local arc-connectivity} of $D$, that is $max_{u\neq v}\lambda(u,v)$. 
Neumann-Lara~\cite{NL82} proved that for every digraph $D$, $\dic(D) \leq \lambda(D) + 1$. 
Since we clearly have that $\lambda(D) + 1 \leq \Delta_{max}(D) + 1$, we get that for every digraph $D$: 
$$ \dic(D) \leq \lambda(D) + 1 \leq \Delta_{max}(D)  + 1 $$

The main result of this paper is  a full characterization of digraphs $D$ for which 
\begin{equation}\label{eq:extremal}
\dic(D) = \lambda(D) + 1    
\end{equation}
 when $\dic(D) \leq 2$  and $\dic(D) \geq 4$. This generalizes the directed Brooks' Theorem and Theorem~\ref{thm:nono_version} except for digraphs with dichromatic number $3$ that we have not been able to characterize. See Section~\ref{sec:2extremal}.  
\medskip 

%There are two easy observations one can make about digraphs that satisfy~\eqref{eq:extremal}. First, a digraph satisfies this property if and only if one of its strong components satisfies it 
%\PA{c'est pas tout à fait vrai, ou en tout cas mal dit: il se peut qu'une strong component soit un directed cycle par exemple, donc satisfait \eqref{eq:extremal}, et pourtant le digraphe entier ne satisfiait pas \eqref{eq:extremal}}
%\PC{Bien vu, il faut fixer $k$. On remplace par : "First, for a given $k$, a digraph satisfies $\dic(D) = \lambda(D) + 1 = k+1$ if and only if one of its strong components satisfies the same relation (with the same $k$)}. Indeed, for both $\dic$ and $\lambda$, the value for a digraph is the maximum of the values for its strong components \PA{c'est faux pour $\lambda$}\PC{oui et en plus on n'en a pas besoin}, which implies our claim because $\dic(G)\leq \lambda(D)+1$ for every digraph. Second, and for the same exact reason, if a digraph $D$ is strongly connected but has cutvertices, then $D$ satisfies $\dic(D) = \lambda(D) + 1 = k+1$ if and only if one of its blocks satisfies the same relation. Note that these blocks also induce strongly connected digraphs. \\ 

There are two easy observations one can make about digraphs that satisfy~\eqref{eq:extremal}. First, a digraph satisfies $\dic(D) = \lambda(D) + 1 = k + 1$ if and only if one of its strong component $S$ satisfies $\dic(S) = \lambda(S) + 1 = k + 1$. Indeed, the dichromatic number of a digraph is equal to the maximum dichromatic number of its strong components, the maximum local arc-connectivity of a digraph is at least the maximum local arc-connectivity of its strong components, and for every digraph, $\dic$ is at most $\lambda + 1$. Second, and for the same exact reason, if a digraph $D$ is strong but has cutvertices, then $D$ satisfies $\dic(D) = \lambda(D) + 1 = k+1$ if and only if one of its block $B$ satisfies $\dic(B) = \lambda(B) + 1 = k+1$. Note that these blocks also induce strong digraphs.

%\begin{lemma}
%Let $D$ be a digraph with $\lambda(D) = k$. Then $\dic(D) = \lambda(D) + 1 = k+1$ if and only if a block $B$ of a strong component of $D$ satisfies $\dic(B) = \lambda(B) + 1 = k+1$. 
%\end{lemma}

%\begin{proof}
%Assume $\dic(D) = \lambda(D) + 1 = k+1$. Then $D$ has a strong component with dichromatic number $k+1$, and this strong component has a block $B$ with $\dic(B)= k+1$. Moreover, $k+1=\dic(B) \leq \lambda(B) + 1 \leq \lambda(B) + 1 = k+1$, which implies the if part of the lemma. 

%Assume now that a block $B$ of a strong component of $D$ satisfies $\dic(B) = \lambda(B) + 1 = k+1$. Then $k+1 =\dic(B) \leq \dic(D) \lambda(D) + 1 = k+1$, so $\dic(D) = k+1$. 
%\end{proof}

\medskip

Thus, our main theorem will be a structural characterization of the class of \emph{$k$-extremal digraphs} (for $k=1$ and $k\geq 3$) where a digraph $D$ is $k$-extremal if $D$ is strong, its underlying graph is $2$-connected, and $\dic(D)  = \lambda(D) + 1 = k+1$.\\

Characterizing $1$-extremal digraphs is rather easy, we will prove in Section~\ref{sec:kextr} (Theorem~\ref{thm:1extremal}) that they are exactly the class of directed cycles. Studying $k$-extremal digraphs for larger values of $k$ requires more engineering. Mimicking the construction appearing in Theorem~\ref{thm:nono_version}, we need to develop an analogue of Haj\'os joins for digraphs. 
As a matter of fact, we will need a wild generalization of Haj\'os join, giving a new way to construct $k$-dicritical  digraphs that is interesting on its own.

The most natural way to generalize Haj\'os join is to take the same definition, and replace edges by digons. Given two vertices $u$ and $v$, we set $[u,v] = \{uv, vu\}$.

\begin{definition}[Bidirected Haj\'os join]\label{def:bidirected_HJ}
  Let $D_1$ and $D_2$ be two digraphs, with $[u,v_1] \subseteq A(D_1)$ and  $[v_2,w] \subseteq A(D_2)$. 
The \emph{bidirected Haj\'os join} of $D_1$ and $D_2$ with respect to $([u,v_1], [w, v_2])$ is the digraph $D$ obtained from the disjoint union of $D_1- [u,v_1]$ and $D_2 - [w, v_2]$, by identifying $v_1$ and $v_2$ to a new vertex $v$, and adding the digon $[u,w]$.   
\end{definition}

Bidirected Haj\'os joins were first introduced and studied in~\cite{BBSS20}. %In this paper, we are going to introduce what we call Haj\'os tree joins, that generalizes bidirected Haj\'os joins. We post-pone a bit their definition. \medskip 

Still inspired by the Haj\'os joins, the so-called directed Haj\'os join, first introduced in~\cite{HK15} and studied in~\cite{BBSS20} is  defined as follows.  See Figure~\ref{fig:directed_join}.

\begin{definition}[Directed Haj\'os join]\label{def:directed_HJ}
Let $D_1$ and $D_2$ be two digraphs, with $uv_1 \in A(D_1)$ and  $v_2w \in A(D_2)$. 
The \emph{directed Haj\'os join} of $D_1$ and $D_2$ with respect to $(uv_1, v_2w)$ is the digraph $D$ obtained from the disjoint union of $D_1- uv_1$ and $D_2 - v_2w$, by identifying $v_1$ and $v_2$ to a new vertex $v$, and adding the arc $uw$.   
\end{definition}

    \begin{figure}[!hbtp]
    \begin{center}
        \begin{tikzpicture}[scale=0.5]

            \begin{scope}[yshift=2cm]
            \begin{scope}[rotate = -60]
                \vertex (x) at (0,0) {$v$};
                \vertex (v1) at (5,0) {$w$};
                \draw (2.5,0) ellipse (3cm and 2cm) {};
            \end{scope}

            \begin{scope}[rotate = -120]
                \vertex (v2) at (5,0) {$u$};
                \draw (2.5,0) ellipse (3cm and 2cm) {};
            \end{scope}
            \end{scope}

            \draw[->-] (v2) to (v1);

            \begin{scope}[shift=(v2)]
                \begin{scope}[xshift=-18cm,rotate = 60]
                    \vertex (l) at (0,0) {$u$};
                    \vertex (v) at (5,0) {$v_1$};
                    \draw[->-] (l) to (v);
                    \draw (2.5,0) ellipse (3cm and 2cm) {};
                    \node () at (-2,0) {$D_1$};
                \end{scope}
    
                \begin{scope}[xshift=-8cm, rotate = 120]
                    \vertex (l) at (0,0) {$w$};
                    \vertex (v) at (5,0) {$v_2$};
                    \draw[->-] (v) to (l);
                    \draw (2.5,0) ellipse (3cm and 2cm) {};
                    \node () at (-2,0) {$D_2$};
                \end{scope}
            \end{scope}

            \node () at (0,-4) {$D$};
           
        \end{tikzpicture}
        \end{center}
           \caption{$D$ is a directed Haj\'os join of $D_1$ and $D_2$.}
           \label{fig:directed_join}
        \end{figure}
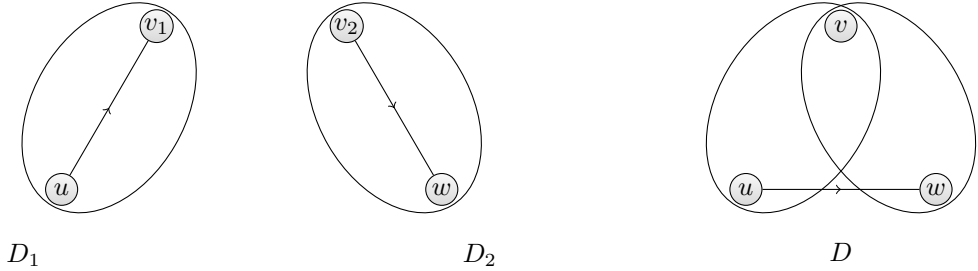

These two operations are particularly interesting because one can prove that the bidirected (directed) Haj\'os join 
$D$ of two digraphs $D_1$ and $D_2$ is $k$-dicritical if and only if both $D_1$ and $D_2$ are $k$-dicritical. They, therefore, provide a way to construct an infinite family of $k$-dicritical digraphs. They are also primordial for us because $D$ satisfies~\eqref{eq:extremal} if and only if both $D_1$ and $D_2$ do. Thus, they also provide a way to construct an infinite family of digraphs satisfying~\eqref{eq:extremal}.    

But these two joins are not enough to capture all digraphs satisfying~\eqref{eq:extremal}. In order to do so, we need to define Haj\'os tree joins, which can be seen as a generalization of  bidirected Haj\'os joins. See Figure~\ref{fig:HTcomplete}.

\begin{definition}[Haj\'os tree join and Haj\'os star join]\label{def:HTJ}
Given
\begin{itemize}
    \item a tree $T$ embedded in the plane with edges $\{u_1v_1, \dots, u_nv_n\}$, 
    \item a circular ordering $C=(x_1, \dots, x_{\ell})$ of the leaves of $T$, taken following the natural ordering given by the embedding of $T$,  and
    \item for $i=1, \dots, n$, a digraph $D_i$ such that
    \begin{itemize}
        \item $V(D_i) \cap V(T) = \{u_i,v_i\}$, 
        \item $[u_i,v_i] \subseteq A(D_i)$, and
        \item for $1 \leq i \neq j \leq n$, $V(D_i) \setminus \{u_i, v_i\} \cap V(D_j) \setminus \{u_j, v_j\} = \emptyset$,
    \end{itemize}
\end{itemize} 
we define the \emph{Haj\'os tree join} $T(D_1, \dots, D_n; C)$ to be the digraph obtained from  $D_i -[a_i,b_i]$ for $i=1, \dots, n$  by adding the directed cycle  $C = x_1 \ra x_2 \ra \dots \ra x_{\ell} \ra x_1$. 

$C$ is called the  \emph{peripheral cycle} of $D$ and  vertices $u_1, v_1, \dots, u_n,v_n$ are the \emph{junction vertices} of $D$ (note that there are $n+1$ of them). 

When $T$ is a star, we call it \emph{Haj\'os star join}. 
\end{definition}

  \begin{figure}[!hbtp]
    \begin{center}
        \begin{tikzpicture}[scale=0.4]

%l'arbre T
            \begin{scope}[xshift=0cm,scale = 0.7]              
                \begin{scope}
                    \vertex (c) at (0,0) {$c$};
                    \vertex (e) at (5,0) {$e$};
                    \draw[green] (c) -- (e);
                \end{scope}
    
                \begin{scope}[shift=(c), rotate = 120]
                    \vertex (a) at (5,0) {$a$};
                    \draw[green] (a) -- (c);
                \end{scope}
    
                \begin{scope}[shift=(c), rotate = -120]
                    \vertex (b) at (5,0) {$b$};
                    \draw[green] (b) -- (c);
                \end{scope}
    
                \begin{scope}[shift=(e), rotate = 60]
                    \vertex (e) at (0,0) {$e$};
                    \vertex (g) at (5,0) {$g$};
                    \draw[green] (e) -- (g);
                \end{scope}
    
                \begin{scope}[shift=(e), rotate = -60]
                    \vertex (e) at (0,0) {$e$};
                    \vertex (h) at (5,0) {$h$};
                    \draw[green] (e) -- (h);
                \end{scope}
    
                \begin{scope}[shift = (g)]
                    \vertex (g) at (0,0) {$g$};
                    \vertex (i) at (5,0) {$i$};
                    \draw[green] (i) -- (g);
                \end{scope}
                
                \begin{scope}[shift = (g), rotate=120]
                    \vertex (g) at (0,0) {$g$};
                    \vertex (d) at (5,0) {$d$};
                    \draw[green] (d) -- (g);
                \end{scope}

            \end{scope}

            \node () at (e |- 0, -8) {$T$};

%---------------------------------------------------

%Le D cartoonish
        \begin{scope}[xshift=22cm]
            \begin{scope}
                \vertex (c) at (0,0) {$c$};
                \vertex (e) at (5,0) {$e$};
                \draw[->-, bend right = 15, green, dashed] (c) to (e);
                \draw[->-, bend right = 15, green, dashed] (e) to (c);
                \draw (2.5,0) ellipse (3cm and 2cm) {};
            \end{scope}

            \begin{scope}[shift=(c), rotate = 120]
                \vertex (a) at (5,0) {$a$};
                \draw[->-, bend right = 15, green, dashed] (c) to (a);
                \draw[->-, bend right = 15, green, dashed] (a) to (c);
                \draw (2.5,0) ellipse (3cm and 2cm) {};
            \end{scope}

            \begin{scope}[shift=(c), rotate = -120]
                \vertex (b) at (5,0) {$b$};
                \draw[->-, bend right = 15, green, dashed] (c) to (b);
                \draw[->-, bend right = 15, green, dashed] (b) to (c);
                \draw (2.5,0) ellipse (3cm and 2cm) {};
            \end{scope}

            \begin{scope}[shift=(e), rotate = 60]
                \vertex (e) at (0,0) {$e$};
                \vertex (g) at (5,0) {$g$};
                \draw[->-, bend right = 15, green, dashed] (e) to (g);
                \draw[->-, bend right = 15, green, dashed] (g) to (e);
                \draw (2.5,0) ellipse (3cm and 2cm) {};
            \end{scope}

            \begin{scope}[shift=(e), rotate = -60]
                \vertex (e) at (0,0) {$e$};
                \vertex (h) at (5,0) {$h$};
                \draw[->-, bend right = 15, green, dashed] (e) to (h);
                \draw[->-, bend right = 15, green, dashed] (h) to (e);
                \draw (2.5,0) ellipse (3cm and 2cm) {};
            \end{scope}

            \begin{scope}[shift = (g)]
                \vertex (g) at (0,0) {$g$};
                \vertex (i) at (5,0) {$i$};
                \draw[->-, bend right = 15, green, dashed] (g) to (i);
                \draw[->-, bend right = 15, green, dashed] (i) to (g);
                \draw (2.5,0) ellipse (3cm and 2cm) {};
            \end{scope}
            
            \begin{scope}[shift = (g), rotate=120]
                \vertex (g) at (0,0) {$g$};
                \vertex (d) at (5,0) {$d$};
                \draw[->-, bend right = 15, green, dashed] (g) to (d);
                \draw[->-, bend right = 15, green, dashed] (d) to (g);
                \draw (2.5,0) ellipse (3cm and 2cm) {};
            \end{scope}

            \node () at (e |- 0, -8) {$G_1$};

            \draw[->-, bend right = 60, red] (a) to (b);
            \draw[->-, bend right = 60, red] (b) to (h);
            \draw[->-, bend right = 60, red] (h) to (i);
            \draw[->-, bend right = 60, red] (i) to (d);
            \draw[->-, bend right = 20, red] (d) to (a);
\end{scope}

%----------------------------------------------------------------

%le  hajos tree de $K_4$$
    \begin{scope}[yshift = -20cm]
        
            \begin{scope}
                \vertex (c) at (0,0) {$c$};
                \vertex (e) at (5,0) {$e$};
                \vertex (up) at (2.5, 1.5) {};
                \vertex (do) at (2.5, -1.5) {}; 
                \draw[->-, bend right = 15] (c) to (up);
                \draw[->-, bend right = 15] (c) to (do);
                \draw[->-, bend right = 15] (e) to (up);
                \draw[->-, bend right = 15] (e) to (do);
                \draw[->-, bend right = 15] (up) to (c);
                \draw[->-, bend right = 15] (do) to (c);
                \draw[->-, bend right = 15] (up) to (e);
                \draw[->-, bend right = 15] (do) to (e);
                \draw[->-, bend right = 15] (up) to (do);
                \draw[->-, bend right = 15] (do) to (up);
            \end{scope}

            \begin{scope}[shift=(c), rotate = 120]
                \vertex (a) at (5,0) {$a$};
                \vertex (up) at (2.5, 1.5) {};
                \vertex (do) at (2.5, -1.5) {}; 
                \draw[->-, bend right = 15] (c) to (up);
                \draw[->-, bend right = 15] (c) to (do);
                \draw[->-, bend right = 15] (a) to (up);
                \draw[->-, bend right = 15] (a) to (do);
                \draw[->-, bend right = 15] (up) to (c);
                \draw[->-, bend right = 15] (do) to (c);
                \draw[->-, bend right = 15] (up) to (a);
                \draw[->-, bend right = 15] (do) to (a);
                \draw[->-, bend right = 15] (up) to (do);
                \draw[->-, bend right = 15] (do) to (up);
            \end{scope}

            \begin{scope}[shift=(c), rotate = -120]
                \vertex (b) at (5,0) {$b$};
                \vertex (up) at (2.5, 1.5) {};
                \vertex (do) at (2.5, -1.5) {}; 
                \draw[->-, bend right = 15] (c) to (up);
                \draw[->-, bend right = 15] (c) to (do);
                \draw[->-, bend right = 15] (b) to (up);
                \draw[->-, bend right = 15] (b) to (do);
                \draw[->-, bend right = 15] (up) to (c);
                \draw[->-, bend right = 15] (do) to (c);
                \draw[->-, bend right = 15] (up) to (b);
                \draw[->-, bend right = 15] (do) to (b);
                \draw[->-, bend right = 15] (up) to (do);
                \draw[->-, bend right = 15] (do) to (up);
            \end{scope}

            \begin{scope}[shift=(e), rotate = 60]
                \vertex (e) at (0,0) {$e$};
                \vertex (g) at (5,0) {$g$};
                \vertex (up) at (2.5, 1.5) {};
                \vertex (do) at (2.5, -1.5) {}; 
                \draw[->-, bend right = 15] (e) to (up);
                \draw[->-, bend right = 15] (e) to (do);
                \draw[->-, bend right = 15] (g) to (up);
                \draw[->-, bend right = 15] (g) to (do);
                \draw[->-, bend right = 15] (up) to (e);
                \draw[->-, bend right = 15] (do) to (e);
                \draw[->-, bend right = 15] (up) to (g);
                \draw[->-, bend right = 15] (do) to (g);
                \draw[->-, bend right = 15] (up) to (do);
                \draw[->-, bend right = 15] (do) to (up);
            \end{scope}

            \begin{scope}[shift=(e), rotate = -60]
                \vertex (e) at (0,0) {$e$};
                \vertex (h) at (5,0) {$h$};
                \vertex (up) at (2.5, 1.5) {};
                \vertex (do) at (2.5, -1.5) {}; 
                \draw[->-, bend right = 15] (e) to (up);
                \draw[->-, bend right = 15] (e) to (do);
                \draw[->-, bend right = 15] (h) to (up);
                \draw[->-, bend right = 15] (h) to (do);
                \draw[->-, bend right = 15] (up) to (e);
                \draw[->-, bend right = 15] (do) to (e);
                \draw[->-, bend right = 15] (up) to (h);
                \draw[->-, bend right = 15] (do) to (h);
                \draw[->-, bend right = 15] (up) to (do);
                \draw[->-, bend right = 15] (do) to (up);
            \end{scope}

            \begin{scope}[shift = (g)]
                \vertex (g) at (0,0) {$g$};
                \vertex (i) at (5,0) {$i$};
                \vertex (up) at (2.5, 1.5) {};
                \vertex (do) at (2.5, -1.5) {}; 
                \draw[->-, bend right = 15] (g) to (up);
                \draw[->-, bend right = 15] (g) to (do);
                \draw[->-, bend right = 15] (i) to (up);
                \draw[->-, bend right = 15] (i) to (do);
                \draw[->-, bend right = 15] (up) to (g);
                \draw[->-, bend right = 15] (do) to (g);
                \draw[->-, bend right = 15] (up) to (i);
                \draw[->-, bend right = 15] (do) to (i);
                \draw[->-, bend right = 15] (up) to (do);
                \draw[->-, bend right = 15] (do) to (up);
            \end{scope}
            
            \begin{scope}[shift = (g), rotate=120]
                \vertex (g) at (0,0) {$g$};
                \vertex (d) at (5,0) {$d$};
                \vertex (up) at (2.5, 1.5) {};
                \vertex (do) at (2.5, -1.5) {}; 
                \draw[->-, bend right = 15] (g) to (up);
                \draw[->-, bend right = 15] (g) to (do);
                \draw[->-, bend right = 15] (d) to (up);
                \draw[->-, bend right = 15] (d) to (do);
                \draw[->-, bend right = 15] (up) to (g);
                \draw[->-, bend right = 15] (do) to (g);
                \draw[->-, bend right = 15] (up) to (d);
                \draw[->-, bend right = 15] (do) to (d);
                \draw[->-, bend right = 15] (up) to (do);
                \draw[->-, bend right = 15] (do) to (up);
            \end{scope}

            \node () at (e |- 0, -8) {$G_2$};

            \draw[->-, bend right = 60, red] (a) to (b);
            \draw[->-, bend right = 60, red] (b) to (h);
            \draw[->-, bend right = 60, red] (h) to (i);
            \draw[->-, bend right = 60, red] (i) to (d);
            \draw[->-, bend right = 20, red] (d) to (a);
    \end{scope}

%----------------------------------------------------

%Bad HT G_2
        \begin{scope}[xshift = 22cm, yshift = -20cm]
           \begin{scope}
                \vertex (c) at (0,0) {$c$};
                \vertex (e) at (5,0) {$e$};
                \vertex (up) at (2.5, 1.5) {};
                \vertex (do) at (2.5, -1.5) {}; 
                \draw[->-, bend right = 15] (c) to (up);
                \draw[->-, bend right = 15] (c) to (do);
                \draw[->-, bend right = 15,green] (e) to (up);
                \draw[->-, bend right = 15] (e) to (do);
                \draw[->-, bend right = 15,green] (up) to (c);
                \draw[->-, bend right = 15] (do) to (c);
                \draw[->-, bend right = 15] (up) to (e);
                \draw[->-, bend right = 15] (do) to (e);
                \draw[->-, bend right = 15] (up) to (do);
                \draw[->-, bend right = 15] (do) to (up);
            \end{scope}

            \begin{scope}[shift=(c), rotate = 120]
                \vertex (a) at (5,0) {$a$};
                \vertex (up) at (2.5, 1.5) {};
                \vertex (do) at (2.5, -1.5) {}; 
                \draw[->-, bend right = 15] (c) to (up);
                \draw[->-, bend right = 15] (c) to (do);
                \draw[->-, bend right = 15] (a) to (up);
                \draw[->-, bend right = 15] (a) to (do);
                \draw[->-, bend right = 15] (up) to (c);
                \draw[->-, bend right = 15] (do) to (c);
                \draw[->-, bend right = 15] (up) to (a);
                \draw[->-, bend right = 15] (do) to (a);
                \draw[->-, bend right = 15] (up) to (do);
                \draw[->-, bend right = 15] (do) to (up);
            \end{scope}

            \begin{scope}[shift=(c), rotate = -120]
                \vertex (b) at (5,0) {$b$};
                \vertex (up) at (2.5, 1.5) {};
                \vertex (do) at (2.5, -1.5) {}; 
                \draw[->-, bend right = 15,green] (c) to (up);
                \draw[->-, bend right = 15] (c) to (do);
                \draw[->-, bend right = 15] (b) to (up);
                \draw[->-, bend right = 15] (b) to (do);
                \draw[->-, bend right = 15] (up) to (c);
                \draw[->-, bend right = 15] (do) to (c);
                \draw[->-, bend right = 15,green] (up) to (b);
                \draw[->-, bend right = 15] (do) to (b);
                \draw[->-, bend right = 15] (up) to (do);
                \draw[->-, bend right = 15] (do) to (up);
            \end{scope}

            \begin{scope}[shift=(e), rotate = 60]
                \vertex (e) at (0,0) {$e$};
                \vertex (g) at (5,0) {$g$};
                \vertex (up) at (2.5, 1.5) {};
                \vertex (do) at (2.5, -1.5) {}; 
                \draw[->-, bend right = 15,blue] (e) to (up);
                \draw[->-, bend right = 15,orange] (e) to (do);
                \draw[->-, bend right = 15] (g) to (up);
                \draw[->-, bend right = 15] (g) to (do);
                \draw[->-, bend right = 15] (up) to (e);
                \draw[->-, bend right = 15] (do) to (e);
                \draw[->-, bend right = 15,blue] (up) to (g);
                \draw[->-, bend right = 15,orange] (do) to (g);
                \draw[->-, bend right = 15] (up) to (do);
                \draw[->-, bend right = 15] (do) to (up);
            \end{scope}

            \begin{scope}[shift=(e), rotate = -60]
                \vertex (e) at (0,0) {$e$};
                \vertex (h) at (5,0) {$h$};
                \vertex (up) at (2.5, 1.5) {};
                \vertex (do) at (2.5, -1.5) {}; 
                \draw[->-, bend right = 15,red] (e) to (up);
                \draw[->-, bend right = 15] (e) to (do);
                \draw[->-, bend right = 15] (h) to (up);
                \draw[->-, bend right = 15] (h) to (do);
                \draw[->-, bend right = 15] (up) to (e);
                \draw[->-, bend right = 15] (do) to (e);
                \draw[->-, bend right = 15,red] (up) to (h);
                \draw[->-, bend right = 15] (do) to (h);
                \draw[->-, bend right = 15] (up) to (do);
                \draw[->-, bend right = 15] (do) to (up);
            \end{scope}

            \begin{scope}[shift = (g)]
                \vertex (g) at (0,0) {$g$};
                \vertex (i) at (5,0) {$i$};
                \vertex (up) at (2.5, 1.5) {};
                \vertex (do) at (2.5, -1.5) {}; 
                \draw[->-, bend right = 15] (g) to (up);
                \draw[->-, bend right = 15] (g) to (do);
                \draw[->-, bend right = 15,green] (i) to (up);
                \draw[->-, bend right = 15] (i) to (do);
                \draw[->-, bend right = 15,green] (up) to (g);
                \draw[->-, bend right = 15] (do) to (g);
                \draw[->-, bend right = 15] (up) to (i);
                \draw[->-, bend right = 15] (do) to (i);
                \draw[->-, bend right = 15] (up) to (do);
                \draw[->-, bend right = 15] (do) to (up);
            \end{scope}
            
            \begin{scope}[shift = (g), rotate=120]
                \vertex (g) at (0,0) {$g$};
                \vertex (d) at (5,0) {$d$};
                \vertex (up) at (2.5, 1.5) {};
                \vertex (do) at (2.5, -1.5) {}; 
                \draw[->-, bend right = 15] (g) to (up);
                \draw[->-, bend right = 15] (g) to (do);
                \draw[->-, bend right = 15,red] (d) to (up);
                \draw[->-, bend right = 15] (d) to (do);
                \draw[->-, bend right = 15,red] (up) to (g);
                \draw[->-, bend right = 15] (do) to (g);
                \draw[->-, bend right = 15] (up) to (d);
                \draw[->-, bend right = 15] (do) to (d);
                \draw[->-, bend right = 15] (up) to (do);
                \draw[->-, bend right = 15] (do) to (up);
            \end{scope}

            \node () at (e |- 0, -8) {$G_3$};

            \draw[->-, bend right = 60] (a) to (b);
            \draw[->-, bend right = 20,green] (b) to (i);
            \draw[->-, bend left = 20] (i) to (h);
            \draw[->-, bend right = 40,red] (h) to (d);
            \draw[->-, bend right = 20] (d) to (a);
            \end{scope}
        \end{tikzpicture}

        \end{center}
           \caption{$T$ is a tree, and $G_1$ illustrates the general shape of a Haj\'os tree join built from $T$. Each block represents one of the $D_i$, where the removed digons corresponding to the  edges of $T$ are drawn with dotted green.
           The circular ordering  is $(a,b,h,i,d)$, it follows the natural ordering given by the embedding of $T$. $G_2$ is the same as $G_1$ where each of the $D_i$ is $\bid K_4$ minus a digon. We have that $\dic(G_2) = \lambda(G_2) + 1 = 4$. Finally, $G_3$ shows the importance of taking a circular ordering corresponding to an embedding of $T$ for the peripheral cycle. Indeed, for $G_3$, the ordering is $(a,b,i,h,d)$ which does not correspond to any embedding of $T$. Observe that $\dic(G_3) = 4 < \lambda(G_3) +1 = 5$. To see that $\lambda(G_3) = 4$, observe that there are four (coloured) arc-disjoint $eg$-dipaths.}
           \label{fig:HTcomplete}
        \end{figure}
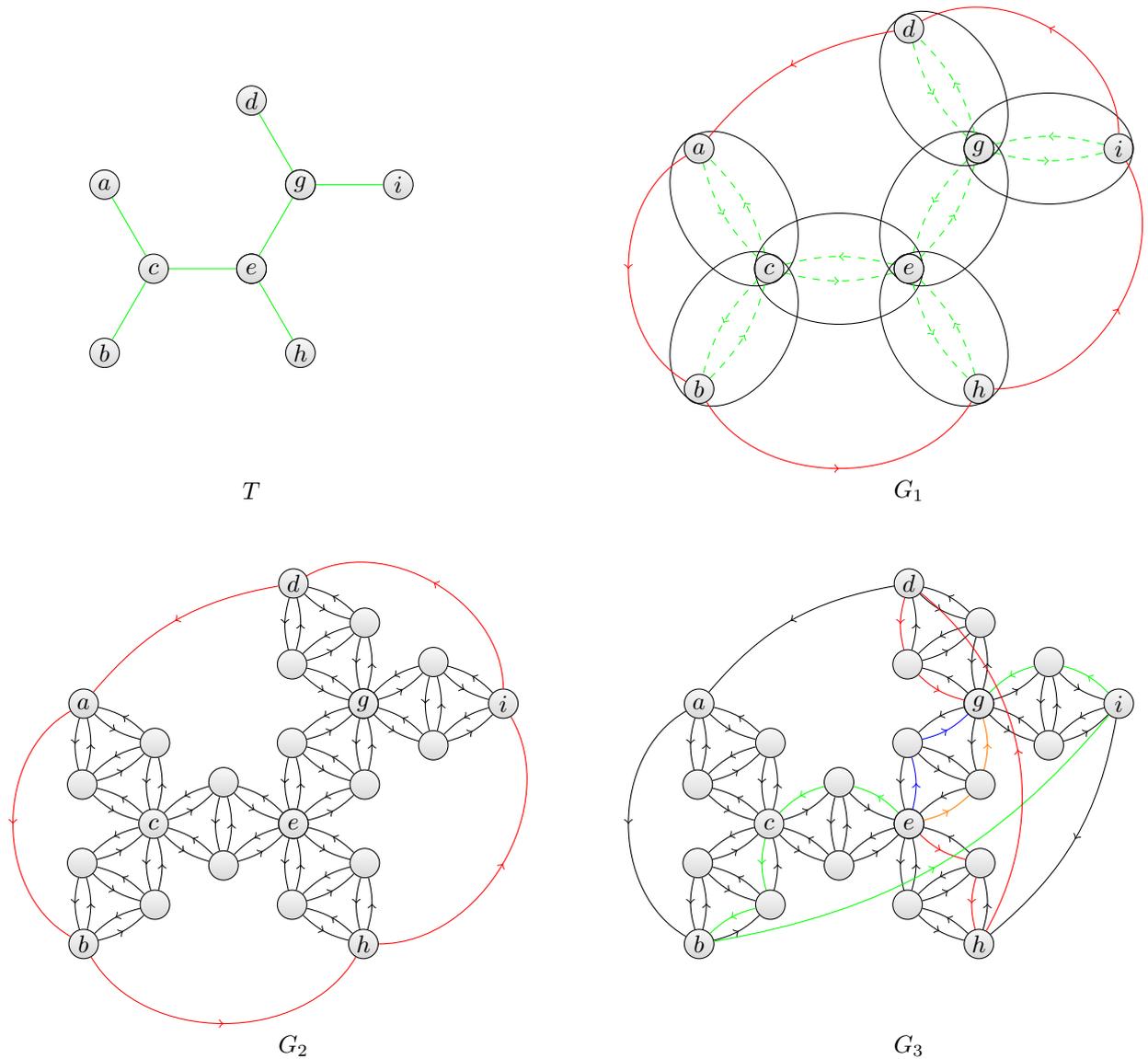

Note that, when $T$ is the path on three vertices, we recover a bidirected Haj\'os join. 
  
The basic idea of the Haj\'os tree join is the following: if each $D_i$ is $k$-dicritical, then any $(k-1)$-dicolouring of $D_i-[a_i,b_i]$ give the same colour to $a_i$ and $b_i$, which implies that in any $(k-1)$-dicolouring of $D-A(C)$, all the junction vertices receive the same colour, and thus $D$ is not $(k-1)$-dicolourable. 
Actually, we can prove that $D$ satisfies~\eqref{eq:extremal} if and only each of the $D_i$ does. Hence, Haj\'os tree join also provides a way to construct an infinite family  of digraphs satisfying \eqref{eq:extremal}.

\begin{definition}[The classes $\mathcal H_k$]   
Let \emph{$\vec{\mathcal H}_3$} be the smallest class of digraphs containing all bidirected odd wheels and closed under taking directed Haj\'os joins and Haj\'os tree joins, and for $k \geq 4$, let  \emph{$\hk$} be the smallest class of digraphs containing $\bid{K}_{k+1}$ and closed under taking directed Haj\'os joins and Haj\'os tree joins.
\end{definition}

The main result of this paper is:
\begin{theorem}\label{thm:main_HT}
    Let $k\geq 3$. Let $D$ be a digraph with $\lambda(D)=k$. Then $\dic(D) = \lambda(D) + 1=k+1$ if and only if a block of a strong component of $D$ is in $\hk$.  
\end{theorem}

 We also describe an algorithm that decides in polynomial time if a digraph $D$ belongs to $\vec {\mathcal H}_k$.\\ 

The rest of the paper is organized as follows. In Section~\ref{sec:def} we give all the needed definitions. In Section~\ref{sec:tools}, we state a few basic tools needed for our proofs. In Section~\ref{sec:kextr} we prove several important structural properties of $k$-extremal digraphs. In Section~\ref{sec:decthm} we prove a first step towards the main theorem by giving a decomposition theorem for $k$-extremal digraphs, and in Section~\ref{sec:mainthm}  we give the final proof of Theorem~\ref{thm:main_HT}. In Section~\ref{sec:algo} we give a polynomial time algorithm to recognize $k$-extremal digraphs. In section~\ref{sec:hypergraph} we discuss an analogue of our main result for hypergraphs, and in the last section, we discuss the case of $2$-extremal digraphs and propose a conjecture for their characterization.

%------------------------------------------------------

\section{Definitions}\label{sec:def}

%-------------------------------------------------

A \emph{directed graph}, or \emph{digraph}, is a pair $D = (V,A)$ of finite sets such that $A$ is a subset of $(V \times V) \setminus \{(v,v) \mid v \in V\}$. Thus our digraphs contain no loops nor parallel arcs. They may contain \emph{digons}, that is a pair of arcs in opposite directions between the same vertices. The digon $\{xy,yx\}$ is denoted by \emph{$[x,y]$}.
We  say that two vertices $x$ and $y$ are \emph{adjacent} if $xy$ or $yx$ is in $A(D)$. 
Given an undirected graph $G$, the \emph{symmetric digraph $\bid G$} is the digraph obtained by replacing each edge of $G$ with two arcs, one in each direction (a digon).

We sometimes write \emph{$x \ra y$} when $xy \in A(D)$. 
A \emph{trail} of a digraph $D$ is a sequence of vertices $x_1x_2\ldots x_p$ such that $x_ix_{i+1}\in A(D)$ for each $i<p$ and each arc is used once (but vertices can be used several times). It is \emph{closed} if $x_1 = x_p$. A  trail (resp. closed  trail) in which  vertices  are pairwise distinct is called a \emph{directed path} or \emph{dipath} for short (resp.  \emph{directed cycle} or \emph{dicycle} for short).  A $xy$-dipath is a dipath with first vertex $x$ and last vertex $y$.

Let $G$ be a connected undirected graph. A vertex $v \in V(G)$ is a \emph{cutvertex} if $G \setminus v$ is not connected. $G$ is \emph{$2$-connected} (or \emph{biconnected}) if it is connected and has no cutvertex. A \emph{block} of $G$ is a maximal biconnected subgraph of $G$. These definitions are extended to digraphs where they are applied to the underlying graphs, where the \emph{underlying graph} of a digraph $D$ is the undirected graph on the same vertex set as $D$, where edges connect vertices that are adjacent in $D$. 

A digraph is \emph{strong}, or \emph{strong}, if for every pair of vertices $x$ and $y$, there exists a $xy$-dipath and a $yx$-dipath. 

An \emph{eulerian digraph} is a digraph where for every vertex the outdegree is equal to its indegree.

The \emph{local arc-connectivity} $\lambda_D(x,y)$ of two distinct vertices $x$ and $y$ is the maximum number of pairwise arc-disjoint $xy$-dipaths. When it is clear from context, we omit the subscript $D$ and write $\lambda(x,y)$ instead of $\lambda_{D}(x, y)$.
The \emph{maximum local arc-connectivity} $\lambda(D)$ of a digraph $D$ is the maximum local arc-connectivity over all pairs of distinct vertices.

Let $X\subseteq V$. We denote by \emph{$\overline{X}$} the set $V-X$. 
We denote by  \emph{$\partial^+(X)$}  the set of arcs from $X$ to $\overline{X}$. 
A \emph{dicut} $(X, \overline{X})$ with  $X \neq \emptyset$ and $X \neq V$ is the set of arcs $\partial^+(X)$ and we say that $(X, \overline{X})$ has \emph{size} $k$ if $|\partial^+(X)|=k$. 
If $x \in X$ and $y \in \overline{X}$, we say that the dicut $(X, \overline{X})$  \emph{separates} $x$ from $y$, or that $(X, \overline{X})$ is an \emph{$xy$-dicut}.
We say that $(X, \overline{X})$  \emph{isolates} a vertex if $X$ or $\overline{X}$ is a singleton.

Let $D$ be a digraph. If $u, v \in V(D)$, we denote as $D + uv$ the digraph $(V(D), A(D) \cup \{ uv \})$. Similarly, $D - uv = (V(D), A(D) \setminus \{ uv \})$.

Let $k \geq 1$. We denote by \emph{[k]} the set $\{1, 2, \dots, k\}$. 
A \emph{$k$-dicolouring}, of a digraph $D$ is a function $\varphi: V(D) \rightarrow  [k]$ such that the digraph induced by $\varphi^-1(i)$ is acyclic for every $i \in [k]$. 
The \emph{dichromatic number $\dic(D)$} of $D$ is the minimum $k$ such that  $D$ is $k$-dicolourable. We will sometimes extend $\dic$ to subsets of vertices, using $\dic(X)$ to mean $\dic(D[X])$ where $X \subseteq V(D)$. 
It is easy to see that for any undirected graph $G$, we have $\chi(G) = \dic(\bid G)$. 

We say that a digraph $D$ is \emph{$k$-dicritical} if it has chromatic number $k$ for some integer $k$ but the removal of any arc yields a digraph of dichromatic number $k-1$. Similarly, a digraph $D$ is \emph{$k$-vertex-dicritical}  if it has dichromatic number $k$ for some integer $k$ but the removal of any vertex yields a digraph of dichromatic number $k-1$.

A digraph $D$ is \emph{$k$-extremal} if it is biconnected, strong and $\dic(D) = \lambda(D) + 1= k+1$.

\section{Tools}\label{sec:tools}

In \cite{M27}, Menger proved the following fundamental result connecting dicuts and arc-disjoint dipaths:

\begin{theorem}[Menger Theorem \cite{M27}]\label{thm:menger}
    Let $D$ be a directed digraph  and let $u, v \in V(D)$ be a pair of distinct vertices. Then $\lambda(u,v)=\partial^+(U)$ where $(U, \overline{U})$ is a minimum  dicut separating $u$ from $v$. 
\end{theorem}

\Lov~\cite{L73} proved the following result:
\begin{theorem}[\Lov~\cite{L73}]\label{thm:lovasz_lambda-reg}
Let $D$ be a digraph in which $\lambda(x,y) = \lambda(y,x)$ for any $x, y \in V(D)$. Then $D$ is Eulerian.
\end{theorem}

The next lemma is crucial as it can be applied to describe the structure of minimum dicuts in $k$-extremal digraphs.

\begin{lemma}\label{lem:mono_vs_rainbow}
    Let $k \geq 1$. Let $D$ be a digraph such that $\dic(D) \geq k +1$. Let $(X_1,X_2)$ be a dicut of $D$ of size at most $k$ such that $D[X_1]$ and $D[X_2]$ are both $k$-dicolourable, and let $\phi_1$ (resp. $\phi_2$) a $k$-dicolouring of $D[X_1]$ (resp. $D[X_2]$). Then the following holds : 
    
    \begin{itemize}
\item either there exists a unique color $i$ such that vertices in $\phi_1^{-1}(i)$ have outneighbours in $X_2$. In this case for every $1\leq j \leq k$ there is exactly one arc from a vertex in $\phi_1^{-1}(i)$ to a vertex in $\phi_2^{-1}(j)$, and at least one arc from $\phi_2^{-1}(j)$ to $\phi_1^{-1}(i)$,

\item or symmetrically there exists a unique color $j$ such that vertices in  $\phi_2^{-1}(j)$ have in-neighbours in $X_1$. In this case for every $1\leq i \leq k$ there is exactly one arc from a vertex in $ \phi_1^{-1}(i)$ to a vertex in $\phi_2^{-1}(j)$, and at least one arc from $\phi_2^{-1}(j)$ to $\phi_1^{-1}(i)$. 

\end{itemize}
In particular $(X_1,X_2)$ has size exactly $k$. 

Consequently, if $(X_2, X_1)$ also has size at most $k$, then the cut has exactly $k$ arcs in both directions, and for any two $k$-dicolouring of the two sides the following holds : on one side all arcs through the cut are incident to the same colour class, and on the other side every colour class is incident to exactly one arc in each direction through the cut.
\end{lemma}

\begin{proof}
    %Assume $\varphi_X$ and $\varphi_{\overline X}$ are  $k$-dicolouring of respectively $D[X]$, and $D[\overline X]$. 
    Let $B$ be the bipartite graph with parts $U_1 = \{1,\dots,k\}$ and $U_2 = \{1,\dots,k\}$ where an edge $ij$ is present if and only if there exists in the cut $(X_1,X_2)$ one ac in each direction between $\phi_1^-1(i)$ and $\phi_2^-1(j)$. Note that by construction, there is an injection from $E(B)$ to the set of arcs from $X$ to $\overline X$, so that $|E(B)|\leq k$.
      
     Let $H$ be the complement of $B$, that is $V(H) = V(B)$ and $E(H) = \{ij \mid i \in U, j \in V, ij \notin E(B)\}$. Observe that if $i \in U$, $j \in V$ and $ij \in E(H)$, then in $D$ there cannot be a directed cycle which uses only color $i$ in $X_1$ and color $j$ in $X_2$ since this would imply the existence of an arc in each direction between $\phi_1^-1(i)$ and $\phi_2^-1(j)$.
     
    First suppose by contradiction that $H$ has a perfect matching $M$ and hence there exists a permutation $\sigma$ such that $i\sigma(i)\in E(H)$ for every $i\in [1,k]$. Let $\varphi:V(D) \rightarrow [k]$ be defined as follows: 
      \begin{equation}
   \varphi(x)=
    \begin{cases}
      \phi_1(x), & \text{if}\ x\in X_1 \\
      \sigma(\phi_2(x)), & \text{if}\ x\in X_2
    \end{cases}
  \end{equation}
By the previous observation, $\varphi$ is a  $k$-dicolouring of $D$, a contradiction.
    
    Thus, $H$ has no perfect matching. By Hall's marriage theorem, there is $Z \subseteq U$ such that $|N_{H}(Z)| < |Z|$. Thus, there are all possible edges between $Z$ and $V - N_H(Z)$ and by counting the number of these edges we get:
    \begin{align*}
        k \geq |E(B)| & \geq |Z| (k-|N_H(Z)|) \\
                            & \geq |Z|(k-(|Z|-1))
                            %&=(k - |Z|)(|Z| - 1)
    \end{align*}
    Hence  $|Z|(k-(|Z|-1)) - k =(k - |Z|)(|Z| - 1) \leq 0$. 
    But as $1 \leq |Z| \leq k$,  we have that $|Z| = 1$ or $|Z| = k$. The first case corresponds to a vertex in $U_1$ with no neighbour in $U_2$ in $H$, so adjacent to all of $U_2$ in $B$. And the second case implies that one vertex in $U_2$ is not adjacent to any neighbour of $U_1$ in $H$, and therefore adjacent to every vertex of $U_1$ in $B$. Since $(X_1,X_2)$ has at most $k$ arcs, it is easy to see that in $D$ these two cases exactly correspond to the two outcomes in the Lemma.
    
    The last statement of the lemma easily follows by applying the argument in both directions.
 \end{proof}

%------------------------------------------------------------------------------------------------------------------------------%
%------------------------------------------------------------------------------------------------------------------------------%
From the previous Lemma we derive the following corollary.
\begin{corollary}\label{coro:small_cut}
Let $D$ be a digraph. If $D$ has a dicut $(X_1,X_2)$ of size at most $k-1$, and $\dic(D[X_i])\leq k$ for $i=1,2$, then $\dic(D)\leq k$. 
\end{corollary}
Note that by induction on the number of vertices (and since $\lambda(D')\leq \lambda(D)$ if $D'$ is a subdigraph of $D$), this corollary (applied with $k=\lambda(D)+1$) implies Neumann's Lara theorem that $\dic(D)\leq \lambda(D)+1$ for any digraph $D$.\\

When proving that some digraphs have small maximum local edge connectivity, we will often use the following Lemma:

\begin{lemma}\label{lem:lambda_diminishing}
    Let $D$ be a digraph, $u \neq v \in V(D)$ and $P$ a $uv$-dipath. Then $\lambda(D + uv - A(P)) \leq \lambda(D)$.
\end{lemma}

\begin{proof}
    Let  $H = D + uv - A(P)$. 
    Assume for contradiction that there exist $x,y \in V(D)$ that are linked by $\lambda(D)+1$ arc-disjoint $xy$-dipaths.  Since these dipaths cannot all exist in $D$, one of them contains the arc $uv$. Then, by replacing $uv$ by $P$, we obtain $\lambda(D)+1$ arc-disjoint $xy$-dipaths in $D$, a contradiction. 
\end{proof}

%------------------------------------------------------------

\section{First properties of $k$-extremal digraphs}\label{sec:kextr}

Recall that a digraph $D$ is \emph{$k$-extremal} if it is biconnected, strong and $\dic(D) = \lambda(D) + 1 = k+1$. 
The following lemma proves easy but fundamental properties of $k$-extremal digraphs that will be used constantly in the proofs. 

\begin{lemma}\label{lem:prop_k-extremal}
    Let $k \geq 1$, and let $D$ be a $k$-extremal digraph. Then $D$ is Eulerian, $(k+1)$-dicritical and $\lambda(x,y)=k$ for every pair of distinct vertices $x$ and $y$. 
    In particular, if $(X, \overline{X})$ is a minimum dicut, then so is $(\overline{X}, X)$. 
\end{lemma}

\begin{proof}
    Let $D$ be a $k$-extremal digraph, and assume $D$ is a minimal counter-example. 

    We first prove that $D$ is $(k+1)$-vertex-dicritical. 
    We proceed by contradiction. Let $X \subsetneq V(D)$ be minimal such that $\dic(D[X]) =k+1$. By minimality of $X$, $D[X]$ is biconnected and strong. 
    Moreover, since  $k+1 = \dic(D[X]) \leq \lambda(D[X]) + 1 \leq \lambda(D)+1 \leq  k+1$, we have $\lambda(D[X]) = k$. 
    So $D[X]$ is $k$-extremal and thus, by minimality of $D$, $\lambda_{D[X]}(u,v)=k$ for every pair of distinct vertices $u$, $v$ in $X$.    
    
    Let $x \in X$ such that $x$ has an outneighbour in $\overline{X}$ (it exists because $D$ is strong). Let $R^+(x)$ (resp. $R^-(x)$) be the set of vertices $y \in V(D) \setminus X$ such that there is a $xy$-dipath (resp. a $yx$-dipath) with vertices in $\overline{X} \cup \{x\}$. 
    Let $y \in R^+(x)$. Since $D$ is strong, there exists a shortest dipath $P$ from $y$ to $X$. Let $x' \in X$ be the last vertex of $P$. If $x \neq x'$, then $\lambda_D(x,x') \geq \lambda_{D[X]}(x,x')+1 = k+1$, a contradiction. So $x=x'$ and thus $y \in R^-(x)$.  Hence, $R^+(x) \subseteq R^-(x)$ and similarly $R^-(x) \subseteq R^+(x)$. So $R^+(x) = R^-(x)$ and we set $R(x)= R^+(x)$. 
    Since $D$ is biconnected, there exists a shortest path $P=x_1\dots x_{\ell}$ in the underlying graph of $D \setminus x$  with $x_1 \in X$ and $x_{\ell} \in R(x)$.
    If $\ell \geq 3$, then $x_{\ell-1} \in V(D) \setminus (X \cup R(x))$. But then if $x_{\ell - 1}x_{\ell} \in A(D)$, then $x_{\ell - 1} \in R^-(x)$ and if $x_{\ell} x_{\ell - 1} \in A(D)$, then $x_{\ell - 1} \in R^+(x)$, and thus $x_{\ell-1} \in R(x)$ in both cases, a contradiction. So $\ell =2$. 
    But then, either $x_1x_2 \in A(D)$ and thus $\lambda_D(x_1,x) = k+1$ or $x_2x_1 \in A(D)$ and thus $\lambda_D(x,x_1)=k+1$, a contradiction in both cases. 
    This proves that $D$ is $(k+1)$-vertex-dicritical. 
    \medskip

    Let $x,y \in V(D)$ and assume for contradiction that $\lambda_D(x,y) \leq k-1$. Then, by Menger Theorem~\ref{thm:menger}, $D$ has a dicut $(X, \overline{X})$ of size at most $k-1$ with $x \in X$ and $y \in \overline{X}$.  %there exists a partition $(X,Y)$ of $V(D)$ with $x \in X$, $y \in Y$ and $\partial^+(X) \leq k-1$. 
    Since $D$ is $k$-vertex-dicritical, we have that $\dic(X) \leq k$ and $\dic(\overline{X}) \leq k$ and thus, by Corollary~\ref{coro:small_cut}, $\dic(D) \leq k$, a contradiction. Hence $\lambda(x,y) = k$ for every pair of distinct vertices $x,\, y$.

    Let $xy \in A(D)$, and let $H= D - xy$. Since $\lambda_D(x,y) = k$, $\lambda_H(x,y) =k-1$ and, as above, $\dic(H) \leq k$. So $D$ is $k+1$-dicritical. 
    
    Finally, by Theorem~\ref{thm:lovasz_lambda-reg}, $D$ is Eulerian. 
\end{proof}

As a direct consequence, we get the characterization of $1$-extremal digraphs.

\begin{theorem}\label{thm:1extremal}
    A digraph is $1$-extremal if and only if it is a directed cycle.
\end{theorem}

\begin{proof}
    It is clear that all directed cycles are $1$-extremal. Conversely, if $D$ is a  $1$-extremal digraph, then it is $2$-dicritical by Lemma~\ref{lem:prop_k-extremal} and is thus a directed cycle. 
    %Then $\dic(D) = 2$, thus $D$ admits an induced directed cycle on vertex set $X \subseteq V(D)$. By Lemma~\ref{lem:prop_k-extremal}, $D$ is $2$-dicritical. But $\dic(D[C]) = 2$, thus $X = V(D)$. Hence, $D$ is a directed cycle. 
\end{proof}

Since a $k$-extremal digraph $D$ is $k+1$-dicritical, for any arc $uv$ there exists a $k$-dicolouring of $D-uv$, and if we put back the arc, then some monochromatic cycles must go through $uv$, and thus the $k$-dicolouring of $D-uv$ has some  monochromatic $vu$-dipath. In the case of a digon, we can say more.
\begin{lemma}\label{lem:extremdigon}
    Let $k \geq 1$. If $D$ is $k$-extremal and $[u,v] \subseteq A(D)$, then for every $k$-dicolouring of $D-[u,v]$, there is no monochromatic $uv$-dipath nor monochromatic $vu$-dipath.
\end{lemma}
Thus in any $k$-extremal digraph $D$ with a digon $[u,v]$, there exists an assignment of $k$ colours to the vertices such that the only monochromatic directed cycle is the digon $[u,v]$.

\begin{proof}
Consider a digon $[u,v]$ in a $k$-extremal digraph $D$. 
By Lemma~\ref{lem:prop_k-extremal}, $D$ is $k+1$-dicritical. Eulerian, and $\lambda(u,v) = k$. Let $\varphi$ be a $k$-dicoulouring of $D-[u,v]$ and let $(X_1,X_2)$ be a $uv$-dicut of size $k$. We now apply Lemma~\ref{lem:mono_vs_rainbow} to $D$ and $\phi_1=\phi|_{X_1}$ and $\phi_2=\phi|_{X_2}$. Note that since $D$ is Eulerian, we are in the situation where $X_1$ and $X_2$ play the same role, so without loss of generality we may assume that $X_1$ is the side of the cut such where only one colour class contains vertices incident to arcs in the cut (the first item in the statement of the lemma). Moreover, there exists exactly one arc in each direction between this colour class and every colour class of $\phi_2$. Since the digon $[u,v]$ is across the cut, and since $u$ and $v$ must get the same colour in $\phi$ (for $D$ is not $k$-dicolourable), the arcs $uv$ and $vu$ are the only monochromatic arcs across the cut, so that in $\phi$ there exists no monochromatic $uv$-dipath nor monochromatic $vu$-dipath except for the arcs of the digon.

\end{proof}

Along the proof, we will sometimes need to contract one side of a minimum dicut of a $k$-extremal digraph and apply induction on the obtained digraph. For this to work properly, we need to ensure that the obtained digraph is also $k$-extremal. This done in Lemma~\ref{lem:contraction_extremal}. We also need to ensure that the dicut does not isolate any vertex so that the obtained digraph is strictly smaller than the original digraph. To prove that we can always find such a dicut, we use a method derived from \cite{R14} (see also Section 3 of \cite{AA22} for its use in proving Brooks' theorem for digraphs) in the case where $k \geq 4$, see Lemma~\ref{lem:all_cuts_isolate_vertices}, and a method derived from  \cite{L66} (see also Section 5 of \cite{AA22} for its use in proving Brooks' theorem for digraphs) in the case where $k=3$, see Lemma~\ref{lem:all_cuts_isolate_vertices_three}.

%\PA{Vérifier que ce lemme est faux si on suppose k-critique à la place de k-extremal.}\PC{????}\PA{Pour les nono c'est vrai en remplaçant k-extremal par k-critique, mais je crois que c'est faux en dirigé}
\begin{lemma}\label{lem:contraction_extremal}
    Let $k \geq 1$. Let $D$ be a $k$-extremal digraph and let $(A,\overline{A})$ be a minimum dicut of 
    $D$. Then $D/A$ or $D/\overline{A}$ is $k$-extremal.
\end{lemma}

\begin{proof}
    Set $H=D/A$ and let $a$ be the vertex into which $A$ is contracted in $H$.
    Since $D$ is strong, so is $H$.
    
    Let us first prove that $\lambda(H) \leq k$. 
     By Lemma~\ref{lem:prop_k-extremal}, $D$ is Eulerian, and thus $d^+(a)=k$ and $d^-(a) =  k$.  
    Let $u,v \in H$, and let us prove that $\lambda(u,v) \leq k$. Since $d^+(a), d^-(a) \leq  k$, the result holds if $a \in \{u,v\}$. 
    Let $(B, \overline{B})$ be a minimum $uv$-dicut in $D$.

    $|\partial_D|$ is submodular, i.e. it satisfies that $\forall S, T \subseteq V(D), \partial_D(S) + \partial_D(T) \geq \partial_D(S \cup T) + \partial_D(S \cap T)$. 
    
    By Lemma~\ref{lem:prop_k-extremal}, the local arc-connectivity of any pair of vertices of $D$ is $k$, so, given $X \subset V(D)$  distinct from $\emptyset$ and $V(D)$, we have $\partial_D(X) \geq 2k$ and equality holds if and only if $|\partial_D^+(X)| = |\partial_D^+(\overline{X})| = k$. 
    %as otherwise either $\partial^+(X) < k$ or $|\partial^+(\overline{X})| < k$, and $|\partial_D(X)| = 2k$ if and only if $|\partial_D^+(X)| = |\partial_D^+(\overline{X})| = k$. 
    
    By submodularity of $|\partial_D|$, $4k = |\partial_D(A)| + |\partial_D(B)| \geq |\partial_D(A \cap B)| + |\partial_D(A \cup B)|$. 
    Moreover, since $v \in \overline{B}$, $A \cap B \neq \emptyset$ and $A\cup B \neq V(D)$, and it is clear that $A \cap B \neq V(D)$ and $A \cup B \neq \emptyset$. 
    %But as $A$ is not a subset of $B$ nor of $\overline{B}$, both $A \cap B$ and $A \cup B$ are distinct from $\emptyset$ and $V(D)$. 
    Thus $|\partial_D(A \cup B)| = 2k$, which implies that $|\partial_{H}^+(B \setminus A \cup \{a\})| \leq k$, i.e. $(B \setminus A \cup \{a\}, \overline B \setminus A)$ is a $uv$-dicut of $H$ of size at most $k$. 
    %which implies that $(A \cup B, \overline{A \cup B})$ is a minimum $uv$-dicut of $D$ separating $u$ and $v$ in $H$, a contradiction. 

    Let us now show that $H$ is biconnected. 
    For every $x \in \overline A$, $H \setminus x = (D \setminus x)/A$ is connected because $D \setminus x$ is connected. So it suffices to show that $H \setminus a$ is connected. 
    %Suppose it is not, and let $x \in V(H)$ be a cutvertex of $H$. If $x \neq a$, then $H \setminus x = (D \setminus x)/A$ is connected because $D \setminus x$ is connected. Thus $x = a$. 
    Let $X,Y$ be two connected components of $H \setminus a$. As $H$ is strong, there must be at least one arc from $a$ to $X$. But as $(A, \overline{A})$ is a minimum dicut of $D$, $a$ has outdegree $k$ in $H$, and thus there are at most $k-1$ arcs from $a$ to $Y$. Thus $(\overline{Y},Y)$ is a dicut of $D$ of size at most $k-1$, a contradiction.  So $H$ is biconnected. 
     
    As $D$ is Eulerian, $(\overline{A},A)$ is also a minimum dicut of $D$. %, and thus the same reasoning holds for $D/\overline{A}$. 
    Thus $D/\overline{A}$ is both strong and biconnected, and $\lambda(D/\overline{A}) = k$.

    We now show that either $\dic(H) \geq  k + 1$ or $\dic(D/\overline{A}) \geq  k+1$.
    Let $\varphi_{A}$ be a $k$-dicolouring of $D[A]$ and $\varphi_{\overline{A}}$  a $k$-dicolouring of $D[\overline{A}]$.
    By Lemma~\ref{lem:mono_vs_rainbow}, we may assume without loss of generality that every vertex $N(\overline{A}) \cap A$ are coloured $1$. , and that every colour is used by vertices in $N(A) \cap \overline{A}$.

    %Let $B$ be the bipartite graph with parts $X = \{1,\dots,k\}$ and $Y = \{1,\dots,k\}$ and edge-set $\{\varphi_{A}(u)\varphi_{\overline{A}}(v) \mid uv \in A(D), u \in A , v \in \overline{A})\}$. As $\dic(D) = k + 1$ and $|E(B)| \leq k$, we can use Lemma~\ref{lem:merge_along_directed_cut} to get that $\Delta(B) = k$, and thus $B$ has a vertex $v$ of degree $k$.
    %Suppose without loss of generality that $v \in X$, and let us prove that  $\dic(H) \geq k+1$ (if $v \in Y$, we can apply the same reasoning to the dicut $(\overline{A}, A)$ to get that $\dic(D/\overline{A}) \geq k+1$). 

    Let us prove that  $\dic(H) \geq k+1$. %(if $v \in Y$, we can apply the same reasoning to the dicut $(\overline{A}, A)$ to get that $\dic(D/\overline{A}) \geq k+1$).
    Suppose for contradiction that $H$ admits a proper $k$-dicolouring $\varphi_H$, chosen, up to permuting colours, so that $\varphi_H(a) = 1$.
    Consider $\varphi : V(D) \to [1,k]$ such that $\varphi(x) = \varphi_A(x)$ if $x \in A$ and $\varphi(x) = \varphi_{H}(x)$ if $x \in \overline{A}$. 
    Since $N(\overline{A}) \cap A$ are coloured $1$ with respect to $\varphi_A$, it is easy to see that $\varphi$ is a $k$-dicolouring of $D$, a contradiction.
    %Let us show that $\varphi'$ is a proper $k$-dicolouring of $D$. Let $C$ be a dicycle of $D$. If $V(C) \cap A = \emptyset$ then $C$ is not monochromatic in $\varphi'$ as $\varphi'$ coincide with $\varphi$ on $\overline{A}$. If $V(C) \cap \overline{A} = \emptyset$ then $C$ is not monochromatic in $\varphi'$ as $\varphi'$ coincide with $\varphi$ on $\overline{A}$.      Only remains the case when $C \cap A \neq \emptyset$ and $C \cap \overline{A} \neq \emptyset$. Then, there exist $u \in A, v \in \overline{A}$ with $uv \in A(C)$. Let $C'$ be the cycle corresponding to $C$ in $H$. $\varphi$ is a proper $k$-dicolouring of $D$, thus there exists $w \in V(C)$ with $w \neq a$ and $\varphi(w) \neq \varphi(a)$. But then, $\varphi'(w) = \varphi(w) \neq \varphi(a) = \varphi(u)$, and $C$ is not monochromatic.     Thus, either $H$ or $D/\overline{A}$ is $k$-extremal.
\end{proof}

Recall that given a vertex $x$, $d_{max}(x) = max(d^+(x), d^-(x))$ and $\Delta_{max}(D) = max_{x \in V(D)} d_{max}(x)$. 

\begin{lemma}\label{lem:all_cuts_isolate_vertices}
Let $k \geq 4$. If all minimum dicuts of a $k$-extremal digraph $D$ isolate a vertex, then $D = \bid{K}_{k+1}$.
\end{lemma}

\begin{proof}
    Let $D$ be a $k$-extremal digraph in which every minimum dicut isolates a vertex. If $\Delta_{max}(D) = k+1$, then $D = \bid K_{k+1}$ by Theorem~\ref{thm:dir_brooks}. Otherwise, since for every vertex of $x \in V(D)$, $d_{max}(x) \geq k$, $D$ has a vertex with outdegree strictly greater than $k$ or a vertex with indegree strictly greater than $k$. 

    If there are two distinct vertices $u$ and $v$ with $d_{max}(u) \geq k+1$ and $d_{max}(v) \geq k+1$, then a minimum $uv$-dicut does not isolate $u$ nor $v$ (because $\lambda(u,v) = k$ and $D$ is Eulerian by Lemma~\ref{lem:prop_k-extremal} and thus, if $(U,V)$ is a minimum $uv$-dicut, then it has size $k$, and since $D$ is Eulerian, $(V,U)$ also has size $k$).   
    So there is a  vertex $u$ with $d_{max}(u) \geq k+1$, and for every $x \in V(D) \setminus u$, $d^+(x) = d^-(x) = k$.    
    
    Let $M \subseteq V(D)$ be a maximal set of vertices such that $u \in M$ and $D[M]$ is acyclic. Then every vertex of $V(D) \setminus M$ has outdegree and indegree at most $k-1$ in $V(D) \setminus M$, i.e. $\Delta_{max}(D[V(D) \setminus M]) \leq k - 1$. 
    As $\dic(D[M]) = 1$ and $\dic(D) = k+1$, we have that $\dic(D[V(D) \setminus M]) \geq k \geq \Delta_{max}(D[V(D) \setminus M]) +1$. 
    By Theorem~\ref{thm:dir_brooks} applied on $D[V(D) \setminus M]$, there exists $K \subseteq V(D) \setminus M$ such that $D[K] = \bid{K_{k}}$. As every vertex of $K$ has in- and outdegree exactly $k$ in $D$ and $k-1$ in $D[K]$, $\partial^+(K) = \partial^-(K)= k$. 
    Thus $(K, \overline K)$ is a minimum dicut and thus $V(D) \setminus K = \{u\}$ by hypothesis, a contradiction to the fact that $d_{max}(u) \geq k+1$. 
\end{proof}

\begin{lemma}\label{lem:all_cuts_isolate_vertices_three}
    If all minimum dicuts of a $3$-extremal digraph $D$ isolate a vertex, then $D = \bid W_{2\ell+1}$ for some $\ell \geq 1$ or $D$ is a directed Haj\'os join or a bidirected Haj\'os join of two digraphs.
\end{lemma}

\begin{proof}
    Let $D$ be a $3$-extremal digraph in which every minimum dicut isolates a vertex, and assume for contradiction that $D$ is not a symmetric odd wheel nor a directed Haj\'os join.
    Similarly to the proof of Lemma~\ref{lem:all_cuts_isolate_vertices}, we can prove that there is a unique vertex $u$ with $d^+(u) = d^-(u) \geq 4$ and for every $v \in V(D) \setminus u$, $d^+(v) =d^-(v) = 3$. %, and since $D$ is Eulerian we have $d^+(u) = d^-(v) \geq k+1$.
    
    %As $\sum_{v \in V(D)} d^+(v) = \sum_{v \in V(D)} d^-(v) $, there is actually both a vertex $u$ with $d^+(u) \geq k+1$ and a vertex $v$ with $d^-(v) \geq k+1$.      We have $u = v$, as otherwise no minimum directed $uv$-cut isolates $u$ nor $v$. \PA{On utilise que dans une min cut il y a k arcs dans un sens et ka rcs dans l'autre}  %(because it has outdegree at least $k+1$) nor $v$ (because it has indegree at least $k+1$), a contradiction. 
    %Thus, there is exactly one vertex $u$ with indegree and outdegree strictly greater than $k$, and every other vertex of $D$ has indegree and outdegree $k$.

    Let $P = (X,Y)$ be a partition of $V(D)$. We say that $P$ is  a \emph{special partition} if $u \in X$, $D[X]$ is acyclic and $X$ has maximum size among all sets $X'$ such that $u \in X'$ and $D[X']$ is acyclic. Note that if $P=(X,Y)$ is a special partition, then every vertex of $Y$ has at least one inneighbour and one outneighbour in $X$, and thus has in- and outdegree at most $2$ in $D[Y]$. 
    
    An \emph{obstruction of $P$} is a connected component of $D[Y]$ isomorphic to a symmetric odd cycle. Since obstructions are $2$-regular, a connected component of $D[Y]$ contains an obstruction if and only if it isomorphic to an obstruction. Note also that every special partition has at least one obstruction, for otherwise, by Theorem~\ref{thm:dir_brooks}, $\dic(D[Y]) \leq 2$ and thus $\dic(D) \leq 3$, a contradiction.
    
    $P$ is said to be a \emph{super-special partition} if it is special and it minimizes the number of obstructions among all special partitions. 
    
    We call the following operation \emph{switching $x$ and $y$}.
    
    \begin{claim}\label{clm:switch}
        Let $P = (X,Y)$ be a super-special partition. Let $y$ be a vertex in an obstruction of $P$, and $x \in X \setminus \{u\}$ be a neighbour of $y$. Then $P' = (X \cup \{y\} \setminus \{x\}, Y \cup \{x\} \setminus \{y\})$ is a super-special partition, and $x$ is in an obstruction of $P'$.
    \end{claim}
    
    \begin{proofclaim}
         Suppose without loss of generality that $xy \in A(D)$. Let $Z \subseteq Y$ be the vertex set of the obstruction containing $y$. As $d^-(y) = 3$ and $y$ has $2$ inneighbours in $Z$, $y$ has no inneighbour in $X \setminus \{x\}$. Thus $D[X \cup \{y\} \setminus \{x\}]$ is acyclic. As $x \neq u$,  $P'$ is special. 
         Since removing any vertex of a symmetric odd cycle yields a digraph that is not a symmetric odd cycle, $D[Z \setminus \{y\}]$ is not a symmetric odd cycle. Thus $x$ is in an obstruction of $P'$ and $P'$ is super-special. 
         %Thus $P'$ is super-special.  By definition, $P$ minimizes the number of obstructions of a special partition, thus $x$ must be in an obstruction of $P'$ for otherwise $P'$ would have strictly less obstructions than $P$. 
    \end{proofclaim}

    The switching operation is particularly useful thanks to the following claim:

    \begin{claim}\label{clm:two_neighbours_obstruction}
     Let $P = (X,Y)$ be a super-special partition, and $Z$ the vertex set of an obstruction of $P$. Vertices in $X \setminus u$ have at most one neighbour in $Z$. 
    \end{claim}

    \begin{proofclaim} 
        %\PA{J'ai réécrit la preuve, je comprenais pas, Guillaume tu peux relire?}
        Let $Z=\{v_1, \dots, v_s\}$ and $v_i$ and $v_{i+1}$ are linked by a digon for $i=1, \dots, s$ (subscript are taken modulo $s$). Suppose for contradiction that there is
        $x \in X \setminus \{u\}$ such that $x$ is a neighbour of $v_i$ and $v_j$ for some $i \neq j$. 
        By claim~\ref{clm:switch}, we can switch $x$ and $v_i$ to obtain the super special partition $P' = (X \cup \{v_i\} \setminus \{x\} , Y \cup \{x\} \setminus \{v_i\} )$. Since $x$ is a neighbour of $v_j$, the obstruction of $P'$  containing $x$ is $D[Z \cup \{x\} \setminus \{a\}]$, i.e. $D[Z \cup \{x\} \setminus \{a\}]$ is a symmetric odd cycle. Hence, $x$ is linked by a digon to $v_{i-1}$ and $v_{i+1}$. Now, by switching $x$ and $v_{i+1}$ in $P$, we get that $x$ is also linked by a digon to $v_{i+2}$ and thus is linked by a digon to every vertex of $Z$. 
        This implies that $s=3$, and thus the dicut $(V(Z),V(D) \setminus V(Z))$ has size $3$, so it is a minimum dicut that does not isolate a vertex, a contradiction. 
    \end{proofclaim}

    Let $P_1=(X_1, Y_1)$ be a super-special partition of $D$ and let $Z_1$ an obstruction of $P_1$. If no vertex of $Z_1$ has a neighbour in $X_1 \setminus \{u\}$, then $D[Z \cup \{u\}]$ is a symmetric odd wheel and we are done. So there exist $x_1 \in X_1 \setminus \{u\}$ and $y_1 \in Z_1$ such that $x_1$ and $y_1$ are adjacent. Set $Q_1 = Z_1 \setminus \{y_1\}$. 
    
    By claim~\ref{clm:switch}, $P_2=(X_2, Y_2)$ with $X_2= X_1 \cup \{y_1\} \setminus \{x_1\}$ and $Y_2 = Y_1 \cup \{x_1\} \setminus \{y_1\}$ is a super-special partition and $x_1$ is in an obstruction $Z_2$ of $P_2$. Let $Q_2 = Z_2 \setminus \{x_1\}$, so $Q_2$ is a symmetric path and is a connected component of $D[Y_1]$. 
    Observe that no vertex of $Q_2$ is adjacent with $y_1$. 
    If $y_1$ is the only vertex in $X_2 \setminus \{u\}$ with a neighbour in $V(Z_2)$, then either $x_1$ and $y_1$ are linked by a digon and $D$ is a bidirected Haj\'os join (by Lemma~\ref{lem:HB_sufficient}, because deleting $u$ and $[x_1, y_1]$ separates $Z_2$ from the rest of the digraph), or $D$ is a directed Haj\'os join (by Lemma~\ref{lem:HJsufficient}, because deleting $u$ and the arc linking $x_1$ and $y_1$ separates $Z_2$ from the rest of the digraph). A contradiction in both cases. 
    Hence, there is $x_2 \in X_2 \setminus \{u, y_1\}$ such that $x_2$ has a neighbour $y_2 \in V(Z_2)$. %$y_2=x_1$ is possible
    
    Let $P_3=(X_3, Y_3)$ where $X_3 = X_2 \cup \{y_2\} \setminus \{x_2\}$ and $Y_3 = Y_2 \cup \{x_2\} \setminus \{y_2\}$. By claim~\ref{clm:switch}, $P_3$ is a super-special partition and $x_2$ is in an obstruction $Z_3$ of $P_3$.  
    Note that, claim~\ref{clm:two_neighbours_obstruction},  $x_2$ has at most one neighbour in $Z_1$ and in $Z_2$, but it has two neighbours in $Z_3$, this implies that  $Z_3$ is disjoint from $Z_1$ and $Z_2$. As previously, if $y_2$ is the only vertex in $X_3 \setminus u$ with a neighbour in $Z_3$, then $D$ is a directed Haj\'os join or a bidirected Haj\'os join, a contradiction. 
    So there exists $x_3 \in X_3 \setminus \{u, y_1, y_2\}$ such that $x_3$ has a neighbour $y_3 \in V(Z_3)$. 

    This process can be continued and never stop, a contradiction.

\end{proof}

%-----------------------------------------------

\section{Directed Haj\'os joins and Haj\'os bijoins - A First Decomposition Theorem}\label{sec:decthm}
Our main theorem presented in the introduction (Theorem~\ref{thm:main_HT}) is a {\em structure theorem} for the class of $k$-extremal digraphs, in the sense that it is an "if and only if". The goal of this section is to prove an intermediate result that is a {\em decomposition theorem} for this class (an "only if" theorem). It involves the notion of Haj\'os bijoin described just after the statement.

\begin{theorem}\label{thm:decHJHBJ}
Let $k\geq 3$. If $D$ is $k$-extremal, then:
    \begin{itemize}
        \item either $D=\bid{K}_{k+1}$
        \item or $D$ is a symmetric odd wheel (only in the case $k=3$),
        \item or $D$ is a directed Haj\'os join of two $k$-extremal digraphs,
        \item or $D$ is a Haj\'os bijoin of two $k$ extremal-digraphs.
    \end{itemize}
\end{theorem}

\begin{definition}[Haj\'os bijoin and degenerated Haj\'os bijoin]\label{def:bijoin}
    Let $D_1$ and $D_2$ be two digraphs. Let $ta_1, a_1w \in A(D_1)$ ($t=w$ is possible) and $t$ and $w$ are in the same connected component of $D_1 \setminus a_1$. Let $va_2,a_2u \in A(D_2)$ ($u=v$ is possible) and $u$ and $v$ are in the same connected component of $D_2 \setminus a_2$. 
    The \emph{Haj\'os bijoin} of $D_1$ and $D_2$ with respect to $\big((t,a_1,w), (v,a_2, u)\big)$ is the digraph $D$ obtained from the disjoint union of $D_1 - \{ta_1, a_1w\}$ and $D_2 - \{va_2, a_2u\}$ by identifying $a_1$ and $a_2$ into a new vertex $a$, and adding the arcs $tu$ and $vw$. \\
    If $t=w$ and $u \neq v$ (or $u=v$ and $t \neq w$), we say it is a \emph{degenerated Haj\'os bijoin}.\\
    If $t=w$ and $u=v$, then it is the bidirected Haj\'os join of $D_1$ and $D_2$ with respect to $([t,a_1], [u,a_2])$. See Figure~\ref{fig:bijoin}.
\end{definition}

%\PA{J'ajoute l'intuition qui avait disparu, mais elle est peut-être pas terrible. On pourrait faire une belle digure dans laquelle on glue deux Hajos tree ensemble}
%\PA{In some sense, Haj\'os bijoin gives a way to glue two digraphs that are Haj\'os tree join into a new digraph that is also a Haj\'os tree join. Let us say an informal word on this, in order to give the intuition. 
%Let $D_1=T_1(A_1,\dots, A_k; C_1)$ and $D_2=T_2(B_1,\dots, B_{\ell}; C_2)$ be two directed Haj\'os tree join with $C_1 = (x_1, \dots, x_a)$ and $C_2 = (y_1, \dots, y_{b})$. Then the digraph obtained from the  disjoint copies of $D_1-\{x_kx_1, x_1x_2\}$ and $D_2-\{y_{\ell}y_1, y_1y_2\}$ by identifying $x_1$ and $y_1$ and adding the arcs $y_{\ell}x_2$ and $x_ky_2$ is a Haj\'os tree join, and is also the bijoin of $D_1$ and $D_2$ with respect to $\big((x_k,x_1,x_2), (y_{\ell},y_1, y_2)\big)$. It is also a Haj\'os tree join where the tree is obtained from $T_1$ and $T_2$ by identifying $x_1$ and $y_2$}

%\PA{Observe  that, if  $D$ is "almost" the bijoin of two digraphs, in the sense that one of the connectivity property is not satisfied, for example $t \neq w$ and $t$ and $w$ are in two distinct components of $D_1 \setminus a_1$, then $D \setminus a - tu$ is not connected and $t$ and $u$ are in two distinct connected components of $D \setminus a - tu$, which implies that $D$ is a directed Haj\'os join of two digraphs.}

    \begin{figure}[!hbtp]
    \begin{center}
        \begin{tikzpicture}[scale=1]
            \begin{scope}[xshift = -1cm]
            \begin{scope}
                \vertex (r) at (1.2,0) {};
                \vertex (u) at (0,1) {$t$};
                \vertex (d) at (0,-1) {$w$};
                \draw (0,0) circle (1.5);
            \end{scope}
        
            \begin{scope}[xshift=2.4cm,xscale=-1]
                \vertex (r2) at (1.2,0) {$a$};
                \vertex (u2) at (0,1) {$u$};
                \vertex (d2) at (0,-1) {$v$};
                \draw (0,0) circle (1.5);
            \end{scope}

            \node () at (1.2,-2) {$D$};
            \end{scope}
            
            \draw[->-, bend left=15] (u) to (u2);
            \draw[->-, bend left=15] (d2) to (d);

            \begin{scope}[xshift=-9cm]
                \vertex (r) at (1.2,0) {$a_1$};
                \vertex (u) at (0,1) {$t$};
                \vertex (d) at (0,-1) {$w$};
                \draw[->-, bend left=15] (u) to (r);
                \draw[->-, bend left=15] (r) to (d);
                \draw (0,0) circle (1.5);
                \node () at (0,-2) {$D_1$};
            \end{scope}
        
            \begin{scope}[xshift=-5.5cm,xscale=-1]
                \vertex (r2) at (1.2,0) {$a_2$};
                \vertex (u2) at (0,1) {$u$};
                \vertex (d2) at (0,-1) {$v$};
                \draw[->-, bend right=15] (d2) to (r2);
                \draw[->-, bend right=15] (r2) to (u2);
                \draw (0,0) circle (1.5);
                \node () at (0,-2) {$D_2$};
            \end{scope}
        
        \end{tikzpicture}
        \end{center}
           \caption{$D$ is the Haj\'os bijoin of $D_1$ and $D_2$ with respect to $(t,a_1, w), (u,a_2, v)$.}
           \label{fig:bijoin}
        \end{figure}
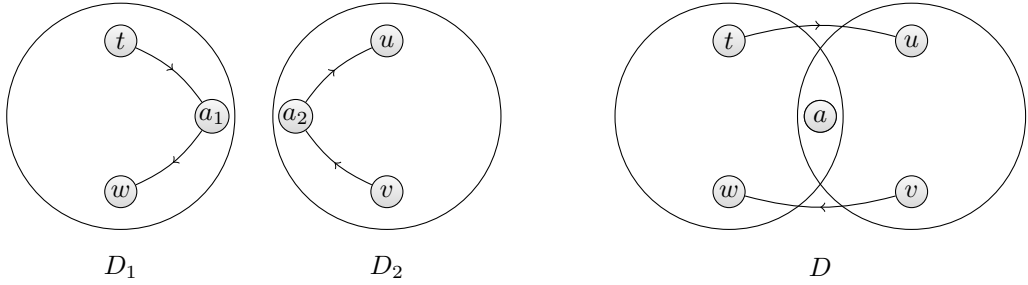

%Note that if $t=w$ and $u=v$, then we get what we called earlier a bidirected Haj\'os join.

\subsection{Properties of directed Haj\'os joins}

For the definition of directed Haj\'os join, see Definition~\ref{def:directed_HJ}. We first prove an essential result about $k$-extremal digraphs and directed Haj\'os joins

\begin{lemma}\label{lem:extremal_directed_Haj}
    Let $k \geq 1$. 
    Let $D$ be the directed Haj\'os join of two digraphs $D_1$ and $D_2$. Then $D$ is $k$-extremal if and only if  both $D_1$ and $D_2$ are.
\end{lemma}

\begin{proof}
    Suppose $D$ is a directed Haj\'os join of two digraphs $D_1$ and $D_2$ with respect to $(uv_1, v_2w)$, i.e. there is $uv_1 \in A(D_1)$ and $v_2u \in A(D_2)$ such that $D$ is obtained from disjoint copies of $D_1-uv_1$ and $D_2-v_2w$ by identifying the vertices $v_1$ and $v_2$ to a new vertex $v$ and adding the edge $uw$. 

    \begin{claim}[Theorem 2 in~\cite{BBSS20}]\label{clm:kawa_DHJ}
        $D$ is $k+1$-dicritical if and only if both $D_1$ and $D_2$ are. 
    \end{claim}
    
    Let us first suppose that $D_1$ and $D_2$ are $k$-extremal.  
    By Claim~\ref{clm:kawa_DHJ}, $D$ is $k+1$-dicritical, so it is also biconnected and strong. Since the maximum local arc-connectivity of a digraph equal the maximum maximum local arc-connectivity of its blocks, we have that  $\lambda(D-\{uw\} + \{uv_1, v_2w\}) =\max(\lambda(D_1), \lambda(D_2)) = k$, and by Lemma~\ref{lem:lambda_diminishing}, $\lambda(D) \leq \lambda(D-\{uw\} + \{uv_1, v_2w\} =  k$. Thus $k +1 = \dic(D) \leq \lambda(D) + 1 = k+1 $, so $\lambda(D) = k$ and $D$ is $k$-extremal. 
    
    Suppose now that $D$ is $k$-extremal. 
    By claim~\ref{clm:kawa_DHJ}, both $D_1$ and $D_2$ are $(k+1)$-dicritical and thus are also strong and biconnected. 
    Since $D$ is strong, it has a $wv$-dipath, and this dipath uses only arcs in the copy of $D_2 -wv_2$. Let $P$ be such a dipath to which we add the arc $uw$ at the beginning. 
    Then $\lambda(D + uv - A(P)) \leq \lambda(D) = k$, and since $D_1$ is a subdigraph of $D + uv - A(P)$, we have that $\lambda(D_1) \leq k$. Thus  $k +1 = \dic(D_1) \leq \lambda(D_1) + 1 = k+1 $, so $\lambda(D) = k$ and $D_1$ is $k$-extremal. Similarly, $D_2$ is $k$-extremal. 
\end{proof}

 If $D$ is a directed Haj\'os join of two digraphs, then there exists an arc $uw$, such that $D-uw$ has a cutvertex. The following lemma asserts that if $D$ is $k$-extremal, then the converse holds. This is sometimes useful to prove that $D$ is a directed Haj\'os join. 
 \begin{lemma}\label{lem:HJsufficient}
     Let $k \geq 1$. Let $D$ be a $k$-extremal digraph with an arc $uw\in A(D)$, such that $D-uw$ has a cutvertex $v$. Then $D$ is a directed Haj\'os join of two digraphs $D_1$ and $D_2$ with respect to $(uv, vw)$.
 \end{lemma}
 \begin{proof}
     Let $v$ be the cutvertex of $D-uw$. Since $D$ is biconnected, $D-uw$ has exactly two blocks $D_1$ and $D_2$  containing respectively $u$ and $w$. 
     Thus it is enough to prove that $uv\not\in A(D)$ and $vw\not\in A(D)$. We only prove it for $uv$, the argument  for $vw$ is identical. Assume by contradiction $uv\in A(D)$. 
     By Lemma~\ref{lem:prop_k-extremal}, $D$ is $k+1$-dicritical, so $D_i$ has a $k$-dicolouring $\varphi_i$ for $i=1,2$. Since $uv$ is an arc, there is no monochromatic $vu$-dipath with respect to $\varphi_1$. 
     Up to permuting colours, we may assume that $\varphi_2(v) = \varphi_1(v)$. Let $\varphi:V(D) \rightarrow [k]$ defined as follow: $\varphi(x) = \varphi_1(x)$ if $x \in V(D_1)$, and $\varphi(x) = \varphi_2(x)$ if $x \in V(D_2)$. 
     We claim that $\varphi$ is a $k$-dicolouring of $D$. Indeed, by construction of $\varphi$ there is no monochromatic dicycle included in $D_1$ or $D_2$, and a dicycle intersecting both $D_1$ and $D_2$ contains a $vu$ dipath included in $D_1$, and thus cannot be monochromatic. Thus $\varphi$ is a $k$-dicolouring of $D$, a contradiction. 
     %If we now consider $D_2$ to be the union of all the other blocks of $D-a$, then again it admits a $k$-dicolouring. Up to permuting the colours, we can assume the $2$ $k$-dicolourings agree on the colour of $v$, and thus get a $k$-dicolouring of $D$ that is proper: indeed any dicycle in $D$ is either contained in $D_1$ or $D_2$ or contains $uw$ and goes through $v$, so must contain a $vu$-dipath. In all  cases it cannot be monchromatic. 
 \end{proof}

\subsection{Properties of Haj\'os bijoin}

We start with the analogue of Lemma~\ref{lem:extremal_directed_Haj} for bijoins, but note that we have only one direction here.

\begin{lemma}\label{lem:extremal_bijoin_Haj}
    Let $k \geq 3$. Let $D$ be a Haj\'os bijoin of two digraphs $D_1$ and $D_2$. If $D$ is $k$-extremal, then both $D_1$ and $D_2$ are $k$-extremal.
\end{lemma}

\begin{proof} 
    Let $D$ be a $k$-extremal digraph, and $D$ is a Haj\'os bijoin of two digraphs $D_1$ and $D_2$ with respect to $\big((t,a_1,w), (v,a_2, u)\big)$, i.e. there exists $tu,vw \in A(D)$ and $a \in V(D)$ such that $D\setminus \{a\} - \{tu,vw\}$ has two connected components with vertex sets $V'_1$ and $V'_2$ such that $D_1=D[V'_1 \cup a] + \{ta,aw\}$, and $D_2= D[V'_2 \cup a] + \{va,au\}$ and $t$ and $w$ are in the same connected component of $D_1\setminus a$ and  $u$ and $v$ are in the same connected component of $D_2\setminus a$. 
    \smallskip

    Let us first prove that $D_1$ is biconnected. Assume for contradiction that $D_1$ has a cutvertex $x$. Observe that $\{t,w,a_1\} \setminus x$ are in the same connected component of $D_1 \setminus x$. Indeed, if $x=a$ it is by hypothesis, and otherwise it is because $ta, aw \in A(D_1)$. Hence, $D \setminus x$ has a connected component disjoint from $\{t,a,w\}$, and thus $x$ is a cutvertex of $D$, a contradiction.  
    %Since $ta,aw \in A(D_1)$, either $t$, $a$ and $w$ are in the same block of $D_1 \setminus x$, or $x=a$, and $t$ and $w$ are not in the same connected component of $D_1 \setminus x$. In the first case, $x$ is a cutvertex of $D$, a contradiction. In the second case, $t$ and $w$ are in two distinct connected component of $D_1 \setminus a$, contradicting the definition of a Haj\'os bijoin. 
    \smallskip 

    As $D$ is Eulerian, so is $D_1$ by construction. And since an Eulerian connected digraph is strong, $D_1$ is strong. 
    \smallskip 
    
    Let $x,y \in V(D_1)$ and let us prove that $\lambda_{D_1}(x,y) \leq k$. 
    %Assume without loss of generality that $x \neq a$.  
    Let $(X,\overline{X})$ be a minimum $xy$-dicut in $D$, i.e. $x \in X$ and $y \in \overline{X}$. Since $D$ is $k$-extremal, $(\overline{X}, X)$ is a minimum $yx$-dicut, and thus, up to permuting $y$ and $x$, we may assume without loss of generality that $a \in \overline{X}$. 
    Let $X_{D_1} =X \cap V(D_1)$ and  consider the $xa$-dicut of $D$ $(X_{D_1}, \overline{X}_{D_1})$.  If $t \in X$, then $\partial_{D_1}^+(X) = \partial_D^+(X) -tu + ta$, and otherwise $\partial_{D_1}^+(X) = \partial_D^+(X)$. Hence $|\partial_{D_1}^+(X)| = |\partial_D^+(X)| = k$, so $\lambda_{D_1}(u,v) \leq k$ and thus $\lambda(D_1) \leq k$. 
    \smallskip 
    
    Let us now prove that $\dic(D_1) \geq k + 1$. Suppose $D_1$ admits a $k$-dicolouring $\varphi_1$. Let $\varphi_2$ be a $k$-dicolouring of $D_2 - \{va,au\}$  
    such that $\varphi_1(a) = \varphi_2(a)$ and, if $\varphi_1(a) \neq \varphi_1(t)$, such that $\varphi_2(u) \neq \varphi_1(t)$ (this can always be done because $k \geq 3$). 
    Consider $\varphi : V(D) \to [1,k]$ such that $\varphi(x) = \varphi_1(x)$ if $x \in V(D_1) \cup a$ and $\varphi(x) = \varphi_2(x)$ if $x \in V(D_2)$. 

    Since $\dic(D) = k+1$, $D$ has a monochromatic directed cycle $C$ with respect to $\varphi$. 
    By construction of $\varphi$, $C$ goes through $tu$, or $vw$ or both. 
    If $C$ goes through   $tu$ but not  $vw$, then $C$ contains an $at$-dipath, which is not monochromatic since $ta \in A(D_1)$ and $\varphi_1$ is a $k$-dicolouring of $D_1$, a contradiction.
    Similarly, we get a contradiction if $C$ goes through   $vw$ but not  $tu$. 
    We may thus assume that $C$ uses both $tu$ and $vw$. In particular, $C$ contains an $uv$-dipath $P_{uv}$ included in $D_2$, and a $wt$-dipath $P_{wt}$ included in $D_1$. 
    If $\varphi_1(a) = \varphi_1(t)$, then $P_{wt}$  plus the arcs $ta$ and $aw$ form a monochromatic dicyle of $D_1$, a contradiction. 
    Thus $\varphi_1(a) \neq \varphi_1(t)$ and thus $\varphi(t) \neq \varphi(u)$ by construction of $\varphi$, so $C$ is not monochromatic, a contradiction. This finishes the proof that  $\dic(D_1) \geq k+1$. 
    \smallskip 
    
    Now, since $k+1 \leq \dic(D_1) \leq \lambda(D_1) + 1 \leq k+1$, we have $\dic(D_1) = \lambda(D_1) + 1 = k+1$. Altogether we get that $D_1$ is $k$-extremal. Similarly, $D_2$ is $k$-extremal. 
\end{proof}

Note that the reciprocal of lemma~\ref{lem:extremal_bijoin_Haj} does not hold, observe for example that the Haj\'os bijoin of two $\bid{K}_{k+1}$ is $k$-dicolourable and is thus not $k$-extremal. See Figure~\ref{fig:bijoin_not_reciprocal}.

    \begin{figure}[!hbtp]
    \begin{center}
        \begin{tikzpicture}[scale = 1.5]

            \begin{scope}
                \vertex[color=green] (r) at (0,0) {\textcolor{black}{a}};
                \vertex[color=blue] (l) at (-2,0) {};
                \vertex[color=red] (u) at (-1,1) {\textcolor{black}{t}};
                \vertex[color=green] (d) at (-1,-1) {\textcolor{black}{w}};
                \draw[->-, bend right=15] (d) to (r);
                \draw[->-, bend right=15] (r) to (u);
                
                \draw[->-, bend left=15] (r) to (l);
                \draw[->-, bend left=15] (l) to (r);
                \draw[->-, bend left=15] (u) to (d);
                \draw[->-, bend left=15] (d) to (u);
                \draw[->-, bend left=15] (u) to (l);
                \draw[->-, bend left=15] (l) to (u);
                \draw[->-, bend left=15] (d) to (l);
                \draw[->-, bend left=15] (l) to (d);
            \end{scope}
        
            \begin{scope}[shift=(r),xscale=-1,yscale=-1]
                \vertex[color=blue] (l) at (-2,0) {};
                \vertex[color=red] (uu) at (-1,1) {\textcolor{black}{u}};
                \vertex[color=green] (dd) at (-1,-1) {\textcolor{black}{v}};
                \draw[->-, bend right=15] (dd) to (r);
                \draw[->-, bend right=15] (r) to (uu);
                \draw[->-, bend left=15] (r) to (l);
                \draw[->-, bend left=15] (l) to (r);
                \draw[->-, bend left=15] (uu) to (dd);
                \draw[->-, bend left=15] (dd) to (uu);
                \draw[->-, bend left=15] (uu) to (l);
                \draw[->-, bend left=15] (l) to (uu);
                \draw[->-, bend left=15] (dd) to (l);
                \draw[->-, bend left=15] (l) to (dd);
            \end{scope}

            \draw[->-, bend left=15] (u) to (dd);
            \draw[->-, bend left=15] (uu) to (d);

        \end{tikzpicture}
        \end{center}
           \caption{A $4$-dicolouring of a bijoin of two $\bid{K}_{4}$.}
           \label{fig:bijoin_not_reciprocal}
    \end{figure}
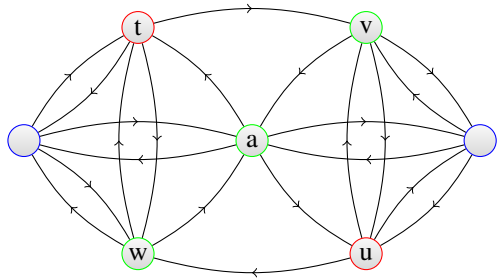

However, we can still prove the following holds:

\begin{lemma}\label{lem:extremal_bijoin_Haj_only_if}
    Let $k \geq 3$. Let $D$ be a Haj\'os bijoin of two digraphs $D_1$ and $D_2$. If, for $i=1,2$,  $D_i$ is biconnected, strong, Eulerian and $\lambda(D_i) \leq k$, then $D$ is biconnected, strong,  Eulerian and $\lambda(D) \leq k$. 
\end{lemma}

\begin{proof} 
    Let, for $i=1,2$,  $D_i$ be a biconnected, strong, Eulerian digraph with $\lambda(D_i) \leq k$.
    Let $D$ be a Haj\'os bijoin of $D_1$ and $D_2$ with respect to $\big((t,a_1,w), (v,a_2, u)\big)$, i.e. there exists $tu,vw \in A(D)$ and $a \in V(D)$ such that $D\setminus \{a\} - \{tu,vw\}$ has two connected components with vertex set $V'_1$ and $V'_2$ such that $D_1=D[V'_1 \cup a] + \{ta,aw\}$, and $D_2= D[V'_2 \cup a] + \{va,au\}$ and $t$ and $w$ are in the same connected component of $D_1\setminus a$ and  $u$ and $v$ are in the same connected component of $D_2\setminus a$. 
    Set $V_1=V'_1 \cup a$ and $V_2 = V'_2 \cup a$.  
    \smallskip 

    Let us first prove that $D$ is biconnected. 
    Assume for contradiction that $D$ has a cutvertex $x$. 
    Since for $i=1, 2$ $D_i$ is biconnected, $D[V'_i]=D_i\setminus a$ is connected, and since moreover there is an (actually two) arc between $D[V'_1]$ and $D[V'_2]$, $D \setminus a$ is connected. Thus $x \neq a$. 
    Assume without loss of generality that $x \in V'_1$ and let $C$ be the connected component of $D \setminus x$ containing $a$. Then $V_2$ is included in $C$, and since there is an arc between $u$ and $V_2$ and between $w$ and $V_2$, $u$ and $w$ are also in $C$. Thus $a$, $u$ and $w$ are in the same connected component of $D \setminus x$, which implies that $x$ is a cutvertex of $D_1$, a contradiction. 
    
    %Let us first prove that $D$ is biconnected. $D \setminus a$ is connected, since $D \setminus A = D[V_1 \cup V_2\setminus a$, $D[V_1] \setminus a$ and $D[V_2] \setminus a$ are biconnected and there is an arc between $V_1$ and $V_2$ ($tu$ for example).  Let $c \neq a$ be a vertex of $D$. Without loss of generality, suppose $c \in V_1$. Let $x \in V(D)$. If $x \in V_2$, then there exists a $xa$-dipath in $D[V_2]$, since $\lambda_{D_2 - \{va,aw\}}(x,a) \geq \lambda_{D_2}(x,a) - 1 = k - 1$. Thus all vertices of $V_2$ are in the same connected component of $D \setminus c$.     If $x \in V_1$, then there exists a $ax$-path in the underlying graph of $D_{1} \setminus c$. Thus, $x$ is in the same connected component as either $t$, $w$ or $a$ in the underlying graph of $D_1 \setminus c$. Since $a \in V_2$, $tu \in A(D)$ and $vw \in A(D)$ and vertices of $V_2$ are in the same connected component of $D \setminus c$, $x$ is in the same connected component as $a$. Thus $D \setminus c$ is connected. This proves that $D$ is connected.

    \smallskip

    As $D_1$ and $D_2$ are Eulerian,  so is $D$ by construction. And since an Eulerian biconnected digraph is strong, $D_1$ is strong. 

    \smallskip 
    
    Let $D' = D - \{tu,vw\} + \{ta,au,va,aw\}$. As the blocks of $D'$ are $D_1$ and $D_2$, $\lambda(D') \leq k$. Hence, by applying twice Lemma~\ref{lem:lambda_diminishing}, we get that $\lambda(D) \leq \lambda(D')  \leq k$.  
\end{proof}

The following lemma is an analogue of Lemma~\ref{lem:HJsufficient} in the case of degenerated Haj\'os bijoins. %\PA{Proving it in its full generality of (non-degenerated) Haj\'os bijoin is substantially more difficult, and we don't need it anyway, but oddly, to prove it for degenerated, it is more easy to prove}
\begin{lemma}\label{lem:HB_sufficient}
    Let $k \geq 3$. Let $D$ be a $k$-extremal digraph. Suppose there exists $tu$ and $uw$ in $A(D)$, such that $D - \{tu,uw\}$ has a cutvertex. Then $D$ is a directed Haj\'os join or a Haj\'os bijoin. 
    %has two connected components $D_1, D_2$ with $t,w \in V(D_1)$ and $u,v \in V(D_2)$. Then $ta \notin A(D)$. 
\end{lemma} 

\begin{proof} 
    Note that since $D$ is biconnected, $D - \{tu,uw\}$ is connected. Let $a$ be a cut-vertex of $D - \{tu,uw\}$. Assume first that $t$ and $w$ are in two distinct connected components of $D \setminus a - \{tu,uw\}$ and assume without loss of generality that $u$ is not in the same connected component as $t$. Then $D\setminus a - \{tu\}$ is disconnected, and thus $D$ is a directed Haj\'os join by Lemma~\ref{lem:HJsufficient}. 

    %PREVIOUS VERSION: Let $a$ be a cut-vertex of $D - \{tu,uw\}$. If $t$ and $w$ are not in the same connected component of $D \setminus a - \{tu, uw\}$, then $D \setminus a - tu$ is not connected \PA{faut montrer que u n'a pas d'autres voisins que t dans la composante connexe qui contient t nan?}, and thus $D$ is a Haj\'os join. 

    We may thus assume that $t$ and $w$ are in the same connected component of $D \setminus a - \{tu,uw\}$. To prove that $D$ is a Haj\'os bijoin, it is enough to prove that $ta,aw,ua,au \notin A(D)$. We will prove something a bit stronger.

    \begin{claim}
        Let $k \geq 3$. Let $D'$ be a $k$-extremal digraph. Suppose there exists $a, t,u,v,w \in V(D')$ with $a \notin \{t,u,v,w\}$ and $\{t,u\} \cap \{v,w\} = \emptyset$ (in other words, $t=u$ and $v=w$ are the only possible equalities) such that $tu,vw \in A(D')$, $a$ is a cutvertex of $D' - \{tu,vw\}$ and $D'-\{tu,vw\}$  has two blocks $D_1$ and $D_2$ with $t,w \in V(D_1)$ and $u,v \in V(D_2)$. Then $ta,au,va,aw \notin A(D')$. 
    \end{claim}

    \begin{proofclaim}
        Suppose that $ta \in A(D')$. Let us first prove that $D_1 + aw$ is biconnected, strong, and $\lambda(D_1 + aw) \leq k$. 
        
        Assume for contradiction that $D_1 + aw$ has a cutvertex $x$. Observe that $\{t,w,a\} \setminus x$ are in the same connected component of $D_1 + aw \setminus x$. Indeed, if $x=a$ it is by hypothesis, and otherwise it is because $ta, aw \in A(D_1) \cup \{aw\} $. Hence, $D' \setminus x$ has a connected component included in $D_1$ and disjoint from $\{t,a,w\}$, and thus $x$ is a cutvertex of $D'$, a contradiction. 
    
        Since for any two vertices $u,v$ of $D'$, $\lambda_{D'}(u,v) = k$, we have that $\lambda_{D' + aw - \{tu,vw\}}(u,v) \geq k - 2 \geq 1$, and thus $D' + aw - \{tu,vw\}$ is strong. As $D_1 + aw$ is a block of $D' + aw - \{tu,vw\}$, $D_1 + aw$ is strong.
    
        Let $P_{av}$ be an $av$-dipath in $D_2$, which  exists for $\lambda_{D' - \{tu,vw\}}(a,v) \geq \lambda_{D'}(a,v) - 2 \geq k - 2 \geq 1$. By Lemma~\ref{lem:lambda_diminishing}, $\lambda(D' + aw - A(P_{av}) - vw) \leq \lambda(D') = k$, and since $D_1 + aw$ is a subgraph of $D' + aw - A(P_{av}) - vw$, we have  $\lambda(D_1 + aw) \leq k$.

        Hence, $D_1 + aw$ is biconnected, strong, and $\lambda(D_1 + aw) \leq k$. 
        Yet $D_1 + aw$ is not Eulerian, for $D'$ is Eulerian, and the indegree of $t$ does not change in $D_1$ while its outdegree decreases by $1$. Thus, by Lemma~\ref{lem:prop_k-extremal}, $D_1 + aw$ is not $k$-extremal, hence $\dic(D_1 + aw) \leq k$. Let $\varphi_2$ be a $k$-dicolouring of $D_2$. Let $\varphi_1$ be a $k$-dicolouring of $D_1 + aw$ chosen so that 
        \begin{itemize}
            \item $\varphi_1(a) = \varphi_2(a)$ and, 
            \item if $\varphi_1(a) \neq \varphi_1(t)$, then $\varphi_1(t) \neq \varphi_2(u)$ (which is always possible up to permuting colours, since $k \geq 3$).
        \end{itemize}
       Consider $\varphi:V(D') \rightarrow [k]$ such that  %\PA{qui sont $x$ et $x_1$? ça ressemble à du copier coller, je pense que c'est $a$}         
        \[
        \varphi(y) = \left\{
            \begin{array}{ll}
                %\varphi_1(x_1) & \text{ if } y=x \\
                \varphi_1(y)   & \text{ if } y \in V(D_1)\\ 
                \varphi_2(y)   & \text{ if } y \in V(D_2)
            \end{array}
        \right.
        \]
        Since $\dic(D') = k+1$, $\varphi$ contains a monochromatic dicycle $C$. Since $\varphi_1$ and $\varphi_2$ are $k$-dicolourings of respectively $D_1 + aw$ and $D_2$, $C$ intersects both $V(D_1) \setminus a$ and $V(D_2) \setminus a$. \\
        If $C$ contains $tu$ but not $vw$, then $C$ goes through $a$, and  there is a monochromatic $at$-dipath in $D_1+aw$ which, together with the arc $ta$, forms a monochromatic dicycle with respect to $\varphi_1$, a contradiction to the fact that $\varphi_1$ is a dicolouring of $D_1 + aw$. \\
        If $C$ contains $vw$ but not $tu$, then $C$ goes through $a$, and  there is a monochromatic $wa$-dipath in $D_1 +wa$ which, together with the arc $aw$, forms a monochromatic dicycle with respect to $\varphi_1$, a contradiction with the fact that $\varphi_1$ is a dicolouring of $D_1 + aw$.\\
        If $C$ contains both $tu$ and $vw$, then there is a monochromatic $wt$-dipath in $D_1$ with respect to $\varphi$. Since $ta, aw \in A(D_1 + aw)$ and $\varphi_1$ is a $k$-dicolouring of $D_1 + aw$, this implies that $\varphi_1(a) \neq \varphi(t)$. But then $\varphi(t) \neq \varphi(u)$ by the choice of $\varphi_1$ and $\varphi$, a contradiction.

        Hence, $\varphi$ has no monochromatic dicycle, a contradiction. 
        Thus we have proven that $ta \notin A(D')$. By symmetry, we have that $ta,au,va,aw \notin A(D')$.
    \end{proofclaim}

    Applying this claim with $u = v$, this proves $ta, au, ua, aw \notin A(D)$ and thus that $D$ is a Haj\'os bijoin.

\end{proof}

%%%%%%%%%%%%%%---------------------------------------------------------%%%%%%%%%%%%%%%%%%%%%%
%%%%%%%%%%%%%%---------------------------------------------------------%%%%%%%%%%%%%%%%%%%%%%

\subsection{Proof of Theorem~\ref{thm:decHJHBJ}}
    
    Let $k \geq 3$. 
    We prove the theorem by induction on the number of vertices.  Let $D$ be $k$-extremal, and assume by contradiction that $D$ is neither a symmetric complete graph,  a symmetric odd wheel, a directed Haj\'os  join, nor a Haj\'os bijoin.

    Given a digraph  $D$, a \emph{flower} of $D$ is an induced subdigraph $F$ of $D$ isomorphic to a symmetric path $P$ with an even number of vertices plus a vertex $x$ linked to each vertex of $P$ by a digon, and such that internal vertices of $P$ has no neighbour outside $F$, while other (that is $x$ and the two extremities of $P$) have exactly one inneighbour and one outneighbour outside $F$. The  vertex $x$ is called the \emph{center} of $F$.

    \begin{claim}\label{clm:coupe_extreme}
    $D$ has a minimum dicut $(X, \overline{X})$ such that either $k \geq 4$ and $D[\overline{X}] = \bid{K}_k$, or $k=3$ and $D[\overline{X}]$ is a flower of $D$. Moreover, if $D[\overline{X}] = \bid K_k$, then each vertex of $\overline{X}$ has exactly one inneighbour and one outneighbour in $X$.
    \end{claim}

    \begin{proofclaim}    
        By  Lemma~\ref{lem:all_cuts_isolate_vertices} and Lemma~\ref{lem:all_cuts_isolate_vertices_three}, $D$ has  a minimum dicut $(A,\overline{A})$  such that $|A| > 1$ and $|\overline{A}| > 1$. 
        By Lemma~\ref{lem:contraction_extremal}, up to permuting $A$ and $\overline{A}$,  we may assume that $D / A$ is $k$-extremal and thus by induction is either a symmetric complete graphs on $k+1$ vertices, a symmetric complete wheel, a Haj\'os bijoin of two $k$-extremal digraphs or a directed Haj\'os  join of two $k$-extremal digraphs.   
        
        Let $a$ be the vertex into which $A$ is contracted in $D / A$. Observe that, since $(A, \overline{A})$ is a minimum dicut and $D$ is Eulerian, $a$ has indegree an outdegree $k$ in $D/A$. 
        
        Assume first that there exist two $k$-extremal digraphs $D_1$ and $D_2$ such that $D / A$ is either a Haj\'os bijoin of $D_1$ and $D_2$ with respect to $\big((t,b_1,w), (v,b_2, u)\big)$ or a directed Haj\'os  join of $D_1$ and $D_2$ with respect to $(ub_1, b_2w)$, and let $b$ the vertex into which $b_1$ and $b_2$ are identified in $D$. Then, as $b_1$ and $b_2$ both have outdegree at least $k$ in respectively $D_1$ and $D_2$, $b$ has outdegree at least $2k-2 > k$ in $k$, and thus $b \neq a$. 
        We may thus assume that $a$ is in $D_1 \setminus a$ or in $D_2 \setminus a$. Now, when one un-contracts the vertex $a$ to get the original digraph $D$, it is clear that the structure of the directed Haj\'os join or Haj\'os bijoin is preserved, i.e. $D$ is a directed Haj\'os join or a Haj\'os bijoin,  a contradiction.
        
        Hence, we may assume that $D/A$ is a symmetric complete graph or a symmetric odd wheel in which $a$ is a vertex of outdegree $k$. Hence, $D[\overline{A}]$ is either a $\bid K_k$, or a symmetric path $P$ with an even number of vertices plus a vertex $x$ linked via a digon to each vertex of $P$. 

        Assume first that $D[\overline{A}]=\bid K_k$. Then each vertex of $\overline{A}$ has at least one inneighbour and one outneighbour in $A$ (because they linked to $a$ via a digon in $D/A$), and since $(A,\overline{A})$ is a minimum dicut, it has exactly one of each. 
        Assume now that we are in the  case where $D/A$ is a symmetric odd wheel, and thus $D[\overline{A}]$ is a symmetric path $P$ with an even number of vertices plus a vertex $x$ linked via a digon to each vertex of $P$. Observe that $x$ and each extremity of $P$ has at least one inneighbour and one outneighbour in $A$ (because they are linked via a digon to $a$ in $D/A$). Hence, since $(A, \overline{A})$ is a minimum dicut, $D[\overline{A}]$ is a flower. 
        
        Hence, by taking $X=A$, we get the desired properties.

    \end{proofclaim}

    %Since $(X, \overline{X})$ is a minimum dicut and  each vertex $v \in \overline{X}$ has exactly $k-1$ in-neighbours and exactly $k-1$ out-neighbours in $\overline{X}$, each vertex in $\overline{X}$ has exactly one in- and one out-neighbour in $X$.
    %The frontier of $\overline{X}$ is the set of vertices of $\overline{X}$ which have a neighbour in $X$. Thus, iff $k$ $k \geq 4$ 
        
    We distinguish between two cases depending on whether or not there is a vertex of $\overline{X}$ that has two distinct in- and out-neighbour in $X$.%: either a vertex in $\overline{X}$ have distinct in- and out-neighbour in $X$, or every vertex in $\overline{X}$ is linked to a unique vertex in $X$ via a digon. 
    \smallskip 

    \noindent\textbf{Case 1}: There exist $v \in \overline{X}$ and $u,w \in X$ such that $uv, vw \in A(D)$ and $u \neq w$
        
    Let $D' = D - \{uv,vw\} + \{uw\}$. Due to Lemma~\ref{lem:lambda_diminishing}, $\lambda(D') \leq \lambda(D) = k$. Let us now show that $\dic(D') \geq k+1$. Suppose for contradiction that  $D'$ admits a $k$-dicolouring $\varphi$. 
    %As $ab \in A(D')$, there is no monochromatic $ba$-dipath in $D'$. 
    As $\dic(D) = k+1$, $\varphi$ is not a $k$-dicolouring of $D= D' + \{uv, vw\} -\{uw\}$, and thus there is a monochromatic dicycle $C$ in $D' + \{uv, vw\} -\{uw\}$ containing $uv$ or $vw$ or both. 
    By claim~\ref{clm:coupe_extreme}, $v$ is linked via a digon to all its neighbours except for $u$ and $w$, so $C$ contains both $uv$ and $vw$. By replacing the arcs $uv$ and $vw$ by $uw$ in $C$, we get a monochromatic dicycle in $D'$, a contradiction. 
    Thus $k+1 \leq \dic(D') \leq \lambda(D') + 1 \leq k+1$ and hence $\dic(D') = \lambda(D') + 1 = k+1$. 

    Since there are $k$ arc-disjoint dipaths between any pair of vertices in $D$, there are at least $k-2 \geq 1$ arc-disjoint dipaths between any pair of vertices in $D'$, thus $D'$ is strong.
    Since $|\partial_{D'}^{+}(X)| = k-1$, $D'$ is not $k$-extremal by  Lemma~\ref{lem:prop_k-extremal}.  Thus $D'$ is not biconnected. 
    
    Let $a$ be a cutvertex of $D'$. Since $uw \in A(D')$, $u$ and $w$ are together in a block $B_1$ of $D'$. Since $D[\overline{X}]$ is either $\bid K_k$ or a flower of $D$, $\overline{X}$ is included in a block $B_2$ of $D'$. If $D'$ has a block distinct from $B_1$ and $B_2$, then $a$ is a cutvertex of $D$, a contradiction. Thus $B_1$ and $B_2$ are distinct and are the only blocks of $D'$. Now, $a$ is a cutvertex of $D- \{uv,vw\}$, and thus, by Lemma~\ref{lem:HB_sufficient}, $D$ is a Haj\'os bijoin or a directed Haj\'os join, a contradiction. 

    \medskip 
    
    \noindent\textbf{Case 2:} There are only digons between $X$ and $\overline{X}$ %Every vertex in $\overline{X}$ is linked to a unique vertex in $X$ via a digon.

     Assume $D[X]$ is not strong and let $C_t$ be a terminal component of $D[X]$. We have  $\partial^+(C_t) \subseteq \overline X$ and $|\partial^+(C_t)| \geq k$, so the digons linking $X$ and $\overline{X}$ are all incident with some vertex of $C_t$. Thus $D$ is not strong, a contradiction. Hence $D[X]$ is strong.

    \noindent\textbf{Case 2a:} $D[X]$ is not biconnected.
    
    Consider $B = (B_1, \dots, B_n)$ a longest path of blocks in $D[X]$. Since $D[X]$ is not biconnected, $n \geq 2$. 
    
    Let $V(B_1) \cap V(B_{2}) = \{a\}$. 
    There is a digon between $V(B_1) \setminus a$ and $\overline{X}$, for otherwise $a$ is a cutvertex of $D$. 
    Suppose there is only one digon $[b,y]$ between $V(B_1) \setminus a$ and $\overline{X}$, with $b \in V(B_1) \setminus a$ and $y \in \overline{X}$. 
    Then $D \setminus a - [b,y]$ is not connected, thus by Lemma~\ref{lem:HB_sufficient}, $D$ is a directed Haj\'os join or a Haj\'os bijoin, a contradiction.
    
    Thus we can assume that there are at least two digons between $\overline{X}$ and $V(B_1) \setminus V(B_2)$, say $[a_1,y_1]$ and $[b_1,y_1']$ with $a_1, b_1 \in V(B_1) \setminus V(B_2)$ and $y_1,  y_1' \in \overline{X}$. 
    Similarly, there are two digons between $\overline{X}$ and $V(B_n) \setminus V(B_{n-1})$, say $[a_n,y_n]$ and $[b_n,y_n']$ with $a_n, b_n \in V(B_n) \setminus V(B_{n-1})$. 
    Note that $a_1 = b_1$ and $a_n = b_n$ are possible, but $y_1, y'_1, y_n, y'_n$ are pairwise distinct by claim~\ref{clm:coupe_extreme}. As there are at least $4$ digons between $X$ and $\overline{X}$, we have $k \geq 4$.

    Set $H= B + [a_1, a_n]$ and let us prove that $\lambda(H) = k$. 
    
    By Lemma~\ref{lem:mono_vs_rainbow}, $a_1$ and $a_n$ receive the same colour in all $k$-dicolouring of $D[X]$, and since any $k$-dicolouring of $B$ can easily be extended to a $k$-dicolouring of $D[X]$, the same holds for any $k$-dicolouring of $B$, and thus  
    $\dic(H) \geq k+1$. Let $ P_{a_1a_n} = a_1 \rightarrow y_1 \rightarrow y_n \rightarrow a_n$ and  $P_{a_na_1} = a_n \rightarrow y_n \rightarrow y_1 \rightarrow a_1$. 
    By Lemma~\ref{lem:lambda_diminishing},  $\lambda(D+[a_1, a_n] - A(P_1) - A(P_2)) \leq \lambda(D) = k$, and thus $\lambda(H) \leq k$. 
    Thus $k+1 \leq \dic(H) \leq  \lambda(H) + 1 \leq k+1$, so $\lambda(H) = k$. 
        
    Hence there are $k$ arc-disjoint $b_1 b_n$-dipaths in $H$. But replacing any potential use of $a_1a_n$ by $P_{a_1a_n}$ and any potential use of $a_na_1$ by $P_{a_na_1}$, and considering the dipath $b_1 \rightarrow y'_1 \rightarrow y'_n \rightarrow b_n$, we get $k+1$ arc-disjoint  $b_1b_n$-dipaths in $D$, a contradiction.
    \medskip 
    
    \noindent\textbf{Case 2b:} $D[X]$ is biconnected.

    If a vertex $a \in X$ is adjacent to every vertex in $\overline{X}$, then $a$ is a cutvertex of $D$, a contradiction.%$D[\overline{X} \cup a] = \bid K_{k+1}$ and thus $D= \bid{K}_{k+1}$, a contradiction. 

% \PC{pourquoi on peut pas conclure en disant direct ce qui suit ? Si il y a au moins deux sommets dans $X$ avec des voisins dans la clique, alors on considere le graphe induit par X où on rajoute le digone entre ces deux sommets. Ce digraphe ne peut pas etre $k$ colourable (car $X$ est le cote mono et ces deux sommets doivent forcemment avoir la meme couleur dans un $k$ colouring de $X$),  lambda ne peut pas augmenter (argument classique, car on peut facilement passer par la clique) et il est clairement strong (on a tué au plus $k$ chemins entre deux sommets en supprimant la clique). Du coup $X$ a un cutvertex (meme avec le digone en plus). Ah mais merde en fait le pb c'est que ce graphe peut ne pas avoir de cutvertex, dans ce cas il est $k$-extremal et faut repartir dans l'induction.  }\PA{Dans ce cas par induction il a un join, et on montre que ce join est aussi un join de D, c'est ce qu'on fait dans le cas k=3, à voir si ça passe}
    
    % Assume now that there exist $a,b \in X$ such that $a$ is adjacent with $k-1$ vertices in $\overline{X}$ and $b$ with one, say $b'$. Then $D\setminus a - [b,b']$ is not connected and thus by Lemma~\ref{lem:HB_sufficient}, $D$ is a Haj\'os bijoin or a Haj\'os join, a contradiction. 

    Let $a, b \in X$, $a \neq b$, such that there exist $a',b' \in \overline{X}$ with $[a,a'], [b,b'] \subseteq A(D)$. If $k=3$, let them be chosen so that neither $a'$ nor $b'$ is the center of $D[\overline{X}]$.
    Let $D' = D[X] + [a,b]$.
    Since $D[X]$ is strong and biconnected, so is $D'$.
    By Lemma~\ref{lem:mono_vs_rainbow}, $\dic(D') \geq k + 1$ for in every $k$-dicolouring $\varphi$ of $D[X]$, $\varphi(a) = \varphi(b)$.
    By Lemma~\ref{lem:lambda_diminishing}, $\lambda(D') \leq \lambda(D) = k$.
    Thus $D'$ is $k$-extremal.

    By induction, $D'$ is either a symmetric odd wheel, a symmetric complete graph, a directed Haj\'os join of two digraphs $D_1$ and $D_2$ or a Haj\'os bijoin of two digraphs $D_1$ and $D_2$.

    If $D'$ is a directed Haj\'os join of two digraphs $D_1$ and $D_2$, then either $a, b \in V(D_1)$ or $a, b \in V(D_2)$, and thus $D$ is a directed Haj\'os join, a contradiction.

    If $D'$ is a Haj\'os bijoin of two digraphs $D_1$ and $D_2$ with respect to $((t,x_1,w),(v,x_2,u))$, then if either $a, b \in V(D_1)$ or $a,b \in V(D_2)$, $D$ is itself a Haj\'os bijoin, a contradiction. Otherwise, we have $t=w$, $v = u$ and $\{t,u\} = \{a,b\}$, which contradicts that $D[X]$ is biconnected.

    Thus $D'$ is  $\bid K_{k+1}$ or a symmetric odd wheel. 

    Suppose there is no vertex in $X \setminus \{a,b\}$ with a neighbour in $\overline{X}$. Then  $a$ and $b$  both have at least two neighbours in $\overline{X}$, for otherwise either $D \setminus [a,a']$ or $D \setminus [b,b']$ has a cutvertex, and thus by Lemma~\ref{lem:HB_sufficient}, $D$ is a directed Haj\'os join or a Haj\'os bijoin, a contradiction. 
    Note that this implies that $k \geq 4$. Let $a',a'' \in \overline{X} \cap N(a)$ and $b', b'' \in \overline{X} \cap N(b)$. Then, since $\lambda_{D'}(a,b) = k$, we have that $\lambda_{D[X]}(a,b) \geq k-1$, that is there exist $k-1$ arc-disjoint $ab$-dipaths in $D[X]$. But then, since $a \rightarrow a' \rightarrow b' \rightarrow b$ and $a \rightarrow a'' \rightarrow b'' \rightarrow b$ are arc-disjoint  $ab$-dipaths that do not use any arc of $D[X]$, we have that $\lambda_{D}(a,b) \geq k + 1$, a contradiction.

    Thus, there exists a vertex $c \in X \setminus \{a,b\}$ with a neighbour $c' \in \overline{X}$.     
    Observe that $c$ is not adjacent with $a$ nor $b$, for either $[a,c] \in A(D)$ or $[b,c] \in A(D)$ (because $D' = D[X] + [a,b]$ is symmetric) and in any $k$-dicolouring $\varphi$ of $D[X]$, $\varphi(a) = \varphi(b) = \varphi(c)$ by Lemma~\ref{lem:mono_vs_rainbow}, a contradiction. This implies that $D' \neq \bid K_{k+1}$, and thus $k=3$ and $D'$ is a symmetric odd wheel. Let $x$ be the center of $D'$. Since $c$ is neither a neighbour of $a$ nor of $b$, we have that $x \notin \{a,b,c\}$. Let $d,e$ be the two other neighbours of $c$ in $D$. Then, $x \rightarrow d \rightarrow c$, $x \rightarrow e \rightarrow c$, $x \rightarrow a \rightarrow a' \rightarrow c' \rightarrow c$ and $x \rightarrow c$ are $4 > k$ arc-disjoint $xc$-dipaths, a contradiction.

%------------------------------------------------------------%
%------------------------------------------------------------%
%------------------------------------------------------------%

\section{Haj\'os tree joins - Structure Theorems}\label{sec:mainthm}
At the end of this section, we will prove the main result of this article: $k$-extremal digraphs are exactly the digraphs in $\hk$, which we recall is a class built from  $\bid K_{k+1}$ (for $k\geq 4$) or symmetric odd wheels (for $k = 3$) using directed Haj\'os  joins and Haj\'os tree joins. In order to simplify our arguments, we will prove in fact an equivalence with another class called $\htk$ that is based on a single operation called extended Haj\'os tree join (which generalizes both Haj\'os tree joins and directed Haj\'os join) defined below. See Figure~\ref{fig:extended_hajos_tree_join}.

An \emph{Euler tour} is a closed trail that traverses each arc exactly once. 
Let $T$ be a tree. An \emph{Eulerian list} $C$ of $T$  is the circular list of the vertices of $T$ encountered following an Eulerian tour of $\bid T$ (note that there is a unique such circular list). A \emph{partial Eulerian list} $C'$ of $T$ is a circular sublist of an Eulerian list of $T$  with the following properties: 
\begin{itemize}
    \item each leaf of $T$ is in $C'$, and 
    \item each non-leaf vertex of $T$ appears at most once in $C'$. 
\end{itemize}

\begin{definition}[Extended Haj\'os tree join]\label{def:EHTJ}
Given 
\begin{itemize}
    \item a tree $T$ with edges $\{u_1v_1, \dots, u_nv_n\}$
    \item a  partial Eulerian list $C=(x_1, \dots, x_{\ell})$ of $T$,  and
    \item for $i=1, \dots, n$, $D_i$ a digraph such that 
    \begin{itemize}
        \item $V(D_i) \cap V(T) = \{u_i,v_i\}$, 
        \item $[u_i,v_i] \subseteq A(D_i)$, and 
        \item for $1 \leq i \neq j \leq n$, $V(D_i) \setminus \{u_i, v_i\} \cap V(D_j) \setminus \{u_j, v_j\} = \emptyset$. 
    \end{itemize}
\end{itemize}
We define the \emph{extended Haj\'os tree join} $T(D_1, \dots, D_n;C)$ to be the digraph $D$ obtained from $D_i -[a_i,b_i]$ for $i=1, \dots, n$ by adding the dicycle $C = x_1 \ra x_2 \ra \dots \ra x_{\ell} \ra x_1$. 

We say that $D$ is the \emph{extended Haj\'os tree join} of $(T, D_1, \dots, D_n)$ with respect to $C$. 

$C$ is called the  \emph{peripheral cycle} of $D$  and vertices $u_1, v_1, \dots, u_n,v_n$ are the \emph{junction vertices} of $D$ (note that there are $n-1$ of them). \\

\end{definition}

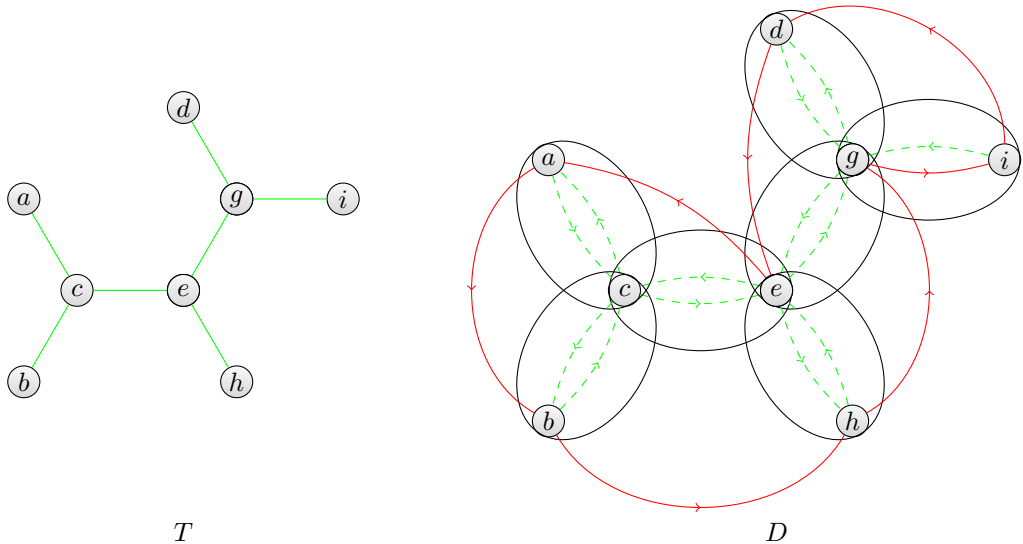
\begin{figure}[!hbtp]
    \begin{center}
        \begin{tikzpicture}[scale=0.4]

            \begin{scope}[xshift=-18cm,scale = 0.7]              
                \begin{scope}
                    \vertex (c) at (0,0) {$c$};
                    \vertex (e) at (5,0) {$e$};
                    \draw[green] (c) -- (e);
                \end{scope}
    
                \begin{scope}[shift=(c), rotate = 120]
                    \vertex (a) at (5,0) {$a$};
                    \draw[green] (a) -- (c);
                \end{scope}
    
                \begin{scope}[shift=(c), rotate = -120]
                    \vertex (b) at (5,0) {$b$};
                    \draw[green] (b) -- (c);
                \end{scope}
    
                \begin{scope}[shift=(e), rotate = 60]
                    \vertex (e) at (0,0) {$e$};
                    \vertex (g) at (5,0) {$g$};
                    \draw[green] (e) -- (g);
                \end{scope}
    
                \begin{scope}[shift=(e), rotate = -60]
                    \vertex (e) at (0,0) {$e$};
                    \vertex (h) at (5,0) {$h$};
                    \draw[green] (e) -- (h);
                \end{scope}
    
                \begin{scope}[shift = (g)]
                    \vertex (g) at (0,0) {$g$};
                    \vertex (i) at (5,0) {$i$};
                    \draw[green] (i) -- (g);
                \end{scope}
                
                \begin{scope}[shift = (g), rotate=120]
                    \vertex (g) at (0,0) {$g$};
                    \vertex (d) at (5,0) {$d$};
                    \draw[green] (d) -- (g);
                \end{scope}

            \end{scope}

            \node () at (e |- 0, -8) {$T$};

            \begin{scope}
                \vertex (c) at (0,0) {$c$};
                \vertex (e) at (5,0) {$e$};
                \draw[->-, bend right = 15, green, dashed] (c) to (e);
                \draw[->-, bend right = 15, green, dashed] (e) to (c);
                \draw (2.5,0) ellipse (3cm and 2cm) {};
            \end{scope}

            \begin{scope}[shift=(c), rotate = 120]
                \vertex (a) at (5,0) {$a$};
                \draw[->-, bend right = 15, green, dashed] (c) to (a);
                \draw[->-, bend right = 15, green, dashed] (a) to (c);
                \draw (2.5,0) ellipse (3cm and 2cm) {};
            \end{scope}

            \begin{scope}[shift=(c), rotate = -120]
                \vertex (b) at (5,0) {$b$};
                \draw[->-, bend right = 15, green, dashed] (c) to (b);
                \draw[->-, bend right = 15, green, dashed] (b) to (c);
                \draw (2.5,0) ellipse (3cm and 2cm) {};
            \end{scope}

            \begin{scope}[shift=(e), rotate = 60]
                \vertex (e) at (0,0) {$e$};
                \vertex (g) at (5,0) {$g$};
                \draw[->-, bend right = 15, green, dashed] (e) to (g);
                \draw[->-, bend right = 15, green, dashed] (g) to (e);
                \draw (2.5,0) ellipse (3cm and 2cm) {};
            \end{scope}

            \begin{scope}[shift=(e), rotate = -60]
                \vertex (e) at (0,0) {$e$};
                \vertex (h) at (5,0) {$h$};
                \draw[->-, bend right = 15, green, dashed] (e) to (h);
                \draw[->-, bend right = 15, green, dashed] (h) to (e);
                \draw (2.5,0) ellipse (3cm and 2cm) {};
            \end{scope}

            \begin{scope}[shift = (g)]
                \vertex (g) at (0,0) {$g$};
                \vertex (i) at (5,0) {$i$};
                %\draw[->-, bend right = 15, green, dashed] (g) to (i);
                \draw[->-, bend right = 15, green, dashed] (i) to (g);
                \draw (2.5,0) ellipse (3cm and 2cm) {};
            \end{scope}
            
            \begin{scope}[shift = (g), rotate=120]
                \vertex (g) at (0,0) {$g$};
                \vertex (d) at (5,0) {$d$};
                \draw[->-, bend right = 15, green, dashed] (g) to (d);
                \draw[->-, bend right = 15, green, dashed] (d) to (g);
                \draw (2.5,0) ellipse (3cm and 2cm) {};
            \end{scope}

            \node () at (e |- 0, -8) {$D$};

            \draw[->-, bend right = 60, red] (a) to (b);
            \draw[->-, bend right = 60, red] (b) to (h);
            %\draw[->-, bend right = 60, red] (h) to (i);
            \draw[->-, bend right = 60, red] (h) to (g);
            \draw[->-, bend right = 15, red] (g) to (i);
            \draw[->-, bend right = 60, red] (i) to (d);
            \draw[->-, bend right = 20, red] (d) to (e);
            \draw[->-, bend right = 20, red] (e) to (a);

        \end{tikzpicture}
        \end{center}
           \caption{A cartoonish drawing of an extended Haj\'os tree join $D$. Its peripheral cycle is in red. Removed digons are in dashed green. $T$ is the corresponding tree.}%\PC{on peut mettre les arcs hg-gi??}}
           \label{fig:extended_hajos_tree_join}
        \end{figure}

Observe that an extended Haj\'os tree join in which the partial Eulerian list only uses leaves of $T$ is a Haj\'os tree join. Observe also that the peripheral cycle may be a digon in the case where $T$ is a path and the partial Eulerian list only contains the extremities of $T$. Note that every digraph $D$ that contains a digon is an extended \Haj tree join of itself, as if we denote by $T$ the tree consisting of the single edge $uv$, then we have the trivial identity $D=T(D,[uv])$.

We denote by $W_{\ell}$ the wheel on $\ell+1$ vertices. 
\begin{definition}[The class $\htk$]
For $k \geq 4$, let $\htk$ be the smallest class of digraphs that contains $\bid K_{k+1}$ and is closed under taking extended Haj\'os tree joins. Let $\htthree$ be the smallest class of digraphs that contains $\bid W_{2\ell+1}$, for every integer $\ell \geq 1$, and is closed under taking extended Haj\'os tree joins.
\end{definition}

%------------------------------------------------------------------------------------------------------------------------------%

%------------------------------------------------------------------------------------------------------------------------------%

The two following lemmata imply (by induction on the number of vertices) that $\htk\subseteq \hk$.

\begin{lemma}\label{lem:HTK=HJorHT}
Let $k \geq 3$. Let $D\in \htk$, $D$ distinct from a symmetric complete digraph and a symmetric odd wheel. Then $D$ is either a directed Haj\'os join or a Haj\'os tree join of  digraphs in $\htk$.
\end{lemma}
\begin{proof} 
    Denote $D=T(D_1, \dots, D_n;C)$ as in Definition~\ref{def:EHTJ}. If $C$ does not use any internal vertex of $T$, then $D$ is a Haj\'os tree join and we are done. So we can assume $C$ uses an internal vertex $v$ of $T$. Let $X$ be the connected component of $T \setminus v$ that contains the out-neighbour of $v$ in $C$. 
    Let $\mathcal L_1$ be the list of digraphs corresponding to the edges of $T[X \cup v]$, and $\mathcal L_2 = \{D_1, \dots, D_n\}\setminus \mathcal L_1$. As $C$ is a partial Eulerian list, which is thus obtained from an Eulerian tour, it is of the form $C = vP_{X}P_{\overline{X}}$ where $P_X$ is the portion of $C$ contained in $X$, and $P_{\overline X}$ is the portion of $C$ contained in $V(T)\setminus (X\cup v)$.  
    Note that, since $v$ is a cutvertex of $T$, $P_{X}$ and $P_{\overline{X}}$ are non-empty

    Let $x$ be the last vertex of $P_{X}$ and $y$ be the first element of $P_{\overline{X}}$. Then $D \setminus v - xy$ is disconnected, and thus $D$ is a directed Haj\'os join of %$D_1 = D[P_{X} \cup \{v\}]$ 
    $D'_1 = D[\bigcup_{G \in \mathcal L_1} V(G)] + xv$, 
    and $D'_2 = D[\bigcup_{G \in \mathcal L_2} V(G)] + vy$, which are clearly in $\htk$.

\end{proof}

%------------------------------------------------------------------------------------------------------------------------------%
\begin{lemma}\label{lem:HJ_preserve_htk}
Let $k \geq 3$. A directed Haj\'os join of two digraphs in $\htk$ is in $\htk$.
\end{lemma}

\begin{proof}
    Let $D$, $D'$ be two digraphs in $\htk$, $uv_1 \in A(D)$, $v_2w \in A(D')$ and let $H$ be the directed Haj\'os join of $D$ and $D'$ with respect to $(uv_1, v_2w)$. We call $v$ the vertex obtained after identifying $v_1$ and $v_2$. 
     We want to prove that $H \in \htk$. We distinguish four cases.

    \noindent\textbf{Case 1:} $D=D'=\bid K_{k+1}$ or $D =\bid W_{2\ell + 1}$ and $D' = \bid W_{2\ell' + 1}$ for some $\ell, \ell' \geq 1$.\\ 
    Let $T = (\{u,v,w\}, \{uv,vw\})$ be the path of length $2$ and let $C=u \ra v \ra  w \ra u)$. Observe that $(u,v,w,u)$ is a partial Eulerian list of $T$. Then it is easy to check that $T(D, D'; C)$ is the directed Haj\'os join of $D$ and $D'$.  See Figure~\ref{fig:Hajos_K_4}. This proves Case 1. \smallskip

    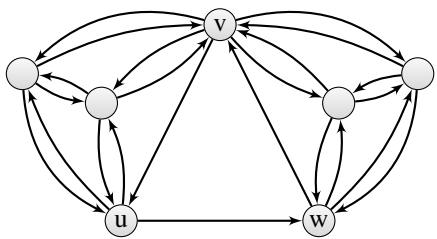
\begin{figure}[!hbtp]
    \begin{center}
        \begin{tikzpicture}[scale=1.3]
            \vertex (l1) at (-1,0) {u};
            \vertex (l2) at (-1.2,1.2) {};
            \vertex (l3) at (-2,1.5) {};
            \vertex (l4) at (0,2) {v};
        
            \vertex (r1) at (1,0) {w};
            \vertex (r2) at (1.2,1.2) {};
            \vertex (r3) at (2,1.5) {};
            
            \draw[arc, bend right=15] (l1) to (l2);
            \draw[arc, bend right=15] (l2) to (l1);
            \draw[arc, bend right=15] (l3) to (l2);
            \draw[arc, bend right=15] (l2) to (l3);
            \draw[arc, bend left = 13] (l1) to (l3);
            \draw[arc, bend right] (l3) to (l1);

            \draw[arc, bend right=15] (l2) to (l4);
            \draw[arc, bend right=15] (l4) to (l2);
            \draw[arc, bend left = 13] (l3) to (l4);
            \draw[arc, bend right] (l4) to (l3);
            
            \draw[arc, bend right=15] (r2) to (l4);
            \draw[arc, bend right=15] (l4) to (r2);
            \draw[arc, bend right = 13] (r3) to (l4);
            \draw[arc, bend left] (l4) to (r3);

            \draw[arc] (l4) to (l1);
            \draw[arc] (r1) to (l4);
            \draw[arc] (l1) to (r1);
            
            \draw[arc, bend right=15] (r1) to (r2);
            \draw[arc, bend right=15] (r2) to (r1);
            \draw[arc, bend right=15] (r3) to (r2);
            \draw[arc, bend right=15] (r2) to (r3);
            \draw[arc, bend right = 13] (r1) to (r3);
            \draw[arc, bend left] (r3) to (r1);
        \end{tikzpicture}
    \end{center}
    \caption{\label{fig:Hajos_K_4} The directed Haj\'os join of two $\bid K_4$.}
    \end{figure}

    From now on, we may assume that $D = T(D_1, \dots, D_{n'};C)$ for some tree $T$, some digraphs $D_1, \dots, D_{n}$ and a peripheral cycle $C$ built from a partial Eulerian list $L$ of $T$.
    \smallskip 

    \noindent\textbf{Case 2:} $uv_1 \notin A(C)$.\\ 
    In this case, there is $i \in [n]$ such that $uv_1 \in A(D_i)$. 
    By induction, there exists $H' \in \htk$ such that $H'$ is the directed Haj\'os join of $D_i$ and $D'$ with respect to $(uv_1, v_2w)$. 
    Then $H=T(D_1, \dots, D_{i-1}, H', D_{i+1}, \dots, D_{n}; C) \in \htk$. This proves Case 2. \smallskip 
    
    \noindent\textbf{Case 3:} $uv_1 \in A(C)$ and $D'$ is a symmetric complete graph or a symmetric odd wheel.\\
    In particular, $u$ and $v_1$ are vertices of $T$. 
    Then  $H=T'(D_1, \dots, D_n, D'; C') \in \htk$, where $T'$ is obtained from $T$ by adding the vertex $w$ and the edge $v_1w$, and $C'$ is obtained from the partial Eulerian list $L'$ obtained from $L$ by adding $w$ between $u$ and $v_1$ (in other words the peripheral cycle $C'$ is obtained from $C$ by deleting $uv_1$ and adding $uw$ and $wv_1$). This proves Case 3.
    \smallskip 

    \noindent\textbf{Case 4:} $uv_1 \in A(C)$ and $D'$ is neither a symmetric complete graph nor a symmetric odd wheel.\\
    Then $D' = T'(D'_1, \dots, D'_{n'};C')$ for some tree $T'$, some digraphs $D'_1, \dots, D'_{n'}$ and a peripheral cycle $C'$ built from a partial Eulerian list $L'$ of $T$.  

    If $v_2w \notin A(C')$, then the result follows from Case 2. So we may assume that $v_2w \in A(C')$. 

    Since $uv_1 \in A(C)$ and $v_2w \in A(C')$, we have $L=(u,v_1, L_1, u)$ and $L'=(v_2,w,L'_1, v_2)$ for some lists $L_1$ and $L'_1$. 
    
    Let $T_H$ be the tree obtained from $T$ and $T'$ by identifying $v_1$ and $v_2$ to a new vertex $v$. 
    Then $H=T_H(D_1, \dots, D_n,D'_1, \dots, D'_{n'}; C_H)$, where $C_H$ is obtained from the partial Eulerian list $L_H=(v, L_1, u, w, L_2, v)$. This proves Case 4 and the lemma. 
\end{proof}

%------------------------------------------------------------------------------------------------------------------------------%

The following result is crucial, as it will allow us to use Haj\'os bijoins given by Theorem~\ref{thm:decHJHBJ} and preserve the fact of being in $\htk$ (this is the main reason why extended Haj\'os tree joins are more convenient to use that Haj\'os tree join combined with directed Haj\'os join).

\begin{lemma}\label{lem:P3_in_peripheral}
    Let $k \geq 3$. Let $D \in \htk$ and $uv, vw \in A(D)$ with $u \neq w$. If in all $k$-dicolourings of $D \setminus \{uv, vw\}$ there is a monochromatic $wu$-dipath, then   $D=T(D_1, \dots, D_n;C)$ such that $uv$ and  $vw$ are in $A(C)$.  
\end{lemma}

\begin{proof}
    Assume that in all $k$-dicolourings of $D \setminus \{uv, vw\}$ there is a monochromatic $wu$-dipath. 
    Assume that the result holds for every digraph in $\htk$ with strictly less vertices than $D$.

    Assume first that $D$ is a symmetric complete graph or a symmetric odd wheel.
    Let $\varphi$ be a $k$-dicolouring of $D - uv$ (which must exist for $D$ is dicritical). Then $w$ has a colour distinct from the colours of each of its neighbours with respect to $\varphi$, and thus $\varphi$ is a $k$-dicolouring of $D - \{uv,vw\}$ in which there is no monochromatic $wu$-dipath. So we may assume that $D$  is not a symmetric complete graph nor a symmetric odd wheel. 
 
    Thus, $D=T(D_1, \dots, D_n;C)$ and let $\{u_i,v_i\} = V(D_i) \cap V(T)$ for $i=1, \dots, n$. 
    Assume for contradiction that, $uv \notin E(C)$ or $vw \notin E(C)$.

     \noindent\textbf{Case 1:} $u \in V(D_i) \setminus \{u_i,v_i\}$ for some $i \in[n]$, $v=v_i$, and $w \notin V(D_i) \setminus \{u_i\}$. \\ 
     In this case, we will find a $k$-dicolouring of $D\setminus \{uv,vw\}$ with no monochromatic $wu$-dipath, thus obtaining a contradiction. 
    Since $D$ is $(k+1)$-dicritical, there is a $k$-dicolouring $\varphi$  of $D \setminus uv_i$. 
    Since $\varphi$ is, in particular, a $k$-dicolouring of $D\setminus \{uv_i,v_iw\}$, it contains a monochromatic $wu$-dipath $P$ by hypothesis. Hence $\varphi(u) = \varphi(w)$. Since $w \notin V(D_i)\setminus \{u_i\}$, $P$ goes through $u_i$ or $v_i$. 
    
    Observe first that $\varphi(u) = \varphi(v_i)$, for otherwise $\varphi$ is a a $k$-dicolouring of $D$, a contradiction. 
    If $P$ goes through $v_i$, the arc $v_iw$ yields a monochromatic directed cycle, a contradiction. Hence $P$ foes through $u_i$. 
    Hence, $\varphi(u_i) = \varphi(u) = \varphi(v_i)$. This implies that all junction vertices receive the same colour, and thus the peripheral cycle is monochromatic, a contradiction.  This proves case 1. 
    \medskip 

    \noindent\textbf{Case 2:} $u,v,w \in V(D_i)$ for some $i \in [n]$,   $uv \notin A(C)$ and $vw \notin A(C)$. \\
    Note that, since $uv,vw \notin A(C)$, $\{uv,vw\} \cap \{u_iv_i, v_iu_i\}=\emptyset$.

    Suppose first that $D_i = T'\big(D'_1, \dots, D'_m;C'\big)$ with $uv \in A(C')$ and $vw \in A(C')$. As $u \neq w$,  $|V(C)| \geq 3$, and thus there exists $j$ such that $[u_i, v_i] \in D'_j$. Then:
    
    $$D = T'\big(D'_1, \dots, D'_{j-1}, T\big(D_1, \dots, D_{i-1}, D'_{j}, D_{i+1}, \dots, D_n;C\big)  , D'_{j+1}, \dots, D'_m;C'\big)$$
    and we are done.  

    So $uv$ or $vw$ is not in $A(C)$. 
    Hence, by induction, there is a $k$-dicolouring $\varphi_i$ of $D_i\setminus \{uv,vw\}$ such that there is no monochromatic $wu$-dipath in $D_i$. Observe that $\varphi_i(u_i) \neq \varphi_i(v_i)$ since $u_i$ and $v_i$ are linked by a digon in  $D_i\setminus \{uv,vw\}$. 

    Now, let $\varphi$ be a $k$-dicolouring of $D \setminus \big(V(D_i) \setminus \{u_i,v_i\}\big)$.
    If $\varphi(u_i) = \varphi(v_i)$, then all junction vertices receive this same colour, and $C$ is monochromatic, a contradiction. So $\varphi(u_i) \neq \varphi(v_i)$. Now, we may assume without loss of generality that $\varphi(u_i) = \varphi_i(u_i)$ and $\varphi(v_i) = \varphi_i(v_i)$, and obtain a $k$-dicolouring of $D \setminus \{uv,vw\}$ with no monochromatic $wu$-dipath. Indeed, a $wu$-dipath is either included in $D_i-[u_i,v_i]$, or contains both $u_i$ and $v_i$.  \smallskip

    Let us now explain why these two cases cover all possible cases. 
    Since $uv \notin E(C)$ or $vw \notin E(C)$, we may assume that at least one vertex of $\{u,v,w\}$ is not a junction vertex, for an arc linking two junction vertices is an arc of $C$. 

    \begin{itemize}
        \item If none of $\{u,v,w\}$ is a junction vertex, we are in case 2. 
        \item  If $v$ is a junction vertex and $u$ is not. Then $u \in V(D_i)\setminus \{u_i,v_i\}$ for some $i \in [n]$. Then either $w \notin V(D_i) \setminus \{u_i\}$, and we are in case 1, or $w \in V(D_i) \setminus u_i$, and we are in case 2. 
        \item By directional duality, the previous case is the same as the case where $v$ is a junction vertex and $w$ is not. 
        \item If $v$ is not a junction vertex, then $v \in V(D_i) \setminus \{u_i,v_i\}$ for some $i \in [n]$, and thus $u,w \in V(D_i)$, and we are in case 2.  
    \end{itemize}
    
      \end{proof}

From the previous lemma, we deduce an analogous of Lemma~\ref{lem:HJ_preserve_htk} for Haj\'os bijoins.

\begin{lemma}\label{lem:HBJ_preservehtk}
   Let $k\geq 3$. If $D$ is not $k$-dicolourable and $D$ is the Haj\'os bijoin of two digraphs $D_1\in \htk$ and $D_2\in \htk$, then $D\in \htk$. 
\end{lemma}

 \begin{proof}
        Let $ta_1, a_1w \in A(D_1)$, and $t$ and $w$ are in the same connected component of $D_1 \setminus a_1$. Let $va_2, a_2u \in A(D_2)$,  and $u$ and $v$ are in the same connected component of $D_2 \setminus a_2$.
        Assume that $D$ is obtained from the disjoint union of $D_1 - \{ta_1, a_1w\}$ and $D_2 - \{va_2,a_2u\}$ by identifying $a_1$ and $a_2$ into a new vertex $a$, and adding the arcs $tu$ and $vw$, i.e. $D$ is the bijoin of $D_1$ and $D_2$ with respect to $((t,a_1, w),(u, a_2, v))$. 
               
        Suppose first that $t\neq w$ and $D_1 -\{ta_1, a_1w\}$ admits a $k$-dicolouring $\varphi_1$ with no monochromatic $wt$-dipath.
        Then either there is no monochromatic $wa_1$-dipath, or no monochromatic $a_1t$-dipath. Without loss of generality, suppose there is no monochromatic $wa_1$-dipath.
        As $D_2$ is $(k+1)$-dicritical, $D_2 - va_2$ is $k$-dicolourable, and thus there is a $k$-dicolouring $\varphi_2$ of $D_2- \{va_2, a_2u\}$ with no monochromatic $ua_2$-dipath. Up to permuting colours, we may assume that $\varphi_1(a_1) = \varphi_2(a_2)$. 
        Consider $\varphi:V(D) \rightarrow [k]$ such that           
        $$
        \varphi(x) = \left\{
            \begin{array}{ll}
                \varphi_1(a_1) & \text{ if } x=a \\
                \varphi_1(x)   & \text{ if } x \in V(D_1)\\ 
                \varphi_2(x)   & \text{ if } x \in V(D_2)
            \end{array}
        \right.
        $$  
        Since $\dic(D) \geq k+1$, $\varphi$ contains a monochromatic dicycle $C$. Since $\varphi_1$ and $\varphi_2$ are $k$-dicolourings of respectively $D[V(D_1)]$ and $D[V(D_2)]$, $C$ intersects both $V(D_1) \setminus a_1$ and $V(D_2) \setminus a_2$. Thus it contains $tu$, or $vw$ or both. \\ 
        If $C$ contains $tu$ but not $vw$, then $C$ goes through $a$, and thus there is a monochromatic $ua_2$-dipath in $D_2$, which together with $a_2u$ forms a monochromatic dicycle in $D_2$ with respect to $\varphi_2$ a contradiction to the choice of $\varphi_2$. \\ 
        If $C$ contains $vw$ but not $tu$, then there is a monochromatic $wa$-dipath in $D_1$ with respect to $\varphi_1$, a contradiction to the choice of $\varphi_1$. \\ 
        %Similarly, we get a contradiction if $C$ contains $vw$ but not $tu$. 
        Hence $C$ contains both $tu$ and $vw$ and thus there is a monochromatic $wt$-dipath in $D_1$, a contradiction to the choice of $\varphi_1$.

        Hence, if $t\neq w$, then all $k$-dicolourings of $D_1-\{ta_1,a_1w\}$ admit a monochromatic $wt$-dipath. Similarly, if $u\neq v$ then all $k$-dicolourings of $D_2-\{va_2,a_2u\}$ admit a monochromatic $uv$-dipath.
        
        If $t=w$, then in fact we can write $D_1=T^1(D^1,[ta1])$, where $T^1$ is simply the tree consisting of the edge $ta_1$. Otherwise by Lemma~\ref{lem:P3_in_peripheral}, $D_1=T^1(D^1_1, \dots, D^1_n;C^1)$ for a tree $T^1$,   digraphs $D^1_1, \dots, D^1_n$ and peripheral cycle $C^1$ such that $ta_1$ and $a_1w$ are in $A(C^1)$. Similarly, either $u=v$ in which case we write $D_2=T^2(D_2,[a2u])$ with $T_2$ being the tree consisting of the single edge $a_2u$, or  $D_2=T^2(D^2_1, \dots, D^2_m;C^2)$ for a tree $T^2$,   digraphs $D^2_1, \dots, D^2_m$ and peripheral cycle $C^2$ such that $va_2$ and $a_2u$ are in $A(C^2)$. 
        In all cases let $T$ be the tree obtained from $T_1$ and $T_2$ by identifying $a_1$ and $a_2$ to a vertex $a$. Now, $D = T(D^1_1, \dots, D^1_n, D^2_1, \dots, D^2_m;C)$ where $C$ is obtained from $C_1$ and $C_2$ after deleting arcs $ta_1$, $a_1w$, $va_2$ and $a_2w$, and adding $tu$ and $vw$. %This is a contradiction to claim~\ref{clm:HTJ}.
        %But then, $D = T'\big((D'_1, [u_1, v_1]), \dots, (D'_n,[u_n, v_n]), (D_1,[a,b]); C'\big)$  where $T'$ is the tree obtained from $T$ by adding a leaf adjacent (namely $a$) to $v$, and $C'$ is obtained from $C$ by replacing  $u \rightarrow a \rightarrow w$ by $u \rightarrow v \rightarrow w$. This is a contradiction to claim~\ref{clm:HTJ}. 

    \end{proof}
    
For the proof of our main theorem, we need to prove that digraphs in $\hk$ are indeed $k$-extremal. We know directed Haj\'os joins preserve $k$-extremality (Lemma~\ref{lem:extremal_directed_Haj}), we do it now for Haj\'os tree joins (we prove an if and only if for the purpose of the recognition algorithm of the next section)

  %------------------------------------------------------------------------------------------------------------------------------%

\begin{lemma}\label{lem:HT_iff_ext}
 Let $k \geq 2$. Let $D, D_1, \dots, D_n$ be digraphs such that $D$ is a Haj\'os tree join of the $D_i$. Then $D$ is $k$-extremal if and only if all digraphs $D_1, \dots, D_n$ are $k$-extremal. 
\end{lemma}
 \begin{proof}
Let $D=T(D_1, \dots, D_n;C)$ where $T$, $C$, $D_1, \dots, D_n$ are as in Definition~\ref{def:HTJ}. For each $D_i$, let  $\{u_i,v_i\} = V(T)\cap V(D_i)$ such that the digon $[u_i,v_i]$ is in $A(D_i)$ but was removed in the construction of $D$. 

Let  $D'$ the digraph obtained from $D$ by putting back all digons $[u_i,v_i]$ between vertices of $T$, and by removing the arcs in the peripheral cycle $C$. $D'$ is a digraph whose blocks are exactly the $D_i$. One can easily observe that $\lambda(D')=\max_{i=1}^n \lambda(D_i)$ and $\dic(D')=\max_{i=1}^n \dic(D_i)$. For every arc $uv\in A(C)$, let $P_{uv}$ be the unique $uv$-dipath of $D'$ that uses only arcs between vertices of $T$ (arcs from the digons that were removed to construct $D$). It is easy to notice that all $P_{uv}$ are pairwise arc-disjoint (each cycle $uv+P_{uv}$ correspond to one face of the planar graph $T+C$). Therefore one can go from $D'$ to $D$ by applying successive operations where one replaces the $uv$-dipath $P_{uv}$ by the arc $uv$ for each arc $uv\in A(C)$. By Lemma~\ref{lem:lambda_diminishing}, we  obtain 
$\lambda(D) \leq \lambda(D')=\max_{i=1}^n \lambda(D_i)$. 
\smallskip 

Assume $D$ is $k$-dicolourable. Then in any $k$-dicolouring of $D$, the vertices of $T$ do not all get the same colour (otherwise $C$ would me monochromatic). So there is a digraph $D_i$ such that the vertices $u_i$ and $v_i$ get distinct colours. But this provides a proper $k$-dicolouring of the corresponding $D_i$. Hence, if $D$ is $k$-dicolourable, then  
$\min_{i=1}^n \dic(D_i) \leq k$.
\smallskip 

With the two previous paragraphs, we can already prove that if each $D_i$ is $k$-extremal, then $D$ is $k$-extremal. We have $\dic(D_i)=k+1=\lambda(D_i)+1$ for every $i$. So, by the first paragraph, $\lambda(D) \leq k$, and by the second paragraph, $k+1 \leq \dic(D)$. Hence $k+1 \leq \dic(D) \leq  \lambda(D) +1 \leq k+1$, so $\dic(D) = \lambda(D) + 1 = k+1$.  
If $D$ admits a cutvertex, then by construction of the Haj\'os tree join, a block is included in some $D_i$, which contradicts the fact that $D_i$ is biconnected. Now observe that since every $D_i$ is Eulerian (for they are $k$-extremal), $D$ is also Eulerian and is thus strong.\smallskip 

Now assume $D$ is $k$-extremal. 
If for some $i \in [n]$, $D_i$ is not connected (resp. has a cutvertex), then at least one connected component (resp. block) is disjoint from $\{u_i,v_i\}$ (recall that $[u_i,v_i] \subseteq A(D_i)$), and we get that $D$ is also not connected (resp. has a cutvertex); a contradiction. So each $D_i$ is biconnected. 

Since $D$ is Eulerian, every $D_i$ is also Eulerian and thus every $D_i$ is strong.

Let $i \in [n]$ and let  us prove that $\lambda(D_i)=k$. 
Let $u,v \in V(D_i)$, and set $\lambda_{D_i}(u,v) = p$, i.e. there is $p$ arc-disjoint $uv$-dipaths in $D_i$. Then either these dipaths do not use any arc in the digon $[u_i,v_i]$, in which case these dipaths are still present in $D$, or they do use one of the arcs $\{u_iv_i, v_iu_i\}$, but then we can assume they don't use both, for otherwise we could reroute the dipaths to obtain a collection of $uv$-dipaths that do not use any arc in the digon $[u_i,v_i]$. %So assume without loss of generality that these paths use $u_iv_i$. 
But then we can replace this arc by a $u_iv_i$-dipath or $v_iu_i$-dipath using only arcs in some $D_j$ for $j\neq i$ and some peripheral arcs of $C$. Hence we still get $p$ pairwise arc disjoint $uv$-dipaths. Therefore $\max_{i=1}^n \lambda(D_i) \leq \lambda(D)$, so $\lambda(D_i)\leq k$ for every $i$.

In order to conclude that every $D_i$ is extremal, and since $\dic$ is always at most  $1+\lambda$ for any digraph, we only need to prove that $\dic(D_i)>k$ for every $i$. 

%If $\dic(D_i) = k+1$ for some $i$, then $k+1 \leq \dic(D) \leq \lambda(D) + 1 \leq k+1$ and thus $\dic(D_i) = \lambda(D_i) = k+1$. So it is enough to prove that, for every $i$, $\dic(D_i) = k+1$. 

%Since $D$ is $k$-extremal, it is $(k+1)$-dicritical (by Lemma~\ref{lem:prop_k-extremal}), so digraph $D_i-[u_i,v_i]$ is $k$-dicolourable, and thus $\dic(D_i) \leq k+1$ for every $i$. 
%We first prove the following claim. 
\begin{claim}
    If $T$ is a tree and  $A'\subsetneq  A(T)$, there exists $\varphi : V(T)\to \{1,2,3\}$ such that each edge in $A'$ is monochromatic, no edge in $A(T) \setminus A(T')$ is monochromatic, and the leaves of $T$ do not all receive the same colour.
\end{claim}
\begin{proofclaim}
Let $uv$ be an edge not in $A'$ and consider the two connected components $T_u$ and $T_v$ of $T-uv$. If some connected component only contains edges in $A'$ we colour it with one single colour. If not we apply induction. Up to permuting the colours we can do so that $u$ and $v$ receive distinct colours. If we applied induction to either $T_u$ or $T_v$, then all leaves do not get the same colour, and if not it means all edges of $T$ except $uv$ was in $A'$, but in that case since the colour of $u$ is distinct from the colour of $v$, the leaves in $T_u$ and $T_v$ must get distinct colours.
\end{proofclaim}

Now, let $A'$ be the set of edges $u_iv_i$ of $T$ such that $\dic(D_i)=k+1$. If all edges are in $A'$, then we are done. So we may assume for contradiction that it is not the case. By the claim above, there exists $\varphi:V(T) \ra \{1, 2, 3\}$ such that each edge of $T'$ is monochromatic, no edge of $A(T) \setminus A'$ is monochromatic, and the leaves of $T$ do not all receive the same colour. 

If $u_iv_i\in A'$, then the digraph $D_i$ is $k$-extremal and we can apply Lemma~\ref{lem:extremdigon} to get a $k$-dicolouring $\varphi_i$ of $D_i-[u_i,v_i]$ chosen such that $\varphi_i(u_i) = \varphi_i(v_i) = \varphi(u_i)$ and  such that there are no monochromatic dipath between $u_i$ and $v_i$. 

If $u_iv_i\not\in A'$, then  $\dic(D_i)\leq k$ and we let $\varphi_i$ be a $k$-dicolouring of $D_i$. Since the digon $[u_i v_i]$ is in $D_i$ $u_i$ and $v_i$ get distinct colours and we can choose $\phi_i$ such that $\varphi_i(u_i) = \varphi(u_i) \neq \varphi(v_i) = \varphi_i(v_i)$. 

Altogether, we obtain a $k$-dicolouring of the vertices of $D$ that is proper on each $D_i-[u_i,v_i]$, and such that there is no monochromatic dipath between any pair of vertices of $T$, and such that $C$ is not monochromatic. This is a $k$-dicolouring of $D$, our final contradiction.

% \begin{proof} \PC{LAST THING TO FIX}
%  Let $D=T(D_1, \dots, D_n;C)$ where $T$, $C$, $D_1, \dots, D_n$ are as in definition~\ref{def:HTJ}. We proceed by induction on $n$, the number of edges of $T$.

% Let $u_i$ be a leaf of $T$ and $v_i$ its parent.
%  As $u_i$ is a leaf of $T$, it has an inneighbour $a$ and an outneighbour $b$ in $C$. Let $T'$ be the tree $T\setminus v_i$ and $D'=T'(D_1, \dots, D_{i-1}, D_{i+1}, \dots, D_n;C - \{au_i, u_ib\} + \{av_i, v_i,b\})$. $D$ is the Haj\'os bijoin of $D_i$ and $D'$ (note that it is possible that $a=b$ in which case $D$ is a bidirectional Haj\'os join , which is a particular case of Haj\'os bijoins).
 
%  Now by Lemma~\ref{lem:extremal_bijoin_Haj}, if $D$ is $k$-extremal then $D_i$ and $D'$ are extremal, and by induction applied to $D'$, we get that all digraphs $D_i$ are extremal.
 
%  Conversely if we assume all $D_i$ are $k$-extremal, then by induction $D'$ is $k$-extremal. Also, by Lemma~\ref{lem:extremal_bijoin_Haj_only_if}, $D$ is biconnected, strongly connected, and $\lambda(D) \leq k$. We therefore just need to prove that $D$ is not $k$-colourable. Assume by contradiction that  $\varphi$ is a $k$-colouring of $D$. Since the peripheral cycle $C$ is not monochromatic, the vertices of $T$ cannot all get the same colour. thus we can find $u_iv_i \in E(T)$ such that $\varphi(u_i) \neq \varphi(v_i)$. But this is impossible because this would induce a proper $k$-colouring of  $D_i$, which is $k$-extremal. 

\end{proof}

%------------------------------------------------------------------------------------------------------------------------------%
%------------------------------------------------------------------------------------------------------------------------------%
We are now ready to prove our main theorem.
\begin{theorem}~\label{thm:struct} 
Let $k \geq 3$ and let $D$ be a digraph. The three following statements are equivalent: 
\begin{description}
\item[(i)] $D$ is $k$-extremal
\item[(ii)] $D\in \htk$
\item[(iii)] $D \in \hk$
\end{description}
\end{theorem}

\begin{proof}
If $D$ is a symmetric complete digraph of a symmetric odd wheel, then the result holds, so assume it is not. We prove the statements by induction on the number of vertices of $D$.

\begin{description}
\item[(i)$\Ra$ (ii)]

Assume $D$ is $k$-extremal. By Theorem~\ref{thm:decHJHBJ}, $D$ is either a directed Haj\'os join or a Haj\'os bijoin of two $k$-extremal digraphs.

Assume first that $D$ is the directed Haj\'os join of two digraphs $D_1$ and $D_2$. By Lemma~\ref{lem:extremal_directed_Haj}, both $D_1$ and $D_2$ are $k$-extremal. Thus by induction $D_1$ and $D_2$ belong to $\htk$ and since directed Haj\'os joins preserve the fact of being in $\htk$  (Lemma~\ref{lem:HJ_preserve_htk}), $D$ is in $\htk$.

So we can assume that $D$  is a Haj\'os bijoin of two $k$-extremal digraphs $D_1$ and $D_2$. By Lemma~\ref{lem:extremal_bijoin_Haj}, both $D_1$ and $D_2$ are $k$-extremal. So by induction hypothesis they both belong to $\htk$. By Lemma~\ref{lem:HBJ_preservehtk}, $D$ is in $\htk$. 

%But now by Lemma~\ref{l}, and since $D$ is not a directed Haj\'os join we know that $D$ is a Haj\'os tree of a collection of digraphs $D_i$ which are in $\htk$ so are in $\hk$-extremal by induction hypothesis. But then $D$ is in $\hk$.

\item[(ii)$\Ra$ (iii)] By Lemma~\ref{lem:HTK=HJorHT} a digraph $D$ in $\htk$ is either a directed Haj\'os join or a Haj\'os tree join of digraphs in $\htk$  and thus in $\hk$ by induction. Hence $D$ is in $\hk$.

\item[(iii)$\Ra$ (i)] This is guaranteed by the fact that both directed Haj\'os joins and Haj\'os tree joins preserve the fact of being $k$-extremal (Lemmata~\ref{lem:extremal_directed_Haj} and~\ref{lem:HT_iff_ext}).
\end{description}
\end{proof}

\section{Recognition algorithm}\label{sec:algo}

In this section, we give a polynomial time algorithm deciding if a given digraph $D$ satisfies $\dic(D) = \lambda(D) + 1$. For algorithmic reasons, we need to avoid Haj\'os tree joins, so we need to devise another characterization, using the notion of parallel Haj\'os joins that we define now.

%-----------------------------------------

%\subsection{Parallel Haj\'os join}

%-----------------------------------------------

\begin{definition}[Parallel Haj\'os join]\label{def:parallel_join}
    Let $D_B$ be a digraph, set $B=V(D_B)$ and let $[a,b] \subseteq A(D_{B})$. \\  
    Let $D_{AC}$ be a digraph with $V(D_{AC}) = A \cup C$, $A \cap C= \{x\}$, let $t,w \in A \setminus x$ such that $t,w$ are in the same connected component of $D_{AC} \setminus x$, and let $u,v \in V(C)$ such that $u$ and $v$ are in the same connected component of $D_{AC}[C] \setminus x$.\\ 
    The \emph{parallel Haj\'os join} $D$ of $D_{AC}$ and $D_B$ with respect to $(t,u,v,w, [a,b])$ is the digraph obtained from disjoint copies of $D_B-[a,b]$, $D_{AC}[A]$ and $D_{AC}[C]$, by identifying the copy of $x$ in $D_{AC}[A]$ to $a$, and the copy of $x$ in $D_{AC}[C]$ to $b$. 
    See Figure~\ref{fig:parallel_join}.
\end{definition}

    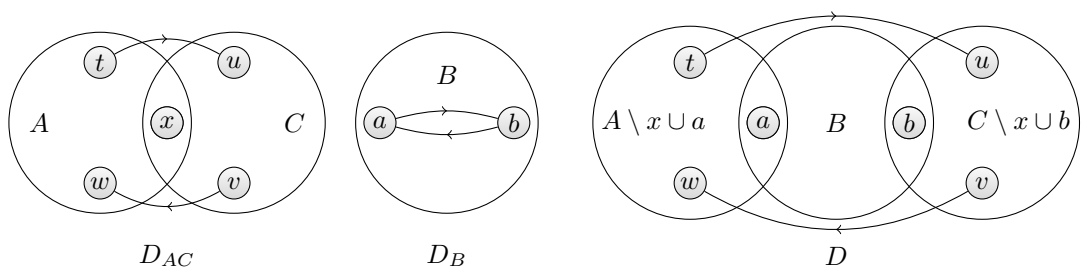
\begin{figure}[!hbtp]
    \begin{center}
        \begin{tikzpicture}[scale=0.8]

            \begin{scope}[xshift=0cm]
                \vertex (r) at (1.2,0) {$a$};
                \vertex (u) at (0,1) {$t$};
                \vertex (d) at (0,-1) {$w$};
                \draw (0,0) circle (1.6);
                \node () at (-0.6,0) {$A \setminus x \cup a$};
            \end{scope}

            \begin{scope}[xshift=2.4cm]
                \vertex (lc) at (1.2,0) {$b$};
                \vertex (rc) at (-1.2,0) {$a$};
                \draw (0,0) circle (1.6);
                \node () at (0,0) {$B$};
                \node () at (0,-2.2) {$D$};
            \end{scope}
        
            \begin{scope}[xshift=4.8cm,xscale=-1]
                \vertex (r2) at (1.2,0) {$b$};
                \vertex (u2) at (0,1) {$u$};
                \vertex (d2) at (0,-1) {$v$};
                \node () at (-0.6,0) {$C\setminus x \cup b$};
                \draw (0,0) circle (1.6);
            \end{scope}

            \draw[->-, bend left=30] (u) to (u2);
            \draw[->-, bend left=30] (d2) to (d);

            \begin{scope}[xshift=-9.7cm]
                \vertex (r) at (1.1,0) {$x$};
                \vertex (u) at (0,1) {$t$};
                \vertex (d) at (0,-1) {$w$};
                \node () at (-1,0) {$A$};
                \draw (0,0) circle (1.5);
            \end{scope}

            \begin{scope}[xshift=-7.5cm,xscale=-1]
                \vertex (r2) at (1.1,0) {$x$};
                \vertex (u2) at (0,1) {$u$};
                \vertex (d2) at (0,-1) {$v$};
                \node () at (-1,0) {$C$};
                \draw (0,0) circle (1.5);
            \end{scope}

            \draw[->-, bend left=30] (u) to (u2);
            \draw[->-, bend left=30] (d2) to (d);

            \node () at (-8.6,-2.2) {$D_{AC}$};

            \begin{scope}[xshift=-4cm]
                \vertex (lc) at (1.1,0) {$b$};
                \vertex (rc) at (-1.1,0) {$a$};
                \draw[->-, bend left=15] (lc) to (rc);
                \draw[->-, bend left=15] (rc) to (lc);
                \node () at (0,0.8) {$B$};
                \draw (0,0) circle (1.5);
                \node () at (0,-2.2) {$D_B$};
            \end{scope}

        \end{tikzpicture}
        \end{center}
           \caption{$D$ is a parallel Haj\'os join of $D_{AC}$ and $D_{B}$ with respect to $(t,u,v,w)$.}
           \label{fig:parallel_join}
        \end{figure}

Let us say an informal word on the intuition behind parallel Haj\'os join. 
Let $D=T(D_1, \dots, D_n;C)$ be a Haj\'os tree join and assume $u_iv_i \in E(T)$ is such that  both $u_i$ and $v_i$ are interior vertices of $T$. Then $D$ is the Haj\'os parallel join of $D_{i}$ and the Haj\'os tree join obtained after contracting $D_i$. 

%Let $D_1=T_1(A_1,\dots, A_k; C_1)$ and $D_2=T_2(B_1,\dots, B_{\ell}; C_2)$ be two directed Haj\'os tree with $C_1 = (x_1, \dots, x_a)$ and $C_2 = (y_1, \dots, y_{b})$. Then the digraph obtained from  disjoint copies of $D_1-\{x_kx_1, x_1x_2\}$ and $D_2-\{y_{\ell}y_1, y_1y_2\}$ by identifying $x_1$ and $y_1$ into a new vertex $x*y$ and adding the arcs $y_{\ell}x_2$ and $x_ky_2$ is a Haj\'os tree join, and is also the bijoin of $D_1$ and $D_2$. 

As before, we need to prove that this operation preserves extremality.

\begin{lemma}\label{lem:extremal_parallel_Haj}
    Let $k \geq 3$. A parallel Haj\'os join of two digraphs $D_{AC}$ and $D_C$ is $k$-extremal if and only if both $D_{AC}$ and $D_{B}$ are $k$-extremal.
\end{lemma}

\begin{proof} 
    Let $D$ be the parallel Haj\'os join of $D_{AC}$ and $D_{B}$ as in Definition~\ref{def:parallel_join}. 
    
    Let $D_A = D[A] + \{ta,aw\}$, $D_{C} = D[C] + \{vb,bu\}$ and $D_{BC} = D[V(B) \cup V(C)] + \{va,au\}$. See Figure~\ref{fig:DA_DB_DC_DBC}.

    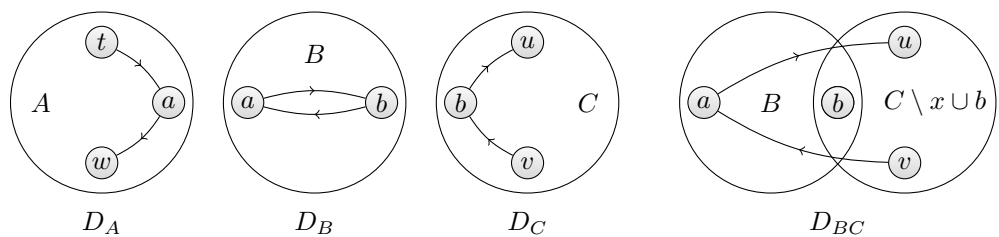
\begin{figure}[!hbtp]
    \begin{center}
        \begin{tikzpicture}[scale=0.8]
            \begin{scope}[xshift=-11cm]
                \vertex (r) at (1.1,0) {$a$};
                \vertex (u) at (0,1) {$t$};
                \vertex (d) at (0,-1) {$w$};
                \draw[->-, bend left=15] (u) to (r);
                \draw[->-, bend left=15] (r) to (d);
                \node () at (-1,0) {$A$};
                \draw (0,0) circle (1.5);
                \node () at (0,-2) {$D_A$};
            \end{scope}

            \begin{scope}[xshift=-7.5cm]
                \vertex (lc) at (1.1,0) {$b$};
                \vertex (rc) at (-1.1,0) {$a$};
                \draw[->-, bend left=15] (lc) to (rc);
                \draw[->-, bend left=15] (rc) to (lc);
                
                \node () at (0,0.8) {$B$};
                \draw (0,0) circle (1.5);
                \node () at (0,-2) {$D_B$};
            \end{scope}
        
            \begin{scope}[xshift=-4cm,xscale=-1]
                \vertex (r2) at (1.1,0) {$b$};
                \vertex (u2) at (0,1) {$u$};
                \vertex (d2) at (0,-1) {$v$};
                \draw[->-, bend right=15] (d2) to (r2);
                \draw[->-, bend right=15] (r2) to (u2);
                \draw (0,0) circle (1.5);
                \node () at (-1,0) {$C$};
                \node () at (0,-2) {$D_C$};
            \end{scope}

            \begin{scope}[xshift=0cm]
                \vertex (ll) at (1.1,0) {$b$};
                \vertex (rr) at (-1.1,0) {$a$};
                \node () at (0,0) {$B$};
                \draw (0,0) circle (1.5);
            \end{scope}
        
            \begin{scope}[xshift=2.2cm,xscale=-1]
                \vertex (r2) at (1.1,0) {$b$};
                \vertex (u2) at (0,1) {$u$};
                \vertex (d2) at (0,-1) {$v$};
                
                \node () at (-0.5,0) {$C\setminus x \cup b$};
                \draw (0,0) circle (1.5);
            \end{scope}

            \draw[->-, bend left=15] (d2) to (rr);
            \draw[->-, bend left=15] (rr) to (u2);

            \node () at (1.1,-2) {$D_{BC}$};

        \end{tikzpicture}
        \end{center}
           \caption{$D_A$, $D_B$, $D_C$ and $D_{BC}$ in the proof of Lemma~\ref{lem:extremal_parallel_Haj}.}\label{fig:DA_DB_DC_DBC}
        \end{figure}

    Observe that 
    \begin{itemize}
        \item $D$ is a degenerated Haj\'os bijoin of $D_{A}$ and $D_{BC}$, 
        \item $D_{BC}$ is a degenerated Haj\'os bijoin of $D_{B}$ and $D_{C}$, and
        \item $D_{AC}$ is a Haj\'os bijoin of $D_{A}$ and $D_{C}$.  
    \end{itemize}
   
    Suppose first that $D$ is $k$-extremal ane let us prove that $D_{AC}$ and $D_B$ are $k$-extremal. 

    Since $D$ is a Haj\'os bijoin of $D_{A}$ and $D_{BC}$, and $D_{BC}$ is a degenerated Haj\'os bijoin of $D_{B}$ and $D_{C}$, we get that $D_A$ and $D_B$ and $D_C$ are $k$-extremal by Lemma~\ref{lem:extremal_bijoin_Haj}. 
    So it remains to prove that $D_{AC}$ is $k$-extremal. 
    
    Since $D_A$ and $D_C$ are  $k$-extremal, and  $D_{AC}$ is the Haj\'os bijoin of $D_A$ and $D_C$, by Lemma~\ref{lem:extremal_bijoin_Haj_only_if} we have that $D_{AC}$ is strong, biconnected and $\lambda(D_{AC}) \leq k$.

    Let us now prove that $\dic(D_{AC}) \geq k+1$. Suppose $D_{AC}$ admits a $k$-dicolouring $\varphi_{AC}$. Since $D_B$ is $k$-extremal, $D_B - [a,b]$ admits a $k$-dicolouring $\varphi_B$ with $\varphi_B(a) = \varphi_B(b)$. 
    Up to permuting colours, we may assume that $\varphi_B(a) = \varphi_{AC}(x)$. Let $\varphi: V(D) \to [1,k]$ be such that $\varphi(y) = \varphi_{B}(y)$ if $y \in B$ and $\varphi(y) = \varphi_{AC}(y)$ if $y \in A \cup C \setminus x$. As any dicycle of $D$  is either included in $D_B$, or contains vertices that form a dicycle in $D/B = D_{AC}$, $\varphi$ is a $k$-dicolouring of $D$, a contradiction. 

    Thus $k+1 \leq \dic(D_{AC}) \leq \lambda(D_{AC}) + 1 \leq k+1$. Hence $\dic(D_{AC}) = \lambda(D_{AC})=k+1$, which ends the proof that $D_{AC}$ is $k$-extremal. 
    \medskip  

    Suppose now that $D_{AC}$ and $D_{B}$ are $k$-extremal and let us prove that $D$ is $k$-extremal. 

    Since $D_{AC}$ is $k$-extremal and is a Haj\'os bijoin of $D_A$ and $D_C$, both $D_A$ and $D_C$ are $k$-extremal by Lemma~\ref{lem:extremal_bijoin_Haj}. 
    
    Since $D_{BC}$ is the Haj\'os bijoin of $D_B$ and $D_C$ and $D_B$ and $D_C$ are both $k$-extremal, $D_{BC}$ is biconnected, strong, Eulerian and $\lambda(D_{BC}) \leq k$  by Lemma~\ref{lem:extremal_bijoin_Haj_only_if}. 
    Finally, since $D$ is the Haj\'os bijoin of $D_A$ and $D_{BC}$, by Lemma~\ref{lem:extremal_bijoin_Haj_only_if}, $D$ is biconnected, strong, Eulerian and $\lambda(D) \leq k$. 
    
    Let us now prove that $\dic(D) \geq k + 1$. Suppose that $D$ admits a $k$-dicolouring $\varphi_D$.
    Then, as $D[B] = D_{B} - [a,b]$, and because $D_B$ is $k$-extremal, $\varphi_D(a) = \varphi_D(b)$.  
    We are going to split the proof into two cases, in each case we prove that $\dic(D_{AC}) \leq k$, a contradiction. 
    
    \smallskip 
    
    \noindent\textbf{Case 1:} $\varphi_D(t) \neq \varphi_D(u)$

    Since $\varphi_D(a) = \varphi_D(b)$, either $\varphi_D(t) \neq \varphi_D(a)$ or $\varphi_D(u) \neq \varphi_D(b)$. Suppose without loss of generality that $\varphi_D(t) \neq \varphi_D(a)$. 
    Let $\varphi_C$ be a $k$-dicolouring of $D_{C} - bu = D[C] +vb$ and, up to permuting colours, assume  that $\varphi_C(b) = \varphi_D(b) (= \varphi_D(a))$. There is no monochromatic $bv$-dipath with respect to  $\varphi_{C}$ as $vb \in A(D_C - bu)$. Also, $\varphi_{C}(b) = \varphi_{C}(u)$ for $D_C$ is $k$-extremal and thus $k$-dicritical. Hence  $\varphi_{C}(u) \neq \varphi_D(t)$. 
    
    Now, let $\varphi_{AC} : V(D_{AC}) \to [1,k]$ be such that 
    
$$
\varphi_{AC}(y) = \left\{
    \begin{array}{ll}
        \varphi_D(a) & \text{ if } y=x \\
        \varphi_{C}(y) & \text{ if } y \in C \setminus b \\
        \varphi_D(y) &  \text{ if } y \in A \setminus a
    \end{array}
\right.
$$

    %\begin{itemize}
    %    \item $\varphi_{AC}(x) = \varphi_D(a)$ ($=\varphi_D(b) = \varphi_C(b)$),
    %    \item $\varphi_{AC}(y) = \varphi_{C}(y)$ if $y \in C \setminus b $, and
    %    \item $\varphi_{AC}(y) = \varphi_D(y)$ if $y \in A \setminus a$. 
    %\end{itemize}

    Observe that, since there is no monochromatic $bv$-dipath with respect to $\varphi_C$, there is no monochromatic $xv$-dipath with respect to $\varphi_{AC}$. 
    Since any dicycle of $D_{AC}$ is either included in $A \cup x$ or in $C \cup x$,  or goes through $tu$, or contains an  $xv$-dipath, $\varphi_{AC}$ is a $k$-dicolouring of $D_{AC}$, a contradiction.

    \smallskip

    \noindent\textbf{Case 2:} $\varphi_D(t) = \varphi_D(u)$. 

    %Observe that $D_{AC} = D/B$, and since $\varphi_D(a) = \varphi_D(b)$, $\varphi_D$ 
    
    %If $\varphi_D(w) \neq \varphi_D(v)$ the we are done by case 1. So we may assume that $\varphi_D(w) = \varphi_D(v)$. Moreover, if $\varphi_D(t) \neq \varphi_D(w)$, then any 
    
    There is not both a monochromatic $wt$-dipath and a monochromatic $uv$-dipath with respect to $\varphi_D$. Without loss of generality, suppose there is no monochromatic $uv$-dipath. Then, either there is no monochromatic $ub$-dipath or no monochromatic $bv$-dipath.
    Suppose without loss of generality that there is no monochromatic $bv$-dipath. Since $D_{A}$ is $k$-extremal, $D_{A} - aw$ admits a $k$-dicolouring $\varphi_{A}$. 
    Up to permuting colours, we may assume that $\varphi_{A}(a) = \varphi_D(a)$. 
    %Since $D_{A}$ is $k$-extremal, $\varphi_A(w) = \varphi_A(a)$ and, up to permuting colours, we may assume that $(\varphi_A(t)=)\varphi_{A}(a) = \varphi_D(a)(=\varphi_D(t))$. 
    Note that $ta \in A(D_{A} - aw)$, and thus there is no monochromatic $at$-dipath in $\varphi_{A}(a)$.

    Let $\varphi_{AC} : V(D_{AC}) \to [1,k]$ be such that
    $$
\varphi_{AC}(y) = \left\{
    \begin{array}{ll}
         \varphi_D(a) & \text{ if } y=x \\
        \varphi_{A}(y) & \text{ if } y \in A \setminus a \\
        \varphi(y) &  \text{ if }y \in A \setminus a
    \end{array}
\right.
$$

    %\begin{itemize}
    %    \item $\varphi_{AC}(x) = \varphi_D(a)$,
    %    \item  $\varphi_{AC}(y) = \varphi_{A}(y)$ if $y \in A \setminus a$ and,
    %    \item $\varphi_{AC}(y) = \varphi(y)$ if $y \in C \setminus b$.
    %\end{itemize}
    %$\varphi_{AC}(y) = \varphi_{A}(y)$ if $y \in V(D_A)$ and $\varphi_{AC}(y) = \varphi(y)$ otherwise. 

    Observe that: 
    \begin{itemize}
        \item since there is no monochromatic $at$-dipath with respect to $\varphi_{A}$, there is no monochromatic $xt$-dipath with respect to $\varphi_{AD}$, 
        \item since there is no monochromatic $uv$-dipath with respect to $\varphi_D$, there is no monochromatic $uv$-dipath with respect to $\varphi_{AC}$, and
        \item since there is no monochromatic $bv$-dipath with respect to $\varphi_{D}$, there is no monochromatic $xv$-dipath with respect to $\varphi_{AD}$. 
    \end{itemize}

    Finally, observe that dicycle of $D_{AC}$  is either included in $D_{AC}[A \cup x]$ or $D_{AC}[C \cup x]$, or contains a $uv$-dipath, or a $xv$-dipath or a $ww$-dipath. Hence, $\varphi_{AC}$ is a $k$-dicolouring of $D_{AC}$, a contradiction. 
    
    This proves that $\dic(D) \geq k+1$. We now have as usual $k+1 \leq \dic(D) \lambda(D) + 1 \leq k+1$, so $\dic(D) = \lambda(D) + 1 = k+1$, and since we already proved that $D$ is strong and biconnected, we get that $D$ is $k$-extremal. 

\end{proof}

The algorithm will use the following third decomposition theorem for $\hk$. Recall that Haj\'os star join is a Haj\'os tree join in which the tree is a star. See Figure~\ref{fig:cyclic_join}. 
\begin{theorem}\label{lem:htk_implies_other_joins}
    Let $k \geq 3$. If $D$ is $k$-extremal, then one of the following holds:
    \begin{itemize}
        \item either $D = \ovlra{K}_k$, 
        \item or $D$ is a symmetric odd wheel (only in the case $k=3$), 
        \item or $D$ is a directed Haj\'os join, 
        \item or $D$ is a parallel Haj\'os join, 
        \item or $D$ is a Haj\'os star join. 
    \end{itemize}
\end{theorem}
 \begin{figure}[!hbtp]
    \begin{center}
        \begin{tikzpicture}[scale=0.5]

            \begin{scope}
                \vertex (lc) at (0,0) {$x$};
                \vertex (v1) at (5,0) {$v_1$};
                \draw (2.5,0) ellipse (3cm and 2cm) {};
            \end{scope}

            \begin{scope}[rotate = 120]
                \vertex (lc) at (0,0) {$x$};
                \vertex (v2) at (5,0) {$v_2$};
                \draw (2.5,0) ellipse (3cm and 2cm) {};
            \end{scope}

            \begin{scope}[rotate = -120]
                \vertex (lc) at (0,0) {$x$};
                \vertex (v3) at (5,0) {$v_3$};
                \draw (2.5,0) ellipse (3cm and 2cm) {};
            \end{scope}

            \draw[->-, bend right=60] (v1) to (v2);
            \draw[->-, bend right=60] (v2) to (v3);
            \draw[->-, bend right=60] (v3) to (v1);

            \node () at (0,-6) {$D$};

            \begin{scope}[xshift = -14cm]
                
                \begin{scope}[rotate=0]
                    \begin{scope}[xshift=1.5cm]
                        \vertex (l) at (0,0) {$x$};
                        \vertex (v) at (5,0) {$v_1$};
                        \draw[->-, bend left=15] (l) to (v);
                        \draw[->-, bend left=15] (v) to (l);
                        \draw (2.5,0) ellipse (3cm and 2cm) {};
                        %\node () at (-2,0) {$D_1$};
                    \end{scope}
                    
                \end{scope}
                
                \begin{scope}[rotate=120]
                    \begin{scope}[xshift=1.5cm]
                        \vertex (l) at (0,0) {$x$};
                        \vertex (v) at (5,0) {$v_2$};
                        \draw[->-, bend left=15] (l) to (v);
                        \draw[->-, bend left=15] (v) to (l);
                        \draw (2.5,0) ellipse (3cm and 2cm) {};
                        %\node () at (-2,0) {$D_2$};
                    \end{scope}
                \end{scope}
    
                \begin{scope}[rotate=-120]
                    \begin{scope}[xshift=1.5cm]
                        \vertex (l) at (0,0) {$x$};
                        \vertex (v) at (5,0) {$v_3$};
                        \draw[->-, bend left=15] (l) to (v);
                        \draw[->-, bend left=15] (v) to (l);
                        \draw (2.5,0) ellipse (3cm and 2cm) {};
                        %\node () at (-2,0) {$D_3$};
                    \end{scope}
                \end{scope}
            \end{scope}
           
        \end{tikzpicture}
        \end{center}
           \caption{$D$ is a Haj\'os star join of three digraphs.}
           %\PC{Je suis pas sur qu'on ait besoin de cette figure en plus}
           \label{fig:cyclic_join}
        \end{figure}
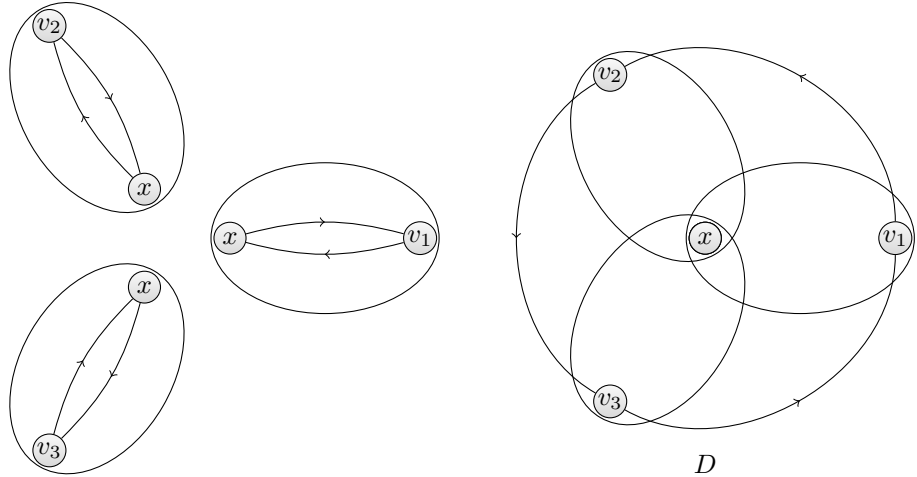
        
\begin{proof}
    Suppose $D \neq \ovlra{K}_k$ and $D$ is not a symmetric odd wheel. 
    By Theorem~\ref{thm:struct}, $D \in \hk$. 
    Since it is one of the possible outputs of this theorem, we can assume that $D$ is not a directed Haj\'os join. Thus  $D$ is a Haj\'os tree join:  there exists a tree $T$ with edges $\{u_1v_1, \dots, u_nv_n\}$, some  digraphs $D_1, \dots, D_n$ with $[u_i,v_i] \subseteq A(D_i)$ for $i=1, \dots, n$ such that $D=T(D_1, \dots, D_n;C)$, where $C$ is the peripheral cycle. 
    \smallskip

    If $T$ is a star, then $D$ is a Haj\'os star join of $D_1, \dots, D_n$ and we are done.
    \smallskip 
    
    Hence, there is $u_i,v_i \in E(T)$ such that $u_i$ and $v_i$ are both interior vertices of $T$.  
    Let $T_{u_i}$ and $T_{v_i}$ be the two connected component of $T-u_iv_i$ containing respectively $u_i$ and $v_i$. 
    Since $C$ goes through each leaves of $T$ following the natural ordering of an embedding of $T$, $C$ is the concatenation of two vertex disjoint dipaths $P_{u_i}$ and $P_{v_i}$ such that $V(P_{u_i})$ are the leaves of $T$ contained in $T_{u_i}$ and $V(P_{v_i})$ are the leaves of $T$ contained in $T_{v_i}$.
    There is $x_{u_i}, y_{v_i} \in V(P_{u_i})$ and $x_{v_i}, y_{v_i} \in V(P_{v_i})$ such that $x_{u_i}y_{v_i}, x_{v_i}y_{u_i} \in A(C)$. 

    Let $D'$ be the digraph obtained from $D$ by deleting $V(D_{i})\setminus \{u_i,v_i\}$ and identifying $u_i$ and $v_i$ to a new vertex $x$. 
    Then $D$ is the Haj\'os parallel join of $D'$ and $D_{ab}$ with respect to $(x_{u_i}, y_{v_i}, x_{v_i}, y_{u_i})$. To see this, look at Definition~\ref{def:parallel_join} and observe that: 
    \begin{itemize}
        \item $D_{i}$ plays the role of $D_B$,
        \item $D'$ plays the role of $D_{AC}$,
        \item $A=\bigcup_{u_jv_j \in T_{u_i}} V(D_{j})$ and $C=\cup_{u_jv_j \in T_{v_i}} V(D_{j})$,
        \item $x_{u_i}, y_{v_i}, x_{v_i}, y_{u_i}$ plays the role of respectively $t,u,v,w$
    \end{itemize} 
    Finally, observe that $x_{u_i}, y_{u_i}$ are in the same connected component of $D[A]\setminus x$ because of $P_{u_i}$ and $x_{v_i}, y_{v_i}$ are in the same connected component of $D[C] \setminus x$ because of $P_{v_i}$.

    %Then, $C = P_{ab}P_{cd}$ or $C = P_{cd}P_{ab}$ where elements of $P_{ab}$ are vertices of $T_{ab}$ and elements of $P_{cd}$ are vertices of $T_{cd}$.     Up to permuting $a$ with $d$ and $b$ with $c$, we can suppose that $C = P_{ab}P_{cd}$.     Let $x_{ab}$ be the first element of $P_{ab}$, $y_{ab}$ be its last element, $x_{cd}$ be the first element of $P_{cd}$ and $y_{cd}$ be its last element.     Thus, $D - \{y_{ab}x_{cd}, y_{cd}x_{ab}\}$ admits $2$ cutvertices, $b$ and $c$. Being vertices of $C$, $x_ab$ and $y_ab$ are in the same connected component of $D[\cup_{e_i \in E(T_{ab})} V(D_i)] \setminus b$. Similarly, $x_cd$ and $y_cd$ are in the same connected component of $D[\cup_{e_i \in E(T_{cd})} V(D_i)] \setminus c$.     Hence, $D$ is a parallel Haj\'os join.
\end{proof}

%\begin{lemma}
%There is an algorithm with the following specification:\\
%\underline{Input}: A digraph $D$ on $n$ vertices. \\
%\underline{Output}: No if $D$ is not the Haj\'os star of some digraphs, otherwise:\\
%\begin{itemize}
%    \item A star $T$ on edges 
%\end{itemize}
%\underline{Running time}:
%\end{lemma}

\begin{theorem}
    Let $k \geq 3$. There is an algorithm that decides if a given digraph $D$ is $k$-extremal in time $\mathcal O(n^{10})$. % can be solved in time $\mc{O}(n^{10})$.
\end{theorem}

\begin{proof} 
    Our algorithm is based on Theorem~\ref{lem:htk_implies_other_joins} together with Lemmata~\ref{lem:extremal_directed_Haj},~\ref{lem:extremal_parallel_Haj} and~\ref{lem:HT_iff_ext}
    
    Let $D$ be a digraph on $n$ vertices. Checking if $D$ is strong and biconnected can be done in time $\mc{O}(n^2)$. It takes time $\mc{O}(n^2)$ to check if $D = \ovlra{K}_k$ or $D$ is a symmetric odd wheel. If $D = \ovlra{K}_k$, then our algorithm outputs that $D$ is $k$-extremal. 
    We may now assume that $D$ is strong, biconnected and distinct from $\bid K_k$ and symmetric odd wheels. 

    \begin{claim}\label{clm:alg_DHJ}
        We can decide in time $\mathcal O(n^5)$ that either $D$ is not  the directed Haj\'os join of two digraphs, or $D$ is the directed Haj\'os join of two digraphs $D_1$ and $D_2$ and compute $D_1$ and $D_2$. 
    \end{claim}

    \begin{proofclaim}
        Checking if $D$ is a directed Haj\'os join of two digraphs $D_1$ and $D_2$ can be done by testing for all triples of vertices $(u, v, w)$ if $uw \in A(D)$, $D \setminus v - uw$ is not connected and $u$ and $w$ are in distinct components of $D \setminus v - uw$. 
        If $(u, v, w)$ is such a triple, Let $R_u$ (resp. $R_w$) be the connected component of $D \setminus v - uw$ containing $u$ (resp. containing $w$).%, and $R_w$ the union of the other connected components.
        Then $D$ is the directed Haj\'os join of $D[R_u \cup v] + uv$ and $D[R_w \cup v] + vw$. 
        This can be done in time $\mc{O}(n^5)$. 
    \end{proofclaim}

    \begin{claim}\label{clm:alg_parallelJ}
        We can decide in time $\mathcal O(n^8)$ that either $D$ is not the parallel Haj\'os join of two digraphs, or $D$ is the parallel Haj\'os join of two digraphs $D_1$ and $D_2$ and compute $D_1$ and $D_2$. 
    \end{claim}

    \begin{proofclaim}
        Checking if $D$ is a directed Haj\'os join of two digraphs $D_1$ and $D_2$ can be done by testing for all $6$-tuples of vertices $(t, u, v, w, a, b)$ if $tu, vw \in A(D)$, $D \setminus \{a,b\} - \{tu, vw\}$ has a connected component $A$ containing both $t$ and $u$, a connected component $C$ containing both $u$ and $v$, and some other connected components union of which we name $B$. Then $D$ is the parallel Haj\'os join of  the digraphs obtained from $D[A \cup a]$ and $D[C \cup b]$ by deleting $tu, vw$, identifying $a$ and $b$ into a new vertex $x$, and adding arcs $tx,xw,vx,xu$, and $D[B] + [a,b]$.         
        This can be done in time $\mc{O}(n^8)$.
    
    \end{proofclaim}

    \begin{claim}\label{clm:alg_HSJ}
        We can decide in time $\mc{O}(n^5)$ that either $D$ is not the Haj\'os star join of some digraphs, or $D$ is a Haj\'os star join of some digraphs  $D_1, \dots, D_{\ell}$ and compute $D_1, \dots, D_{\ell}$. 
    \end{claim}
    
    \begin{proofclaim}
        %Let $D$ be a $k$-extremal digraph.  By Lemma~\ref{lem:HTJ_iff}, if $D$ is the Haj\'os star join of some digraphs, then each of these digraphs are $k$-extremal, and in particular they are biconnected. 
        Observe that $D$ is a Haj\'os star join of $\ell$ digraphs if and only if it has $\ell +1$ vertices $x, v_1, \dots, v_{\ell}$ such that $C= v_1 \ra \dots \ra v_{\ell} \ra v_1$ and $D\setminus x -A(C)$ has exactly $\ell$ connected component $R_1, \dots, R_{\ell}$ such that $v_i \in R_i$ for $i=1, \dots, \ell$.  
        Indeed, if it is the case then $D=T(D_1, \dots, D_{\ell}, C)$ where $T$ is the tree with edges $\{xv_1, \dots, xv_{\ell}\}$,  and $D_i = D[R_i \cup x] + [x,v_i]$, and the "only if" part is straightforward by definition of a Haj\'os star join.  
        
        Hence, given $\ell+1$ vertices $y, p_1, \dots, p_{\ell}$, we can decide if they can play the role of respectively $x, v_1, \dots, v_{\ell}$ in time $\mc O(n^2)$. 
        But this is not enough to conclude because  $\ell$ can be large. 

        Anyway, we are going to show that given a triple of vertices $(y,p_{\ell}, p_1)$, we can guess in time $\mathcal O(n^2)$ if there exists $p_2, \dots, p_{\ell-1}$ such that $y, p_1, \dots, p_{\ell}$ can play the role of respectively $x, v_1, \dots, v_{\ell}$. In this case, we say that $(y, p_{\ell}, p_1)$ is a \emph{good guess}. 

        Let $(y, p_{\ell}, p_1)$ be a triple of vertices of $D$ such that $p_{\ell}p_1 \in A(G)$. 
        %This can be done in $\mc O(n^3)$ time. For each of them, we are going to guess if they can play the role of respectively $x$, $v_1$, $v_{\ell}$, in this case, we say it is a \emph{good guess}.        Let $(y,u,v)$ be such a triple. 
        Compute the list of bridges $\mathcal B$ of $D \setminus y - p_{\ell}p_1$. This can be done in $\mc O(n^2)$ time. Observe that if $\mathcal B$ does not contain a $p_1p_{\ell}$-dipath, then our guess is wrong. 
        Assume otherwise, and observe that $\mathcal B$ induces a forest, so it has a unique $p_1p_{\ell}$-dipath, say $P=p_1 \ra p_2 \ra \dots \ra p_{\ell}$. 
        Now, $(y, p_{\ell}, p_1)$ is a good guess if and only  $p_2, p_3, \dots, p_{\ell-1}$ can play the role of respectively $v_2, \dots, v_{\ell-1}$, so we are done. 

        Altogether, it takes $\mc O(n^2)$ to decide if a triple of vertices is a good guess, so the total time is $\mathcal O(n^5)$.
           
        %version Guillaume: 
        
        %Observe that if $D$ is a Haj\'os star join of some digraphs $D_1, \dots, D_{\ell}$ then it is equal to $T((D_1,[x,v_1]), \dots, (D_n,[x,v_{\ell}]);C)$ where $T$ is a star with edges $\{xv_1, \dots, xv_{\ell}\}$, $x$ and $C = v_1 \rightarrow \dots v_{\ell} \rightarrow v_1$. 
        %Observe that, for $i=2, \dots, \ell-1$, $v_iv_{i+1}$ is a bridge of $G \setminus x$. 
        
        %We can try all possibilities of vertices $y$, $u$ and $v$ such that $vu \in A(D)$  (at most $n^3$ possibilities), and we will guess whether $y = x, u = v_1, v = v_\ell$ is possible in time $\mc{O}(n^2)$. Let us compute, in time $\mc{O}(n^2)$, the set $A_{C}$ of all bridges of the underlying multigraph of $D \setminus y - vu$. Since there cannot be a cycle of bridges in a multigraph, $(V(D) \setminus \{y\},A_{C})$ is an oriented forest.
        %Thus there is at most one $uv$-dipath $P = p_1 \rightarrow p_2 \rightarrow p_3 \rightarrow p_{\ell}$ with $p_1 = u$ and $p_{\ell} = v$ in $(V(D) \setminus \{y\},A_{C})$. If we guessed correctly that $y = x, u = v_1, v = v_\ell$, then $v_1 \rightarrow \dots \rightarrow v_\ell$ is such a $uv$-dipath. Thus if there is no $uv$-dipath in $(V(D) \setminus \{y\},A_{C})$, we did a wrong guess. Otherwise, this means that our initial guess implies that $p_i = v_i$ for $1 \leq i \leq \ell$. We can then consider $D' = D - A(P) + [y, p_1] + \dots + [y, p_{\ell}]$. Let $D_1, \dots, D_{\ell}$ be the biconnected components of $D'$. Then $D = T((D_1,[y,p_1]), \dots, (D_n,[y,p_{\ell}]);C)$.
    \end{proofclaim}
    
    Thus in time $\mc{O}(n^8)$, we can check whether $D$ is a directed Haj\'os join or a parallel Haj\'os join of two digraphs $D_1$ and $D_2$ and compute $D_1$ and $D_2$, or a Haj\'os star join of $\ell$ digraphs $D_1, \dots D_{\ell}$ and compute $D_1, \dots, D_{\ell}$.
    If all these checks fail, then by Theorem~\ref{lem:htk_implies_other_joins} $D$ is not $k$-extremal.
    
    If $D$ is a directed Haj\'os join of two digraphs $D_1$ and $D_2$, we can then recursively check whether $D_1$ and $D_2$ are $k$-extremal, and our algorithm can return that $D$ is $k$-extremal if and only if they both are $k$-extremal,  by Lemma~\ref{lem:extremal_directed_Haj}. 
    We do the same if $D$ is a parallel join (Lemma~\ref{lem:extremal_parallel_Haj}) or a star Haj\'os join (Lemma~\ref{lem:HT_iff_ext}). 
    %Similarly, if $D$ is a parallel Haj\'os join of two digraphs $D_1$ and $D_2$, we can then recursively check whether $D_1$ and $D_2$ are $k$-extremal, and our algorithm can return that $D$ is $k$-extremal if that is the case, and that $D$ is not $k$-extremal otherwise, by Lemma~\ref{lem:extremal_parallel_Haj}. Similarly, if $D$ is a Haj\'os star join of $\ell$ digraphs $D_1, \dots, D_{\ell}$, we can then recursively check for each $i \in \{1,\dots,\ell\}$ whether $D_i$ is $k$-extremal, and our algorithm can return that $D$ is $k$-extremal if that is the case for all $i$, and that $D$ is not $k$-extremal otherwise, by Lemma~\ref{lem:HTJ_iff}. 

    \medskip
    
    Let us now prove that our algorithm has time complexity $\mc{O}(n^{10})$. First, note that in each case, in time $\mc{O}(n^{8})$, either we conclude that $D$ is not $k$-extremal, or we make $\ell \geq 2$ recursive calls on digraphs $(D_i)_{i \in [1, \ell]}$. We have that $\sum_{i \in [1, \ell]} |V(D_i)| - 1 \leq |V(D)| - 1$ and for $i \in [1, \ell]$, that $2 \leq |V(D_i)| < |V(D)|$. Let us consider $T$, the rooted tree of recursive calls of our algorithm, with each node $v$ labelled with the digraph $D_v$ of the corresponding recursive call. Let $depth$ be the function which associates to a node its depth in $T$. Then, 
    $$\sum_{v \mid depth(v) = 0} |V(D_{v})| - 1 = V(D) - 1 \leq n$$ and, for $i \geq 1$, $$\sum_{v \mid depth(v) \leq i} |V(D_{v})| - 1 \leq \sum_{v \mid depth(v) = i - 1} |V(D_{v})| - 1.$$ Thus we can recursively prove for any $k \in \mathbb{N}$ that $\sum_{v \mid depth(v) \leq k} |V(D_{v})| - 1 \leq n$. As every $D_v$ has $|V(D_v)| \geq 2$, this implies that there are at most $n$ calls at any depth. But, since $T$ has depth at most $n$, this means there are at most $n^2$ recursive calls. Each of these recursive calls takes time at most $\mc{O}(n^8)$, and thus our algorithm has time complexity $\mc{O}(n^{10})$.
\end{proof}

%------------------------------------------

\section{The hypergraph case} \label{sec:hypergraph}

As mentioned in the introduction, Theorem~\ref{thm:nono_version} has already been generalized to hypergraph with chromatic number at least $4$ by  Schweser,  Stiebitz and  Toft~\cite{SST19}. Their result is closely related to ours as we explain below. %But let us first dive a bit into the world of hypergraphs and explain their result. 

Let $H$ be a hypergraph. 
Its chromatic number $\chi(H)$  is the least integer $k$ such that the vertices of $H$ can be coloured in such a way that no hyperedge is monochromatic. 
A \emph{$uv$-hyperpath} in $H$ is a sequence 
$(u_1, e_1, u_2, e_2, \dots , e_{q-1}, u_q)$ of distinct vertices 
$u_1, u_2, \dots , u_q$ of $H$ and distinct hyperedges $e_1, e_2,\dots , e_{q-1}$ of $H$ such that $u = u_1, v = u_q$ and $\{u_i, u_{i+1}\} \subseteq e_i$ for $i \in \{1, 2,\dots  , q - 1\}$. 
The local connectivity $\lambda(u,v)$ of two vertices $u$ and $v$  is the maximum number of hyperedge-disjoint $uv$-hyperpaths linking $u$ and $v$ and the maximum local connectivity of $H$ is $\lambda(H) = max_{u\neq v}\lambda(u,v)$. 

Let $H_1$ and $H_2$ be two hypergraphs and, for $i=1,2$, let $e_i \in E(H_i)$ and $v_i\in e_i$. The \emph{Haj\'os hyperjoin} of $H_1$ and $H_2$ with respect to $((e_1,v_1), (e_2, v_2))$ is the hypergraph $H$ obtained from $H_1$ and $H_2$ by identifying $v_1$ and $v_2$ into a new vertex $v$, deleting $e_1$ and $e_2$ and adding a new edge $e$ where $e=e_1\cup e_2 \setminus \{v_1,v_2\}$ or $e=e_1\cup e_2 \cup \{v\} \setminus \{v_1,v_2\}$. 

Let $\mathcal H_3$ be the smallest class of hypergraphs that contains all odd wheels and is closed under taking Haj\'os hyperjoins, and for $k \geq 4$, 
$\mathcal H_k$  is the smallest class of hypergraphs that contains $K_{k+1}$ and is closed under taking Haj\'os hyperjoins. We do not define precisely here what a block for an hypergraph is - it is the natural extension of the one for a graph - we refer the reader to \cite{SST19} for a rigorous definition.

\begin{theorem}[\cite{SST19}]\label{thm:main_hypergraph}
    Let  $H$ be a hypergraph with $\chi(H) = k+1 \geq 4$. Then $\chi(H) = \lambda(H) + 1$ if and only if a block of $H$ is in $\mathcal H_k$. 
\end{theorem}

Given a digraph $D$, let $H_D$ be the hypergraph on vertex set $V(D)$, and $e \subseteq  V(D)$ is a hyperedge of $H_D$ if and only if it induces a directed cycle in $D$. We clearly have that $\dic(D) = \chi(H_D)$. 
Hence, one could suspect that our result is actually implied by the result of Schweser,  Stiebitz and  Toft. But this is not the case because a dipath of $D$ does not need to translate into a hyperpath of $H_D$, and thus the maximum local arc-connectivity of $D$ does not need to be equal to the maximum local edge-connectivity of $H_D$.  
Actually, we can prove that the class of extremal digraphs contains the class of extremal hypergraph in the following sense: 

\begin{lemma}\label{lem:digraphs>hypergraphs}
Let $k \geq 3$. For every hypergraph $H \in  \mathcal H_k$, there exists a digraph $D \in \hk$ such that  $H_D= H$. 
    %\begin{itemize}
    %\item[(i)] For every hypergraph $H \in  \mathcal H_k$, there exists a %digraph $D \in \hk$ such that  $H_D= H$. 
    %\item[(ii)] There exist  (an infinite family of) digraphs $D$ such that %$D \in \hk$ and  $H_D \notin  \mathcal H_k$.  
%\end{itemize}
\end{lemma}

The above lemma is a direct consequence of the following property of hypergraphs in $\mathcal H_k$. 

\begin{property}\label{prop:hypergraph_nul}    
Let $H \in \mathcal H_k$. Then for every $e,e' \in E(H)$, $|e\cap e'| \leq 1$.
\end{property}

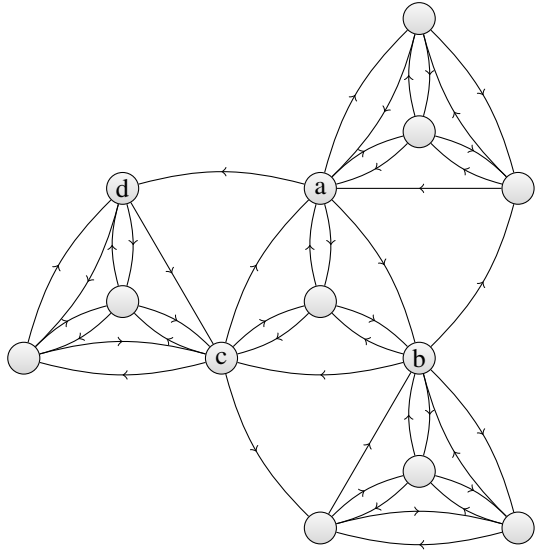
\begin{figure}[!hbtp]
\begin{center}
        \begin{tikzpicture}[scale = 2.6]

            \begin{scope}[]
                \vertex (a) at (0,0) {c};
                
                \begin{scope}[shift=(a)]
                    \vertex (b) at (1,0) {b};
                    
                    \begin{scope}[rotate = 60]
                        \vertex (c) at (1,0) {a};
                    \end{scope}

                    \begin{scope}[rotate = 30]
                        \vertex (d) at (0.57735,0) {};
                    \end{scope}

                    \draw[->-, bend left=15] (d) to (a);
                    \draw[->-, bend left=15] (a) to (d);
                    \draw[->-, bend left=15] (d) to (b);
                    \draw[->-, bend left=15] (b) to (d);
                    \draw[->-, bend left=15] (d) to (c);
                    \draw[->-, bend left=15] (c) to (d);    
                \end{scope}
    
                \draw[->-, bend left=15] (c) to (b);
                \draw[->-, bend left=15] (b) to (a);
                \draw[->-, bend left=15] (a) to (c);

                \begin{scope}[shift=(a), rotate = 120]
                    \vertex (u) at (1,0) {d};
                    \begin{scope}[rotate = 60]
                        \vertex (v) at (1,0) {};
                    \end{scope}
                    \draw[->-, bend left=15] (u) to (v);
                    \draw[->-, bend left=15] (v) to (u);
                    \draw[->-, bend left=15] (a) to (v);
                    \draw[->-, bend left=15] (v) to (a);
                    \draw[->-] (u) to (a);
                    \draw[->-, bend right=15] (c) to (u);  

                    \begin{scope}[rotate = 30]
                        \vertex (w) at (0.57735,0) {};
                    \end{scope}

                    \draw[->-, bend left=15] (w) to (a);
                    \draw[->-, bend left=15] (a) to (w);
                    \draw[->-, bend left=15] (w) to (u);
                    \draw[->-, bend left=15] (u) to (w);
                    \draw[->-, bend left=15] (w) to (v);
                    \draw[->-, bend left=15] (v) to (w);    
                \end{scope}
    
                \begin{scope}[shift=(b), rotate = -120]
                    \vertex (u) at (1,0) {};
                    \begin{scope}[rotate = 60]
                        \vertex (v) at (1,0) {};
                    \end{scope}
                    \draw[->-, bend left=15] (u) to (v);
                    \draw[->-, bend left=15] (v) to (u);
                    \draw[->-, bend left=15] (b) to (v);
                    \draw[->-, bend left=15] (v) to (b);
                    \draw[->-] (u) to (b);
                    \draw[->-, bend right=15] (a) to (u);

                    \begin{scope}[rotate = 30]
                        \vertex (w) at (0.57735,0) {};
                    \end{scope}

                    \draw[->-, bend left=15] (w) to (b);
                    \draw[->-, bend left=15] (b) to (w);
                    \draw[->-, bend left=15] (w) to (u);
                    \draw[->-, bend left=15] (u) to (w);
                    \draw[->-, bend left=15] (w) to (v);
                    \draw[->-, bend left=15] (v) to (w);
                \end{scope}
    
                \begin{scope}[shift=(c), rotate = 0]
                    \vertex (u) at (1,0) {};
                    \begin{scope}[rotate = 60]
                        \vertex (v) at (1,0) {};
                    \end{scope}
                    \draw[->-, bend left=15] (u) to (v);
                    \draw[->-, bend left=15] (v) to (u);
                    \draw[->-, bend left=15] (c) to (v);
                    \draw[->-, bend left=15] (v) to (c);
                    \draw[->-] (u) to (c);
                    \draw[->-, bend right=15] (b) to (u);     

                    \begin{scope}[rotate = 30]
                        \vertex (w) at (0.57735,0) {};
                    \end{scope}

                    \draw[->-, bend left=15] (w) to (c);
                    \draw[->-, bend left=15] (c) to (w);
                    \draw[->-, bend left=15] (w) to (u);
                    \draw[->-, bend left=15] (u) to (w);
                    \draw[->-, bend left=15] (w) to (v);
                    \draw[->-, bend left=15] (v) to (w);
                \end{scope}
            \end{scope}
      
        \end{tikzpicture}
        \end{center}
           \caption{A $3$-extremal digraph. The two induced dicycles $a \rightarrow b \rightarrow c \rightarrow a$ and $a \rightarrow d \rightarrow c \rightarrow a$ share two vertices. Hence its hypergraph of induced dicycles is not $3$-extremal.}\label{fig:graph_vs_hypergraph}
        \end{figure}

\begin{proof}
    The result holds for complete graphs and odd wheels, and it is easy to see that if it holds for two hypergraphs $H_1$ and $H_2$, then it also holds for any Haj\'os hyperjoin of $H_1$ and $H_2$. 
\end{proof}

\begin{proof}[of Lemma~\ref{lem:digraphs>hypergraphs}]
    Using directed Haj\'os join, it is easy to construct a digraph $D \in \hk$ that has two induced directed cycles with two common vertices, see Figure~\ref{fig:graph_vs_hypergraph} for an example. By Property~\ref{prop:hypergraph_nul}, $H_D \notin \mathcal H_k$.
\end{proof}

 We actually believe that the class of extremal digraphs \textit{strictly} contains the class of extremal hypergraph, but we were not able to prove the following that there exist  (an infinite family of) digraphs $D$ such that $D \in \hk$ and  $H_D \notin  \mathcal H_k$. 

%\begin{figure}[h]
%\begin{center}
%\includegraphics[width=10cm]{figures/fig:digraph_vs_hypergraph.jpg}
%\end{center}
%\caption{A digraph $D \in \hk$ such that for every $H \in \mathcal H_k$, $D_H \neq H$}\label{fig:graph_vs_hypergraph}
%\end{figure}

%-------------------------------------

%--------------------------------------------

\section{$2$-extremal digraphs}\label{sec:2extremal}
Similarly to  hypergraphs, the $2$-extremal digraphs seems to be more difficult to characterize.  
$\bid K_3$, and more generally symmetric odd cycles are of course $2$-extremal. 
%A \emph{directed wheel} is a digraph made of a directed cycle plus a vertex linked by a digon to every vertex of the directed cycle. Directed wheels are $2$-extremal. We now give a way to generalize directed wheels to get a simple family of $2$-extremal digraphs that cannot be obtained (at least for some of them) from $\bid K_3$ and directed wheels by applying Haj\'os directed join or Haj\'os tree join. 

%\begin{definition}[Generalized directed wheels]\label{def:HTJ2}
%A digraph $D$ is a \emph{generalized directed wheel} if it is obtained from an embedded symmetric tree $T$ on at least $3$ vertices, in which each pair of leaves is at even  distance, plus a directed cycle $x_1 \ra x_2 \ra \dots \ra x_{\ell} \ra x_1$ where $(x_1, \dots, x_{\ell})$ is a circular ordering of the leaves of $T$ following the natural ordering of an embedding of $T$. 
%\end{definition}

\begin{figure}[!hbtp]
    \begin{center}
        \begin{tikzpicture}[scale=0.4]
            \begin{scope}
                \vertex (c) at (0,0) {$c$};
                \vertex (e) at (5,0) {$e$};
                \draw[->-, bend right = 15] (c) to (e);
                \draw[->-, bend right = 15] (e) to (c);
            \end{scope}

            \begin{scope}[shift=(c), rotate = 120]
                \vertex (a) at (5,0) {$a$};
                \draw[->-, bend right = 15] (c) to (a);
                \draw[->-, bend right = 15] (a) to (c);
            \end{scope}

            \begin{scope}[shift=(c), rotate = -120]
                \vertex (b) at (5,0) {$b$};
                \draw[->-, bend right = 15] (c) to (b);
                \draw[->-, bend right = 15] (b) to (c);
            \end{scope}

            \begin{scope}[shift=(e), rotate = 60]
                \vertex (e) at (0,0) {$e$};
                \vertex (g) at (5,0) {$g$};
                \draw[->-, bend right = 15] (e) to (g);
                \draw[->-, bend right = 15] (g) to (e);
            \end{scope}

            \begin{scope}[shift=(e), rotate = -60]
                \vertex (e) at (0,0) {$e$};
                \vertex (h) at (5,0) {$h$};
                \draw[->-, bend right = 15] (e) to (h);
                \draw[->-, bend right = 15] (h) to (e);
            \end{scope}

            \begin{scope}[shift = (g)]
                \vertex (g) at (0,0) {$g$};
                \vertex (i) at (5,0) {$i$};
                \draw[->-, bend right = 15] (g) to (i);
                \draw[->-, bend right = 15] (i) to (g);
            \end{scope}
            
            \begin{scope}[shift = (g), rotate=120]
                \vertex (g) at (0,0) {$g$};
                \vertex (d) at (5,0) {$d$};
                \draw[->-, bend right = 15] (g) to (d);
                \draw[->-, bend right = 15] (d) to (g);

            \end{scope}

                \begin{scope}[shift = (h), rotate=0]
                \vertex (h) at (0,0) {$h$};
                \vertex (j) at (5,0) {$j$};
                \draw[->-, bend right = 15] (h) to (j);
                \draw[->-, bend right = 15] (j) to (h);
            \end{scope}

            \node () at (e |- 0, -8) {$D$};

            \draw[->-, bend right = 30] (a) to (b);
            \draw[->-, bend right = 30] (b) to (j);
            \draw[->-, bend right = 30] (j) to (i);
            \draw[->-, bend right = 30] (i) to (d);
            \draw[->-, bend right = 30] (d) to (a);

        \end{tikzpicture}
        \end{center}
           \caption{A $2$-extremal digraphs}\label{fig:2_extremal}
        \end{figure}
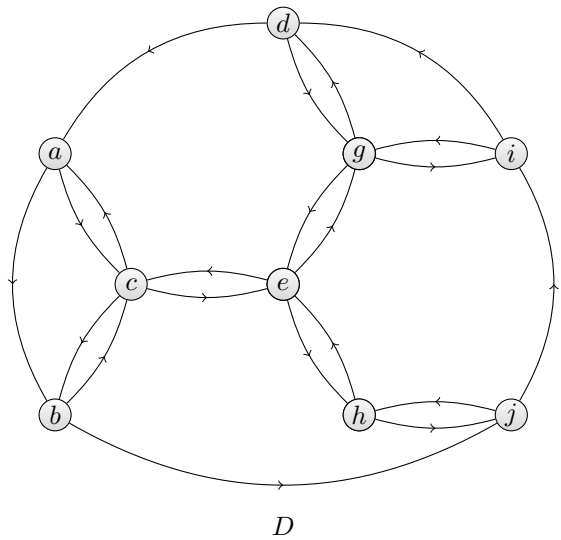

%See Figure~\ref{fig:2_extremal} for an example of a generalized directed wheel. Observe that $\bid K_3$ and directed wheels are generalized wheels.  

\begin{definition}[2-Haj\'os tree join]
Given
\begin{itemize}
    \item a tree $T$ embedded in the plane with at least two edges,
    \item A partition $(A,B)$ of the edges of $T$, with $A=\{u_1v_1, \dots, u_av_a\}$ and $B=\{x_1y_1, \dots, x_by_b\}$ such that every leaf to leaf path in $T$ contains an even number of edges of $B$, 
    \item a circular ordering $C=(x_1, \dots, x_{\ell})$ of the leaves of $T$, taken following the natural ordering given by the embedding of $T$,  and
    \item for $i=1, \dots, a$, a digraph $D_i$ such that
    \begin{itemize}
        \item $V(D_i) \cap V(T) = \{u_i,v_i\}$, 
        \item $[u_i,v_i] \subseteq A(D_i)$, and
        \item for $1 \leq i \neq j \leq a$, $V(D_i) \setminus \{u_i, v_i\} \cap V(D_j) \setminus \{u_j, v_j\} = \emptyset$,
    \end{itemize}
\end{itemize} 
we define the \emph{2-Haj\'os tree join} $T(D_1, \dots, D_a; C)$ to be the digraph obtained from $T$ by replacing each edge  $u_iv_i \in A$ by $D_i -[u_i,v_i]$, each edge $x_iy_i \in B$ by a digon and  by adding the directed cycle  $C = x_1 \ra x_2 \ra \dots \ra x_{\ell} \ra x_1$. 
\end{definition}

Observe that, in the definition of $2$-Haj\'os tree joins, if $A=\emptyset$, then the resulting digraphs is a generalised wheel. 

Let $\mathcal H_2$ be the smallest class of digraphs containing symmetric odd cycle and closed under taking directed Haj\'os join and $2$-Haj\'os tree join. 
It is a routine work to check that digraphs in $\mc H_2$ are $2$-extremal. We conjecture that they are the only ones.  
\begin{conjecture}
    A digraph is $2$-extremal if and only if it is in $\mathcal H_2$. 
\end{conjecture}

%--------------------------------------------- 

\subsubsection*{Acknowledgement}
This research was partially supported by the ANR project DAGDigDec (JCJC)   ANR-21-CE48-0012 and by the group Casino/ENS Chair on Algorithmics and Machine Learning.

\end{document}